\documentclass[11pt,a4paper]{amsart}

\usepackage[margin=3cm,top=2.5cm,bottom=2cm]{geometry}

\usepackage{amsmath, amssymb, amsthm, mathrsfs}
\usepackage{enumitem,microtype} 
\usepackage[dvipsnames]{xcolor}

\usepackage[colorlinks=true,linkcolor=red!70!black,citecolor=MidnightBlue]{hyperref}

\usepackage[utf8]{inputenc} 
\usepackage[T1]{fontenc}
\usepackage{lmodern}

\usepackage{physics}
\usepackage{tikz}
\usepackage{mathdots}
\usepackage{yhmath}
\usepackage{cancel}
\usepackage{siunitx}
\usepackage{array}
\usepackage{multirow}
\usepackage{amssymb}
\usepackage{gensymb}
\usepackage{tabularx}
\usepackage{extarrows}
\usepackage{booktabs}
\usetikzlibrary{fadings}
\usetikzlibrary{patterns}
\usetikzlibrary{shadows.blur}
\usetikzlibrary{shapes}

\newcommand{\R}{\mathbf R}
\newcommand{\T}{\mathbf T}
\renewcommand{\S}{\mathbf S}
\newcommand{\N}{\mathbf N}

\renewcommand{\AA}{\mathcal A}
\newcommand{\BB}{\mathcal B}

\newcommand{\OO}{\mathcal O}

\newcommand{\FF}{\mathcal F}

\newcommand{\NN}{\mathcal N}

\newcommand{\EEE}{\mathscr E}
\newcommand{\LLL}{\mathscr L}

\newcommand{\XXX}{\mathscr X}

\theoremstyle{plain}
\newtheorem{theo}{Theorem}
\newtheorem{prop}[theo]{Proposition}
\newtheorem{lem}[theo]{Lemma}
\newtheorem{cor}[theo]{Corollary}
\theoremstyle{remark}

\theoremstyle{definition}

\numberwithin{equation}{section}
\numberwithin{theo}{section}

\def\le{\leqslant}
\def\ge{\geqslant}
\def\leq{\leqslant}
\def\geq{\geqslant}

\DeclareMathOperator{\Span}{span}
\DeclareMathOperator{\id}{Id}

\def\eps{{\varepsilon}}

\newcommand{\Nt}{|\hskip-0.04cm|\hskip-0.04cm|}

\newcommand{\la}{\langle}
\newcommand{\ra}{\rangle}
\newcommand{\lla}{\left\langle}
\newcommand{\rra}{\right\rangle}

\makeatletter
\newsavebox{\@brx}
\newcommand{\dlla}[1][]{\savebox{\@brx}{\(\m@th{#1\langle}\)}%
	\mathopen{\copy\@brx\mkern2mu\kern-0.9\wd\@brx\usebox{\@brx}}}
\newcommand{\drra}[1][]{\savebox{\@brx}{\(\m@th{#1\rangle}\)}%
	\mathclose{\copy\@brx\mkern2mu\kern-0.9\wd\@brx\usebox{\@brx}}}
\makeatother

\renewcommand{\d}{\,\mathrm{d}}

\renewcommand{\dv}{\,\mathrm{d}v}

\newcommand{\Integ}{\mathbf{I}}	

\newcommand{\disg}{\mathbf{E}^*}
\newcommand{\spg}{\mathbf{E}}

\newcommand{\spp}{\mathbf{X}}
\newcommand{\dispa}{\mathbf{X^*}}

\newcommand{\disp}{\mathbf{Y}}

\newcommand{\sppt}{\XXX}
\newcommand{\spgt}{\EEE}
\newcommand{\sppdt}{\XXX_{\star}}

\allowdisplaybreaks

\newcommand{\step}[2]{\medskip\noindent\textit{Step #1: #2.}}

\title{Non-cutoff Boltzmann equation with soft potentials in the whole space}

\author[K.~Carrapatoso]{Kleber Carrapatoso}
\author[P.~Gervais]{Pierre Gervais}
\date{\today}

\address[K. Carrapatoso]{Centre de Math\'ematiques Laurent Schwartz, \'Ecole
  Polytechnique, Institut Polytechnique de Paris, 91128 Palaiseau cedex, France}
\email{kleber.carrapatoso@polytechnique.edu}

\address[P.~Gervais]{Universit\`{a} degli Studi di Torino, Department of Economics, Social Sciences, Applied Mathematics and Statistics ``ESOMAS'', Corso Unione Sovietica, 218/bis, 10134 Torino, Italy.}
\email{Pierre.Gervais@ens.fr}

\begin{document}

\begin{abstract}
We prove the existence, uniqueness and convergence of global solutions to the Boltzmann equation with non-cutoff soft potentials in the whole space when the initial data is a small perturbation of a Maxwellian with polynomial decay in velocity. Our method is based in the decomposition of the desired solution into two parts: one with polynomial decay in velocity satisfying the Boltzmann equation with only a dissipative part of the linearized operator; the other with Gaussian decay in velocity verifying the Boltzmann equation with a coupling term.
\end{abstract}

\maketitle

\tableofcontents

\section{Introduction}

Consider the Boltzmann equation for the unknown $F=F(t,x,v)$, with~$t \ge 0$, $x\in \R^3$, and~$v \in \R^3$: 
\begin{equation}\label{eq:main}
\partial_t F + v \cdot \nabla_x F = Q(F,F)
\end{equation}
complemented with an initial data $F_0 = F_0(x,v)$.
The collision operator $Q$ is bilinear and acts only on the velocity variable $v \in \R^3$, which represents the fact that collisions are supposed to be localized in space, and it reads
\begin{equation}\label{eq:opBoltzmann}
Q(f,g) (v) = \int_{\R^3} \int_{\S^{2}} B(v-v_*,\sigma) 
\left[ f(v'_*) g(v') - f(v_*) g(v) \right]  \d\sigma \d v_* .
\end{equation}
The pre- and post-collision velocities $(v',v'_*)$ and $(v,v_*)$ are given by
\begin{equation}\label{eq:def-v'-v'*}
v' = \frac{v+v_*}{2} + \frac{|v-v_*|}{2}\sigma 
\quad\text{and}\quad
v'_* = \frac{v+v_*}{2} - \frac{|v-v_*|}{2}\sigma
\end{equation}
which is one possible parametrization of the conservation of momentum and energy in an elastic collision
$$
v' + v'_* = v + v_* \quad\text{and}\quad
|v'|^2 + |v'_*|^2 = |v|^2 + |v_*|^2 .
$$
The collision kernel $B(v-v_*,\sigma)$ encodes the physics of the interaction between particles. It is assumed to be nonnegative and to depend only on the relative velocity $|v-v_*|$ and the angle $\cos \theta = \sigma \cdot \frac{(v-v_*)}{|v-v_*|}$ as
$$
B(v-v_*,\sigma) = |v-v_*|^\gamma b(\cos \theta),
$$
where $-3 < \gamma \le 1$ and the angular part $b$ is a smooth function (except maybe at $\theta = 0$). As it is standard now, we may suppose, without loss of generality, that $\theta \in [0,\pi/2]$ by replacing $B$ by its symmetrized version if necessary.

In this paper we shall consider the case of \textit{non-cutoff soft potentials}, more precisely we assume that $b$ is an implicit function that is locally smooth and has a non-integrable singularity at $\theta=0$ as
\begin{equation}\label{eq:noncutoff1}
\sin\theta\,b(\cos\theta) \, \underset{\theta \sim 0}{\approx} \, C_b \, \theta^{-1-2s}
\quad\text{with}\quad
s \in (0, 1),
\end{equation}
for some constant $C_b>0$, and  
\begin{equation}\label{eq:noncutoff2}
-1 < \gamma + 2s < 0
\quad \text{and} \quad 
-3/2 -s< \gamma  <0.
\end{equation}

\medskip

Let $\mu(v) = (2\pi)^{-3/2} e^{-|v|^2/2}$ be the standard Maxwellian and define the perturbation 
$$
F = \mu + f
$$
which satisfies
\begin{equation}\label{eq:main-pert}
\left\{ 
\begin{aligned}
& \partial_t f + v \cdot \nabla_x f = \LLL f + Q(f,f) \\
& f_{|t=0} = f_0 = F_0 - \mu 
\end{aligned}
\right.
\end{equation}
where $\LLL$ is the linearized collision operator given by
$$
\LLL f = Q(\mu,f) + Q(f, \mu).
$$
We also denote by $\Lambda$ the full linearized operator
$$
\Lambda := \LLL - v \cdot \nabla_x.
$$

It is well known (see for instance \cite{AU1982}) that $\LLL$ is a nonnegative self-adjoint operator on the space~$L^2_v (\mu^{-1/2}) = \left\{ g:\R^3 \to \R \mid \int_{\R^3} |g|^2 \mu^{-1} \, \d v < \infty \right\}$ with kernel given by
$$
\mathrm{ker} (\LLL) = \Span \{ \mu , v_1 \mu , v_2 \mu , v_3 \mu , |v|^2 \mu\}.
$$
We define $\pi$ to be the orthogonal projection onto $\mathrm{ker} (\LLL)$ so that we can decompose
$$
f = \pi f + f^\perp , \quad f^\perp := f - \pi f
$$
with
\begin{equation}
	\label{eq:pif}
	\pi f = \left\{ \rho[f] + u[f] \cdot v + \theta[f] \frac{(|v|^2-3)}{2} \right\} \mu
\end{equation}
where 
$$
\rho[f] = \int_{\R^3} f \d v , \quad
u[f] = \int_{\R^3} v f \d v , \quad
\theta[f] = \int_{\R^3} \frac{(|v|^2-3)}{3} f \d v .
$$

\subsection{Main result}

Before stating our main result we shall introduce the functional spaces we work with. If $X$ is a function space and $w$ a non-negative function, we define the weighted space $X(w)$ as the space associated to the norm
$$
\| f \|_{X(w)} := \| w f \|_X.
$$
In particular, for a weight function $m=m(v)$, we consider the weighted Lebesgue space $L^2_x L^2_v(m)$ as the space associated to the inner product
$$
\la f , g \ra_{L^2_x L^2_v (m)} := \la m f , m g \ra_{L^2_{x,v}}
$$
and the corresponding norm
$$
\| f \|_{L^2_x L^2_v (m)} := \| m f \|_{L^2_{x,v}},
$$
where $\la \cdot , \cdot \ra_{L^2_{x,v}}$ and $\| \cdot \|_{L^2_{x,v}}$ denote the usual inner product and norm of $L^2 (\R^3_x \times \R^3_v)$.


We consider polynomial weight functions $m(v) = \la v \ra^k := (1+|v|^2)^{k/2}$ with $k >0$, and we introduce the anisotropic dissipation space in velocity~$H^{s, *}_v(m)$, inspired from the one presented in~\cite{HTT2020}, as the space associated to the norm
\begin{align}
\| f \|_{H^{s,*}_v(m)}^2 &:= \| \la v \ra^{\gamma / 2} f \|_{L^2_v(m)}^2 + \| f \|_{\dot{H}^{s,*}_v(m)}^2, \label{eq:def-Hs*}\\
\| f \|_{\dot{H}^{s, *}_v(m)}^2 &:= \int_{\R^3 \times \R^3 \times \S^2} b \left(\cos \theta\right) \mu(v_*) \la v_* \ra^{\gamma} \left(  \FF - \FF' \right)^2  \d \sigma  \d v_*  \d v , \label{eq:def-dotHs*}
\end{align}
where we use the shorthand $ \FF = \FF(v) := m(v) \la v \ra^{\gamma/2} f(v) $ and $ \FF' = \FF(v') $ recalling that $v'$ is defined in \eqref{eq:def-v'-v'*}, and which satisfies the following bound (see Lemma \ref{lem:prop_anisotropic_norm}):
\begin{equation*}
	\| \la v \ra^{\gamma/2} f \|_{H^{s}_v(m)} \lesssim \| f \|_{H^{s,*}_v(m)} \lesssim \| \la v \ra^{\gamma/2 + s} f \|_{H^{s}_v(m)}.
\end{equation*}

For functions $f=f(x,v)$ depending on the position $x$ and velocity $v$ variables, we also define the polynomially weighted spaces $\spp(m)$, $\dispa(m)$ and $\disp(m) $ as the spaces associated to the norms
\begin{equation}\label{eq:def-bfX}
\begin{aligned}
\| f \|^2_{\spp(m)} &:= \| f \|^2_{L^2_{x, v}(m)} + \| \la v \ra^{-6 s} \nabla_x^3 f \|_{ L^2_{x, v}(m) }^2, 
\end{aligned}
\end{equation}
\begin{equation}\label{eq:def-bfX*}
\begin{aligned}
\| f \|^2_{\dispa(m)} &:= \| f \|^2_{L^2_{x} H^{s, *}_v(m)}  + \| \la v \ra^{-6 s} \nabla_x^3 f \|_{ L^2_{x}H^{s, *}_v(m) }^2,
\end{aligned}
\end{equation}
and
\begin{equation}\label{eq:def-bfY}
	\| f \|^2_{\disp(m)} := \| f^\perp \|^2_{\dispa(m)} + \| \nabla_x \pi f \|_{ H^2_x L^2_v }^2.
\end{equation}
respectively.

Finally, we denote the Fourier transform $f \mapsto \widehat{f}$ with respect to the space variable $x\in \R^3$ defined as
$$
\widehat{f}(\xi, v) := \frac{1}{(2\pi)^{3/2}}\int_{\R^3} e^{i x \cdot \xi} f(x, v) \d x, \quad \forall \xi \in \R^3.
$$

\medskip

We can now state our main result:
\begin{theo}\label{theo:non-cutoff}
Assume \eqref{eq:noncutoff1}--\eqref{eq:noncutoff2} hold. Consider $k > 13/2 + 7|\gamma| / 2 + 8s$ and define the weight function $m = \la v \ra^k$. There exists $\eps_0 > 0$ small enough such that any initial data $f_0 \in \spp(m)$ satisfying $\| f_0 \|_{\spp(m)} \le \eps_0$ gives rise to a unique global weak solution $f \in L^\infty( \R_+ ; \spp(m ) ) \cap  L^2( \R_+ ; \disp(m) )$ to \eqref{eq:main-pert}, which satisfies the energy estimate
\begin{equation}\label{eq:theo:energy_estimate}
		\sup_{t \ge 0} \| f(t) \|_{\spp(m)}^2 + \int_0^\infty  \| f(t) \|_{\disp(m)}^2  \d t \lesssim \| f_0 \|_{\spp(m)}^2.
	\end{equation}

If moreover the initial data  $\widehat f_0 \in L^p_\xi L^2_v(\la v \ra^{-8s} m)$ satisfies $\| f_0 \|_{\spp(m)} + \| \widehat f_0 \|_{L^p_\xi L^2_v( \la v \ra^{-8s} m)} \le \eps_0$	with $p \in (2 , \infty]$, then for any $0 < \vartheta < \frac{3}{2} \left(\frac{1}{2} - \frac{1}{p} \right)$ and $k_0 < k - |\gamma|$ we have the decay estimate
\begin{equation}\label{eq:theo:decay_estimate}
\| f(t) \|_{\spp(m_0)} \lesssim (1+t)^{-\vartheta} \left( \| f_0 \|_{\spp(m)} +  \| \widehat f_0 \|_{L^p_\xi L^2_v(\la v \ra^{-8s} m)} \right),
\end{equation}
for all $t \ge 0$, where $m_0 = \la v \ra^{k_0}$.
\end{theo}

We now briefly review the known results for the Boltzmann equation near Maxwellian in the torus $\T^3$ and the whole space $\R^3$.

We start by considering the case of cutoff potentials, which corresponds to angular kernels $b$ for which the singularity in~\eqref{eq:noncutoff1} is removed by assuming $b$ integrable. By working near equilibrium, Grad~\cite{G1965} constructed in 1965 the first spatially inhomogeneous solutions for short times. Ukai~\cite{U1974,U1976} gave in the 1970's a new impulse to the Cauchy problem and established, in the case of hard potentials $\gamma \in [0,1]$, the existence of global solutions in $L^\infty_v H^s_x \left( \la v \ra^k \mu^{-1} \d v \right)$, first in the periodic box $\T^3$ in 1974 \cite{U1974}, then in the whole space $\R^3$ in 1976 \cite{U1976}, by relying on spectral studies of the linearized equation \cite{U1974, EP1975, AU1982} (let us mention also \cite{S1983}). The case of soft potentials was then treated in 1980 by Caflisch~\cite{C1980_2}, then in 1982 by Asano and Ukai~\cite{AU1982} only for $\gamma \in (-1, 0)$, but this approach was recently extended to the full range $\gamma \in (-3, 0)$ by Sun and Wu~\cite{SW2021} in 2021 and then Deng~\cite{D2022} in 2022. These results were then proven using energy methods in spaces of the form $L^2_v H^s_x \left( \mu^{-1} \d v \d x \right)$ by Kawashima~\cite{K1990}, Liu, Yang and Yu~\cite{LYY2004}, Guo~\cite{G2003} and Guo and Strain~\cite{GS2006, GS2008}, as well as Duan~\cite{D2008}.

Concerning the non-cutoff case, the first existence result near equilibrium attributed to Ukai~\cite{U1984} ; he constructed local solutions for analytic initial data in $(x, v)$ having Gaussian decay using the Cauchy-Kowalewski theorem. Between 2011 and 2012, Gressman and Strain \cite{GS2011_1, GS2011_2} (in the torus), and Alexandre, Morimoto, Ukai, Xu, Yang \cite{AMUXY2011_1, AMUXY2011_2, AMUXY2012} (in the whole space) constructed the first global solutions in spaces of the form $H^s_{x, v} \left( \la v \ra^k \mu^{-1} \d x \d v \right)$ by working with anisotropic norms. In the whole space framework, Strain~\cite{Strain} obtained the optimal time-decay for solutions in the whole space. Later Sohinger and Strain~\cite{SS2014} extended these results to some Besov spaces in 2014, and Fang and Wang~\cite{FW2022} relaxed some technical regularity and integrability assumptions in 2022. Recently, in the case of the torus, Duan, Liu, Sakamoto and Strain~\cite{DLSS2021} obtained the existence of small-amplitude solutions, that is, in the space $L^1_\xi L^\infty_t L^2_v ( \mu^{-1} \d x \d v )$ where $\xi$ denotes the Fourier variable in space.
Let us also mention two very recent works in the case of the whole space: Deng~\cite{D2022_2}, in the case of hard potentials, who constructec global solutions by working with an anisotropic norm defined from the pseudo-differential study of Alexandre, Hérau and Li~\cite{AHL2019}; and also Duan, Sakamoto and Ueda~\cite{DSU2022} who constructed low regularity solutions (as in~\cite{DLSS2021}) in the case of hard and moderately soft potentials, obtaining also obtained the decay in time of solutions.

All the above results concern solutions with Gaussian decay in velocity, that is, they hold in functional spaces with a weight in velocity of the form $ \mu^{-1}\d v$. 
In 2017, Gualdani, Mischler and Mouhot~\cite{GMM2017}, in the line of \cite{M2006}, constructed solutions with polynomial decay in velocity. More precisely they relaxed the integrability conditions of previous results and constructed solutions in $W^{\ell, p}_v W^{s, q}_x \left( \la v \ra^k \d v \d x \right)$, in the case of hard spheres in the torus $\T^3$. In the same framework, the case of non-cutoff hard potentials was treated in \cite{HTT2020, AMSY2020}, and that of non-cutoff soft potentials in \cite{CHJ2022}.
Very recently, still in the torus and also in spaces with polynomial weights, the case of cutoff soft potentials was studied by Cao~\cite{C2022}.

\medskip

Our result in Theorem~\ref{theo:non-cutoff} gives then, up to our knowledge, the first result of existence of global solutions with polynomial decay in velocity to the non-cutoff Boltzmann equation in the whole space. Moreover we also establish the convergece of solution to the Maxwellian. Inspired by the strategy of \cite{BMM2019}, we shall construct a solution $f$ to \eqref{eq:main-pert} by considering a decomposition of the form~$f = h + g$, where $h(t) \in \spp(m)$ has polynomial decay in velocity and satisfies a ``nice'' semilinear equation in which only a dissipative part of the linearized operator $\Lambda$ is present, and $g(t)$ has Gaussian decay in velocity and evolves according to the Boltzmann equation plus some coupling term coming from~$h$, with convenient decay properties in time and velocity. This system will then be solved using an iterative scheme and an energy method, where it its crucial to have better decay properties for $h$ in order to treat the coupling term in the equation for $g$. We will then combine the energy estimate from the well-posedness theory with dispersive-type estimates for the equation associated to $g$, similarly as in \cite{DSU2022}, in order to obtain the decay estimate.
In Section~\ref{scn:nonlinear_estimates} we prove the necessary nonlinear homogeneous estimates on $Q$, and in Section~\ref{scn:linear_estimates} we prove the necessary coercive-type estimate on the linear part of the equation for $h$ and recall those related to the equation for $g$. In Section~\ref{sec:nonlinear_estimates_inhom} we prove nonlinear inhomogeneous estimates associated to $Q$. We then proceed to prove our main result in Section~\ref{scn:cauchy_th_non_cutoff}.

\subsection{Notations}

The relation denoted~$A \lesssim B$ is to be understood as $A \leq C B$ for some uniform constant $C > 0$, and $A \approx B$ as both $A \lesssim B$ and $B \lesssim A$.


When considering a function $f(v)$ depending on the velocity variable, we shall use the standard shorthand notations
\begin{equation}\label{eq:notation-vv*v'v'*}
f=f(v), \quad 
f' = f(v'), \quad 
f_* = f(v_*), \quad 
f'_* = f(v'_*),
\end{equation}
where we recall that the pre- and post-collision velocities $(v',v'_*)$ and $(v,v_*)$ are defined in~\eqref{eq:def-v'-v'*}.

\section{Estimates on the collision operator}
\label{scn:nonlinear_estimates}

This section is devoted to spatially homogeneous estimates on the collision operator~$Q$. We present some auxiliary results in Section \ref{scn:auxiliary_results} which we will use to prove estimates on $Q$ in polynomially weighted spaces in Section \ref{scn:est_hom_nc_poly}.

\subsection{Auxiliary results}
\label{scn:auxiliary_results}

We state a few results that will be useful in the sequel.
This first lemma will be used to estimates integrals against the kinetic part $|v-v_*|^\gamma$ of the collision kernel $B(v-v_*, \sigma)$.
\begin{lem}
	\label{lem:kinpot_convo}
	Let $\alpha \in (0,3)$ and $s \in (0,1]$. For any smooth enough function $f=f(v)$ one has:
	\begin{enumerate}[leftmargin=*,itemsep=2pt]
		
		\item If $0 < \alpha < \frac{3}{2}$ then for any $\ell > \frac{3}{2}$ there holds, for any $v \in \R^3$,
		\begin{equation}
			\label{eq:kinpot_convo_0_3/2}
			\int_{\R^3} |v-v_*|^{-\alpha} |f(v_*)| \d v_* \lesssim \langle v \rangle^{-\alpha} \| \langle v \rangle^{\ell} f \|_{L^2_v}. 
		\end{equation}
		
		\item If $0 < \alpha < \frac{3}{2}+s$ then for any $\ell > \frac{3}{2}+s$ there holds, for any $v \in \R^3$,
		\begin{equation}
			\label{eq:kinpot_convo_0_3/2_bis}
		\int_{\R^3} |v-v_*|^{-\alpha} |f(v_*)| \d v_* \lesssim \langle v \rangle^{-\alpha} \| \langle v \rangle^{\ell} f \|_{H^s_v}. 
		\end{equation}
		
%

	\end{enumerate}
\end{lem}

\begin{proof}  
	From \cite[Lemma 3.3]{CTW} for instance one has, for any $0 < \beta < 3$ and $\vartheta>3$,
	\begin{equation}\label{eq:CTW}
		\int_{\R^3} |v-v_*|^{-\beta} \langle v_* \rangle^{-\vartheta} \d v_* \lesssim \langle v \rangle^{- \beta}, \quad \forall v \in \R^3.
	\end{equation}
	We now write for $p \in [1,\infty]$, thanks to H\"older's inequality,
	$$
	\int_{\R^3} |v-v_*|^{-\alpha} |f(v_*)| \d v_*
	\le \left( \int_{\R^3} |v-v_*|^{-\alpha \frac{p}{p-1}} \langle v_* \rangle^{-\ell \frac{p}{p-1}} \d v_* \right)^{\frac{p-1}{p}} \| \langle v \rangle^{\ell} f \|_{L^p_v} .
	$$
	We then conclude by using \eqref{eq:CTW} with: $p=2$ if $0 < \alpha < \frac{3}{2}$ ; and $p=\frac{6}{3-2s}$ if $0 < \alpha < \frac{3}{2}+s$ by using the Sobolev embedding $H^s(\R^3) \hookrightarrow L^{\frac{6}{3-2s}}(\R^3)$. 
\end{proof}

\begin{lem}
	\label{lem:kinpot_convo_schwartz}
	Let $\varphi=\varphi(v)$ be a Schwartz function and $\alpha \in (0,3)$. For any $\eta \in (0, 1)$ and $ \ell >0$ there is $C>0$ such that there holds, for any $v \in \R^3$,
	\begin{equation*}
		\int_{\R^3} |v-v_*|^{-\alpha} \varphi(v_*) \d v_* \leq \eta \la v \ra^{-\alpha} + C \la v \ra^{-\ell}.
	\end{equation*}
\end{lem}

\begin{proof}
Let $v \in \R^3$ be fixed. We split the integral for some $M \ge 1$:
\begin{align*}
&\int_{\R^3} |v-v_*|^{-\alpha} |\varphi(v_*)| \d v_* \\
&\qquad \le  \int_{|v-v_*| \ge M} |v-v_*|^{-\alpha} |\varphi(v_*)| \d v_* + \int_{|v-v_*| \le M} |v-v_*|^{-\alpha} |\varphi(v_*)| \d v_*\\
&\qquad =: I_1 + I_2.
\end{align*}
For $I_2$ we write $\la v_* \ra^{-\ell} \leq \la v-v_* \ra^{\ell} \la v \ra^{-\ell}$ for some arbitrary $\ell \geq 0$, thus
$$
I_2 \lesssim \la v \ra^{-\ell} \| \la v \ra^{\ell} \varphi \|_{L^\infty_v} \int_{|v-v_*| \le M} |v-v_*|^{-\alpha} \la v - v_* \ra^{\ell} \d v_* \lesssim \la v \ra^{-\ell}.
$$
For $I_1$ we take $p> 3/(3-\alpha)$, so that $p/(p-1) < 3/\alpha$, and apply H\"older's inequality to get
$$
\begin{aligned}
I_1 
&\le  \left( \int_{\R^3} |v-v_*|^{-\alpha \frac{p}{p-1}} \langle v_* \rangle^{-q \frac{p}{p-1}} \d v_* \right)^{\frac{p-1}{p}}
 \left( \int_{\R^3} \mathbf{1}_{|v-v*| \ge M} \la v_* \ra^{q p} |\varphi(v_*)|^p \d v_* \right)^{\frac{1}{p}} \\
&\lesssim \la v \ra^{-\alpha } \left( \int_{\R^3} \mathbf{1}_{|v-v*| \ge M} \, \la v_* \ra^{q p} |\varphi(v_*)|^p \d v_* \right)^{\frac{1}{p}},
\end{aligned}
$$
where $q>0$ is such that $q \frac{p}{p-1}>3$. We now observe that for $M \ge 2 |v|$, if $|v-v_*| \ge M$ then $|v_*| \ge |v-v_*| - |v| \ge M/2$, and thus we get 
$$
\begin{aligned}
\int_{\R^3} \mathbf{1}_{|v-v_*| \ge M} \, \la v_* \ra^{q p} |\varphi(v_*)|^p \d v_* 
&\lesssim \int_{\R^3} \mathbf{1}_{|v_*| \ge M/2} \, \la v_* \ra^{q p} |\varphi(v_*)|^p \d v_* \\
&\lesssim M^{-p} \int_{\R^3}  \la v_* \ra^{(q+1) p} |\varphi(v_*)|^p \d v_* \\
&\lesssim M^{-p} .
\end{aligned}
$$
We then conclude taking $M \ge \max(1,2 |v|)$ large enough.
\end{proof}

\begin{prop}
	\label{prop:change_of_variables}
	The following change of variables formulas hold:
	\begin{align}
		\label{eq:vstar_to_vprime}
		\int b(\cos \theta) |v-v_*|^\gamma  & f(v,v', \theta) \d \sigma \d v_*\\
		\notag
	 \approx \int & b \left(\cos (\pi - 2 \theta) \right) |v-v_*|^\gamma \sin \left( \frac{\pi}{2} - \theta \right)^{-2-\gamma} f \left(v,v_*, \pi - 2 \theta\right) \d \sigma \d v_*,
	\end{align}
	\begin{gather}
		\label{eq:v_to_vprime}
		\int b(\cos \theta) |v-v_*|^\gamma f (v',v_*, \theta) \d \sigma \d v
		\approx \int b(\cos 2 \theta) |v-v_*|^\gamma f \left(v,v_*, 2 \theta\right) \d \sigma \d v,\\
		\label{eq:pre_post}
		\int B(v-v_*, \sigma) f(v,v_*, v', v_*', \theta) \d \sigma \d v_* \d v = \int B(v-v_*, \sigma) f(v',v_*', v, v_*, \theta) \d \sigma \d v_* \d v.
	\end{gather}
\end{prop}

\begin{proof}
	The pre-post collisional change of variables \eqref{eq:pre_post} is known to be involutive with Jacobian $1$ and it is easy to check that $|v-v_*| = |v'-v_*'|$. We only deal with the first two change of variables. Recall the definition of $v'$:
	$$v' = \frac{v+v_*}{2} + \sigma \frac{|v-v_*|}{2}.$$
	Denote $\displaystyle k := \frac{v-v_*}{|v-v_*|}$ and recall that $\theta$ is the angle $(k, \sigma)$. The differentials of $v'$ with respect to $v$ and $v_*$ writes in the basis $(k, \sigma, w)$ where $w \perp k, \sigma$ (at least when $k$ and $\sigma$ are not colinear)
	\begin{gather}
		\frac{\d v'}{\d v} = \frac{1}{2}\left( \id + \la  \, \cdot \, , k \ra \sigma \right) = \frac{1}{2} \left(
		\begin{matrix}
			1 & 0 & 0\\
			1 & 1 + \cos \theta & 0\\
			0 & 0 & 1
		\end{matrix}
		\right),\\
		\frac{\d v'}{\d v_*} = \frac{1}{2}\left( \id - \la \, \cdot \, , k \ra \sigma \right) = \frac{1}{2} \left(
		\begin{matrix}
			1 & 0 & 0\\
			1 & 1 - \cos \theta & 0\\
			0 & 0 & 1
		\end{matrix}
		\right).
	\end{gather}
	Thus, the following identities hold:
	\begin{gather*}
		\left|\frac{\d v'}{\d v}\right| = \frac{1}{8} (1 + \cos \theta) = \frac{1}{4} \cos^2\left( \frac{\theta}{2} \right),\\
		\left|\frac{\d v'}{\d v_*}\right| = \frac{1}{8} (1 - \cos \theta) = \frac{1}{4} \sin^2\left( \frac{\theta}{2} \right).
	\end{gather*}
	Furthermore, the definition of $v'$ also implies
	\begin{gather*}
		|v'-v_*|^2 = \frac{1}{2}\cos\left( \frac{\theta}{2} \right)^2 |v-v_*|^2,\\
		|v'-v|^2 = \frac{1}{2}\sin\left( \frac{\theta}{2} \right)^2 |v-v_*|^2.
	\end{gather*}
	The angle $\varphi$ formed by $v'-v_*$ and $\sigma$ and the angle $\psi$ formed by $v'-v$ and $\sigma$ are related to $\theta$ by 
	\begin{gather*}
		\varphi = \frac{\theta}{2}, \quad \psi = \frac{\pi - \theta}{2},
	\end{gather*}
	thus the integrals are estimated as follows:
	
	\begin{align*}
		\int b(\cos \theta) |v-v_*|^\gamma  & f(v,v', \theta) \d \sigma \d v_*\\
		\notag
		\approx \int & b \left(\cos (\pi - 2 \psi) \right) |v-v|^\gamma \sin \left( \frac{\pi}{2} - \psi \right)^{-2-\gamma} f \left(v,v_*, \pi - 2 \psi\right) \d \sigma \d v',
	\end{align*}
	\begin{gather*}
		\int b(\cos \theta) |v-v_*|^\gamma f (v',v_*, \theta) \d \sigma \d v
		\approx \int b(\cos 2 \varphi) |v-v_*|^\gamma f \left(v',v_*, 2 \varphi\right) \d \sigma \d v'.
	\end{gather*}
	We conclude to \eqref{eq:v_to_vprime} and \eqref{eq:vstar_to_vprime} by renaming the integration variables.
\end{proof}

\begin{figure}[h]
	\begin{minipage}{.45\textwidth}

		\tikzset{every picture/.style={line width=0.75pt}} 
		
		\begin{tikzpicture}[x=0.75pt,y=0.75pt,yscale=-1,xscale=1]
			
			\draw  [draw opacity=0] (397.26,116.46) .. controls (397.31,116.46) and (397.37,116.46) .. (397.42,116.46) .. controls (401.91,116.51) and (405.82,119.06) .. (407.95,122.84) -- (397.26,129.48) -- cycle ; \draw  [color={rgb, 255:red, 0; green, 0; blue, 0 }  ,draw opacity=1 ] (397.26,116.46) .. controls (397.31,116.46) and (397.37,116.46) .. (397.42,116.46) .. controls (401.91,116.51) and (405.82,119.06) .. (407.95,122.84) ;  
			\draw    (344.85,159.67) -- (449.3,100.24) ;
			\draw  [draw opacity=0] (310.63,119.77) .. controls (312.54,122.82) and (313.61,126.42) .. (313.56,130.25) .. controls (313.52,134.02) and (312.39,137.52) .. (310.48,140.49) -- (293.32,130) -- cycle ; \draw  [color={rgb, 255:red, 200; green, 200; blue, 200 }  ,draw opacity=1 ] (310.63,119.77) .. controls (312.54,122.82) and (313.61,126.42) .. (313.56,130.25) .. controls (313.52,134.02) and (312.39,137.52) .. (310.48,140.49) ;  
			\draw  [color={rgb, 255:red, 200; green, 200; blue, 200 }  ,draw opacity=1 ] (173,130) .. controls (173,63.73) and (226.73,10) .. (293,10) .. controls (359.27,10) and (413,63.73) .. (413,130) .. controls (413,196.27) and (359.27,250) .. (293,250) .. controls (226.73,250) and (173,196.27) .. (173,130) -- cycle ;
			\draw  [color={rgb, 255:red, 200; green, 200; blue, 200 }  ,draw opacity=1 ][dash pattern={on 0.84pt off 2.51pt}] (189.63,70) -- (397,70) -- (397,190) -- (189.63,190) -- cycle ;
			\draw [color={rgb, 255:red, 200; green, 200; blue, 200 }  ,draw opacity=1 ]   (189.63,70) -- (397,190) ;
			\draw [color={rgb, 255:red, 200; green, 200; blue, 200 }  ,draw opacity=1 ]   (397,70) -- (189.63,190) ;
			\draw [color={rgb, 255:red, 255; green, 200; blue, 200 }  ,draw opacity=1 ][line width=1.5]    (293,130) -- (341.17,102.17) ;
			\draw [shift={(344.63,100.17)}, rotate = 149.98] [fill={rgb, 255:red, 255; green, 200; blue, 200 }  ,fill opacity=1 ][line width=0.08]  [draw opacity=0] (6.97,-3.35) -- (0,0) -- (6.97,3.35) -- cycle    ;
			\draw  [color={rgb, 255:red, 200; green, 200; blue, 200 }  ,draw opacity=1 ] (189.01,79.69) -- (200.32,79.69) -- (200.32,69.62) ;
			\draw [color={rgb, 255:red, 255; green, 200; blue, 200 }  ,draw opacity=1 ][line width=0.75]    (293.32,130) -- (340.49,157.25) ;
			\draw [shift={(343.08,158.75)}, rotate = 210.01] [fill={rgb, 255:red, 255; green, 200; blue, 200 }  ,fill opacity=1 ][line width=0.08]  [draw opacity=0] (5.36,-2.57) -- (0,0) -- (5.36,2.57) -- cycle    ;
			\draw   (337,130) .. controls (337,96.86) and (363.86,70) .. (397,70) .. controls (430.14,70) and (457,96.86) .. (457,130) .. controls (457,163.14) and (430.14,190) .. (397,190) .. controls (363.86,190) and (337,163.14) .. (337,130) -- cycle ;
			\draw    (397,70) -- (397,190) ;
			\draw [color={rgb, 255:red, 255; green, 0; blue, 0 }  ,draw opacity=1 ][line width=1.5]    (397,130) -- (428.49,111.81) ;
			\draw [shift={(431.95,109.8)}, rotate = 149.98] [fill={rgb, 255:red, 255; green, 0; blue, 0 }  ,fill opacity=1 ][line width=0.08]  [draw opacity=0] (6.97,-3.35) -- (0,0) -- (6.97,3.35) -- cycle    ;
			\draw  [dash pattern={on 0.84pt off 2.51pt}] (397,70) -- (449.3,100.24) -- (397,190) -- (344.85,159.67) -- cycle ;
			
			\draw (169,58.4) node [anchor=north west][inner sep=0.75pt]  [color={rgb, 255:red, 200; green, 200; blue, 200 }  ,opacity=1 ]  {$v_{*}$};
			\draw (399,193.4) node [anchor=north west][inner sep=0.75pt]  [color={rgb, 255:red, 0; green, 0; blue, 0 }  ,opacity=1 ]  {$v$};
			\draw (402,47.4) node [anchor=north west][inner sep=0.75pt]  [color={rgb, 255:red, 0; green, 0; blue, 0 }  ,opacity=1 ]  {$v'$};
			\draw (170,187.4) node [anchor=north west][inner sep=0.75pt]  [color={rgb, 255:red, 200; green, 200; blue, 200 }  ,opacity=1 ]  {$v'_{*}$};
			\draw (308,97.4) node [anchor=north west][inner sep=0.75pt]  [color={rgb, 255:red, 255; green, 200; blue, 200 }  ,opacity=1 ]  {$\sigma $};
			\draw (318.82,125.48) node [anchor=north west][inner sep=0.75pt]  [font=\scriptsize,color={rgb, 255:red, 200; green, 200; blue, 200 }  ,opacity=1 ]  {$\theta $};
			\draw (312.48,147.89) node [anchor=north west][inner sep=0.75pt]  [color={rgb, 255:red, 255; green, 200; blue, 200 }  ,opacity=1 ]  {$k$};
			\draw (416.48,123.3) node [anchor=north west][inner sep=0.75pt]  [color={rgb, 255:red, 255; green, 0; blue, 0 }  ,opacity=1 ]  {$\sigma $};
			\draw (401.82,96.48) node [anchor=north west][inner sep=0.75pt]  [font=\normalsize,color={rgb, 255:red, 0; green, 0; blue, 0 }  ,opacity=1 ]  {$\psi $};

		\end{tikzpicture}
	\end{minipage}
	\hspace{1cm}
	\begin{minipage}{.45\textwidth}

		\tikzset{every picture/.style={line width=0.75pt}} 
		
		\begin{tikzpicture}[x=0.75pt,y=0.75pt,yscale=-1,xscale=1]
			
			\draw  [draw opacity=0] (329.32,103.58) .. controls (331.21,106.61) and (332.28,110.18) .. (332.25,113.98) -- (312,113.81) -- cycle ; \draw  [color={rgb, 255:red, 0; green, 0; blue, 0 }  ,draw opacity=1 ] (329.32,103.58) .. controls (331.21,106.61) and (332.28,110.18) .. (332.25,113.98) ;  
			\draw  [draw opacity=0] (329.63,163.77) .. controls (331.54,166.82) and (332.61,170.42) .. (332.56,174.25) .. controls (332.52,178.02) and (331.39,181.52) .. (329.48,184.49) -- (312.32,174) -- cycle ; \draw  [color={rgb, 255:red, 200; green, 200; blue, 200 }  ,draw opacity=1 ] (329.63,163.77) .. controls (331.54,166.82) and (332.61,170.42) .. (332.56,174.25) .. controls (332.52,178.02) and (331.39,181.52) .. (329.48,184.49) ;  
			\draw  [color={rgb, 255:red, 200; green, 200; blue, 200 }  ,draw opacity=1 ] (192,174) .. controls (192,107.73) and (245.73,54) .. (312,54) .. controls (378.27,54) and (432,107.73) .. (432,174) .. controls (432,240.27) and (378.27,294) .. (312,294) .. controls (245.73,294) and (192,240.27) .. (192,174) -- cycle ;
			\draw  [color={rgb, 255:red, 200; green, 200; blue, 200 }  ,draw opacity=1 ][dash pattern={on 0.84pt off 2.51pt}] (208.63,114) -- (416,114) -- (416,234) -- (208.63,234) -- cycle ;
			\draw [color={rgb, 255:red, 200; green, 200; blue, 200 }  ,draw opacity=1 ]   (208.63,114) -- (416,234) ;
			\draw [color={rgb, 255:red, 200; green, 200; blue, 200 }  ,draw opacity=1 ]   (416,114) -- (208.63,234) ;
			\draw [color={rgb, 255:red, 255; green, 200; blue, 200 }  ,draw opacity=1 ][line width=1.5]    (312,174) -- (360.17,146.17) ;
			\draw [shift={(363.63,144.17)}, rotate = 149.98] [fill={rgb, 255:red, 255; green, 200; blue, 200 }  ,fill opacity=1 ][line width=0.08]  [draw opacity=0] (6.97,-3.35) -- (0,0) -- (6.97,3.35) -- cycle    ;
			\draw [color={rgb, 255:red, 255; green, 200; blue, 200 }  ,draw opacity=1 ][line width=0.75]    (312.32,174) -- (359.49,201.25) ;
			\draw [shift={(362.08,202.75)}, rotate = 210.01] [fill={rgb, 255:red, 255; green, 200; blue, 200 }  ,fill opacity=1 ][line width=0.08]  [draw opacity=0] (5.36,-2.57) -- (0,0) -- (5.36,2.57) -- cycle    ;
			\draw    (208.01,113.62) -- (416,114) ;
			\draw   (208.63,114) .. controls (208.63,56.56) and (255.19,10) .. (312.63,10) .. controls (370.06,10) and (416.62,56.56) .. (416.62,114) .. controls (416.62,171.44) and (370.06,218) .. (312.63,218) .. controls (255.19,218) and (208.63,171.44) .. (208.63,114) -- cycle ;
			\draw    (222.29,165.65) -- (401.72,61.98) ;
			\draw [color={rgb, 255:red, 255; green, 0; blue, 0 }  ,draw opacity=1 ][line width=1.5]    (312,113.81) -- (360.17,85.98) ;
			\draw [shift={(363.64,83.98)}, rotate = 149.98] [fill={rgb, 255:red, 255; green, 0; blue, 0 }  ,fill opacity=1 ][line width=0.08]  [draw opacity=0] (6.97,-3.35) -- (0,0) -- (6.97,3.35) -- cycle    ;
			\draw  [dash pattern={on 0.84pt off 2.51pt}] (401.72,61.98) -- (416,114) -- (222.29,165.65) -- (208.63,114) -- cycle ;
			
			\draw (188,102.4) node [anchor=north west][inner sep=0.75pt]  [color={rgb, 255:red, 0; green, 0; blue, 0 }  ,opacity=1 ]  {$v_{*}$};
			\draw (418,237.4) node [anchor=north west][inner sep=0.75pt]  [color={rgb, 255:red, 200; green, 200; blue, 200 }  ,opacity=1 ]  {$v$};
			\draw (421,91.4) node [anchor=north west][inner sep=0.75pt]  [color={rgb, 255:red, 0; green, 0; blue, 0 }  ,opacity=1 ]  {$v'$};
			\draw (189,231.4) node [anchor=north west][inner sep=0.75pt]  [color={rgb, 255:red, 200; green, 200; blue, 200 }  ,opacity=1 ]  {$v'_{*}$};
			\draw (327,141.4) node [anchor=north west][inner sep=0.75pt]  [color={rgb, 255:red, 255; green, 200; blue, 200 }  ,opacity=1 ]  {$\sigma $};
			\draw (337.82,169.48) node [anchor=north west][inner sep=0.75pt]  [font=\scriptsize,color={rgb, 255:red, 200; green, 200; blue, 200 }  ,opacity=1 ]  {$\theta $};
			\draw (331.48,191.89) node [anchor=north west][inner sep=0.75pt]  [color={rgb, 255:red, 255; green, 200; blue, 200 }  ,opacity=1 ]  {$k$};
			\draw (322,79.4) node [anchor=north west][inner sep=0.75pt]  [color={rgb, 255:red, 255; green, 0; blue, 0 }  ,opacity=1 ]  {$\sigma $};
			\draw (343.82,98.3) node [anchor=north west][inner sep=0.75pt]  [font=\scriptsize,color={rgb, 255:red, 0; green, 0; blue, 0 }  ,opacity=1 ]  {$\theta /2$};

		\end{tikzpicture}
		
	\end{minipage}
	\caption{The changes of variables $(v, v') \rightarrow (v, v_*)$ and $(v', v_*) \rightarrow (v, v_*)$.}
\end{figure}
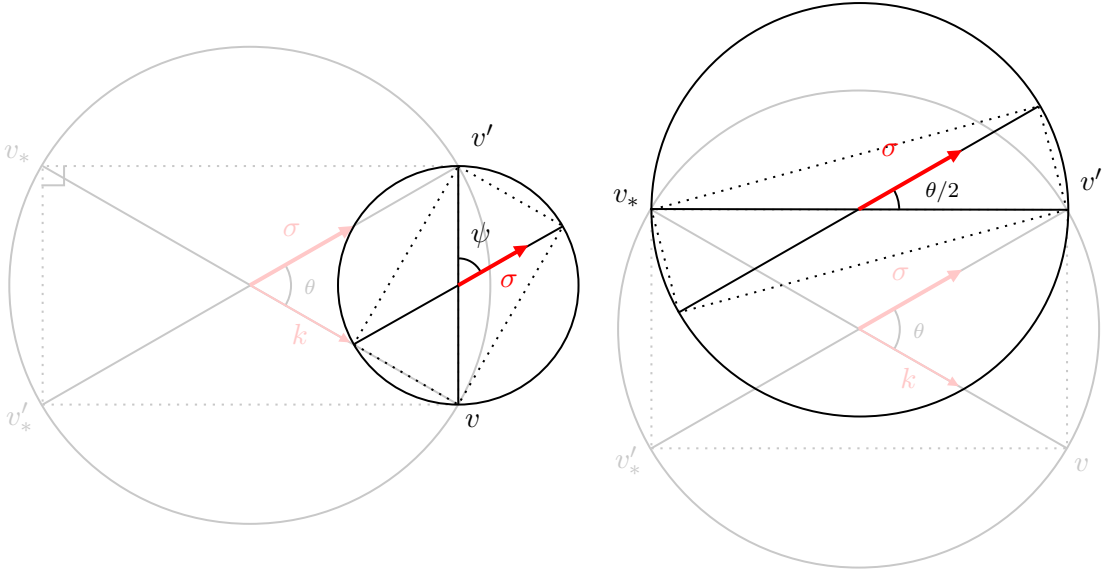

We state the so-called \textit{Cancellation lemma} from \cite[Lemma 1]{ADVW2000}.
\begin{prop}[Cancellation lemma]
	\label{prop:cancellation_lemma}
	The following cancellation formula holds:
	\begin{equation*}
		\int B(v-v_*, \sigma) (f' - f) \d v_* \d \sigma = (f * S)(v),
	\end{equation*}
	where the function $S=S(z)$ is defined as
	\begin{equation*}
		S(z) := 2 \pi \int_0^{\pi/2} \sin \theta \left( \frac{1}{\cos^3 \theta} B\left( \frac{|z|}{\cos (\theta/2)}, \cos \theta \right) - B(|z|, \cos \theta)\right) \d \theta.
	\end{equation*}
	In the particular case of the collision kernel $B_\delta(v-v_*, \sigma) = B(v-v_*, \sigma) \mathbf{1}_{| \theta | \le \delta}$, the corresponding function $S_\delta$ satisfies
	\begin{equation*}
		S_\delta(z) \approx \delta^{2 - 2s} |z|^\gamma.
	\end{equation*}
\end{prop}

The following lemma is from \cite[Lemma 2.3]{AMUXY3}.
\begin{lem}
	Let $m = \la v \ra^k$ with $k \geq 0$, then there holds
	\begin{align}
		\label{eq:weight_split_1}
		m(v) & \lesssim m(v') + m(v'_*),\\
		\label{eq:weight_split_2}
		|m(v) - m(v')| & \lesssim \theta m(v') \la v'_* \ra + \theta^{k} m(v'_*),
	\end{align}
where the pre- and post-collisional velocities $(v',v'_*)$ and $(v,v_*)$ are defined in \eqref{eq:def-v'-v'*}.
\end{lem}

This lemma will serve to remove the kinetic singularity $|v-v_*|^\gamma$ in some integrals involving the collision kernel $B(v-v_*, \sigma)$.
\begin{lem}
	\label{lem:remove_kin_sing}
	For any $a > -3$,  $b \in \R$, $q > 3 + 2 s + a+b$ and any smooth enough function $f$ there holds
	\begin{align*}
		\int b(\cos \theta) |v - v_*|^{a} \la v - v_* \ra^{b} \la v_* \ra^{-q} f(v, v') \d \sigma \d v_* \d v &\\
		\approx \int b(\cos \theta) \la v - v_* \ra^{a + b} \la v_* \ra^{-q} f(v, v') \d \sigma \d v_* \d v
	\end{align*}	
\end{lem}

\begin{proof}
	As we cannot simply control $|v-v_*|^{ a } \la v - v_* \ra^{b}$ by $\la v - v_* \ra^{a+b}$, we resort to using Carleman's representation:
	\begin{align*}
		\int b(\cos \theta) & |v-v_*|^{a} \la v-v_* \ra^{b} \la v_* \ra^{-q} f(v,v') \d \sigma \d v \d v_*\\ &\approx \int_{\substack{v, h \in \R^3,\\ y \perp h, ~ |y| \ge |h|}} \frac{|y|^{1+2 s + a}}{|h|^{1+2s}} \la y \ra^{b} \la v + y \ra^{-q} f(v,v+h) \d y \d h \d v,\\
		& = \int K_{\text{sing}}(v, h) \frac{f(v,v+h)}{|h|^{1+2s}} \d h \d v,
	\end{align*}
	where we denoted the singular $y$-integral, which is well-defined because $1 + 2 s + a+b - q < -2$ and $1 + 2 s + a > -2$
	\begin{equation*}
		K_{\text{sing}} (v, h) := \int_{\substack{y \perp h,\\ |y| \ge |h|} } |y|^{1+2 s + a} \la y \ra^{b} \la v + y \ra^{-q}  \d y
	\end{equation*}
	and aim to prove
	\begin{equation*}
		K_{\text{sing}} (v, h) \approx K_{\text{reg}}(v, h) := \int_{\substack{y \perp h,\\ |y| \ge |h|} } |y|^{1+2 s} \la y \ra^{a+b} \la v + y \ra^{-q} \d y.
	\end{equation*}
	To do so, we split $K_{\text{sing}}$ for $|y| \le \eps$ and $|y| \ge \eps$ where $\eps > 0$ will be chosen later:
	\begin{align*}
		K_{\text{sing}} (v, h) =& \int_{\substack{y \perp h,\\ \eps \ge |y| \ge |h|} } |y|^{1+2 s + a} \la y \ra^{b} \la v + y \ra^{-q} \d y\\
		& + \int_{\substack{y \perp h,\\ |y| \ge \max\{|h|, \eps\}} } |y|^{1+2 s + a} \la y \ra^{b} \la v + y \ra^{-q} \d y.
	\end{align*}
	Concerning the first part, if $\eps$ is small enough, the assumption $|y| \le \eps$ implies $\la v + y \ra^{-q} \approx \la v \ra^{-q}$ uniformly in $y$. Concerning the second part, the assumption $|y| \ge \eps$ implies $\la y \ra \approx |y|$ uniformly in $y$. Thus we have
	\begin{align*}
		K_{\text{sing}} (v, h) =& \int_{\substack{y \perp h,\\ \eps \ge |y| \ge |h|} } |y|^{1+2 s + a} \la y \ra^{b} \la v + y \ra^{-q} \d y + \int_{\substack{y \perp h,\\ |y| \ge \max\{|h|, \eps\}} } |y|^{1+2 s + a} \la y \ra^{b} \la v + y \ra^{-q} \d y\\
		\approx & \la v \ra^{-q} \int_{\substack{y \perp h,\\ \eps \ge |y| \ge |h|} } |y|^{1+2 s + a} \la y \ra^{b} \d y + \int_{\substack{y \perp h,\\ |y| \ge \max\{|h|, \eps\}} } |y|^{1+2 s} \la y \ra^{a+b} \la v + y \ra^{-q} \d y\\
		\approx & \la v \ra^{-q} + \int_{\substack{y \perp h,\\ |y| \ge \max\{|h|, \eps\}} } |y|^{1+2 s} \la y \ra^{a+b} \la v + y \ra^{-q} \d y
	\end{align*}
	With the same reasoning, we have
	\begin{align*}
		K_{\text{reg}} (v, h) =& \int_{\substack{y \perp h,\\ \eps \ge |y| \ge |h|} } |y|^{1+2 s} \la y \ra^{a+b} \la v + y \ra^{-q} \d y + \int_{\substack{y \perp h,\\ |y| \ge \max\{|h|, \eps\}} } |y|^{1+2 s} \la y \ra^{a+b} \la v + y \ra^{-q} \d y\\
		\approx & \la v \ra^{-q} \int_{\substack{y \perp h,\\ \eps \ge |y| \ge |h|} } |y|^{1+2 s} \la y \ra^{a+b} \d y + \int_{\substack{y \perp h,\\ |y| \ge \max\{|h|, \eps\}} } |y|^{1+2 s} \la y \ra^{a+b} \la v + y \ra^{-q} \d y\\
		\approx & \la v \ra^{-q} + \int_{\substack{y \perp h,\\ |y| \ge \max\{|h|, \eps\}} } |y|^{1+2 s} \la y \ra^{a+b} \la v + y \ra^{-q} \d y.
	\end{align*}
	We conclude that $K_{\text{sing}} \approx K_{\text{reg}}$, which concludes this step.
\end{proof}

\begin{lem}
	\label{lem:prop_anisotropic_norm}
	For any $\gamma \in (-\infty, 2]$ and $s \in (0, 1)$, the anisotropic norm defined as
	$$\| f \|_{H^{s, *}_v(m)}^2 = \| \la v \ra^{\gamma / 2} f \|_{L^2_v(m)}^2 + \int b(\cos \theta) (\mu \la v \ra^{-\gamma})_* (\FF' - \FF)^2 \d \sigma \d v_* \d v,$$
	where we denoted $\FF = m \la v \ra^{\gamma / 2} f$, is equivalent to the following norms involving a weight function $\varphi \in L^1_v \left( \la v \ra \right) \cap L \log L$:
	\begin{enumerate}
		\item $\displaystyle \| \la v \ra^{\gamma / 2} f \|_{L^2_v(m)}^2 + \int b(\cos \theta) \varphi_* (\FF' - \FF)^2 \d \sigma \d v_* \d v$, if also $\varphi \in L^1_v\left( \la v \ra^{2 s} \right)$,
		\item $\displaystyle \| \la v \ra^{\gamma / 2} f \|_{L^2_v(m)}^2 + \int b(\cos \theta) \varphi_* \la v \ra^{\gamma} (F' - F)^2 \d \sigma \d v_* \d v$, if also $\varphi \in L^1_v\left( \la v \ra^{4 - \gamma} \right)$,
		\item $\displaystyle \| \la v \ra^{\gamma / 2} f \|_{L^2_v(m)}^2 + \int b(\cos \theta) \la v - v_* \ra^{\gamma} \varphi_* (F' - F)^2 \d \sigma \d v_* \d v$, if also $\varphi \in L^1_v\left( \la v \ra^{4 - \gamma + | \gamma|} \right)$,
	\end{enumerate}
	where we denoted $F = mf$. Furthermore, the anisotropic norm can be compared to isotropic Sobolev norms as follows:
	\begin{equation}
		\label{eq:aniso_iso_comp_poly}
		\| \la v \ra^{\gamma / 2} f \|^2_{H^{s}_v(m) } \lesssim \| f \|_{H^{s, *}_v(m) }^2 \lesssim  \| \la v \ra^{\gamma / 2+2s } f \|^2_{H^{s}_v(m) },
	\end{equation}
	and we have the general bound
	\begin{equation}
		\label{eq:bound_int_b_g_F}
		\int b(\cos \theta) \la v \ra^\gamma \varphi_* (F' - F)^2 \d \sigma \d v_* \d v \lesssim \| \la v \ra^{4 - \gamma} \varphi \|_{L^1_v} \| f \|_{H^{s, *}_v(m)}^2.
	\end{equation}
\end{lem}

\begin{proof}
	We first establish the comparison \eqref{eq:aniso_iso_comp_poly} followed by the equivalence $\| f \|^2_{H^{s, *}_v(m) } \approx (1)$, then proceed to show $(1) \approx (2)$ and $(2) \approx (3)$.
	
	\step{1}{Proof of \eqref{eq:aniso_iso_comp_poly} and $\| f \|^2_{H^{s, *}_v(m) } \approx (1)$}
	From the proof of \cite[Lemma 2.7, estimate of $A$]{AMUXY2011_1}, for some constant $c_\psi > 0$ depending on $\| \psi \|_{L \log L}$ and~$\| \psi \|_{L^1_v(\la v \ra)}$, we have
	\begin{equation}
		\label{eq:dissipation_Sobolev_estimate}
		\| \FF \|_{L^2_v}^2 + \int b(\cos \theta) \psi_* (\FF' - \FF)^2 \d \sigma \d v_* \d v \ge c_\psi \| \FF \|_{H^s_v}^2,
	\end{equation}
	and from the proof of \cite[Lemma 2.8]{AMUXY2011_1}, we also have
	$$\int b(\cos \theta) \left(\mu \la v \ra^{-\gamma}\right)_* (\FF' - \FF)^2 \d \sigma \d v_* \d v \lesssim \| \la v \ra^{2 s} \FF \|_{H^s_v}^2.$$
	The comparison \eqref{eq:aniso_iso_comp_poly} follows from these two bounds with $\psi = \mu \la v \ra^{-\gamma}$. Furthermore, it was established in the proof of \cite[Lemma 2.3-\textit{(ii)}, estimate of $I^\delta_{11}$]{HTT2020} that
	\begin{align}
		\notag
		\int b(\cos \theta) \varphi_* (\FF' - \FF)^2 \d \sigma \d v_* \d v = & \frac{\| \varphi \|_ {L^1_v} }{ \| \mu \la v \ra^{-\gamma} \|_{L^1_v} } \| f \|^2_{\dot{H}^{s, *}_v(m) } \\
		\label{eq:anisotrop_homog_different_weight}
		& + \OO\Big( \| \varphi \|_{L^1_v \left( \la v \ra^{2s} \right) } \| \FF \|_{H^{s}_v }^2 \Big).
	\end{align}
	Thus, interpolating this estimate with \eqref{eq:dissipation_Sobolev_estimate} with $\psi = \varphi$ as to absorb the $\OO$ term, we deduce $\| f \|^2_{H^{s, *}_v(m) } \approx (1)$.

	\step{2}{Proof of $(1) \approx (2)$ and \eqref{eq:bound_int_b_g_F}}
	On the one hand, we have
	\begin{align*}
		\int b(\cos \theta) \varphi_* \la v \ra^{\gamma} & (F' - F)^2 \d \sigma \d v_* \d v \\
		= & \int b(\cos \theta) \varphi_* (\FF' - \FF)^2 \d \sigma \d v_* \d v \\
		& + \int b(\cos \theta) \varphi_* (F')^2 \Big(\la v' \ra^{\gamma / 2} - \la v \ra^{\gamma/2} \Big)^2 \d \sigma \d v_* \d v \\
		& + 2 \int b(\cos \theta) \varphi_* F' ( \FF' - \FF) \Big(\la v' \ra^{\gamma / 2} - \la v \ra^{\gamma/2} \Big) \d \sigma \d v_* \d v,
	\end{align*}
	thus, using Young's inequality, we have
	\begin{align*}
		\frac{1}{2} \int b(\cos \theta) & \varphi_* (\FF' - \FF)^2 \d \sigma \d v_* \d v - \int b(\cos \theta) \varphi_* (F')^2 \Big(\la v' \ra^{\gamma / 2} - \la v \ra^{\gamma/2} \Big)^2 \d \sigma \d v_* \d v  \\
		\le & \int b(\cos \theta) \varphi_* \la v \ra^\gamma (F'-F)^2 \d \sigma \d v_* \d v \\
		\le 2 \int b(\cos \theta) & \varphi_* (\FF' - \FF)^2 \d \sigma \d v_* \d v + 2 \int b(\cos \theta) \varphi_* (F')^2 \Big(\la v' \ra^{\gamma / 2} - \la v \ra^{\gamma/2} \Big)^2 \d \sigma \d v_* \d v.
	\end{align*}
	Next, the inequality \cite[(2.7)]{HTT2020} valid for any $\alpha \le 1$:
	\begin{equation*}
		|\la v \ra^{\alpha} - \la v' \ra^{\alpha}| \lesssim \sin \left(\tfrac{\theta}{2}\right) \la v' \ra^{\alpha} \la v_* \ra^{2-\alpha},
	\end{equation*}
	used with $\alpha = \gamma / 2 \le 1$ allows to tame the angular singularity, making it integrable:
	\begin{align*}
		\int b(\cos \theta) \varphi_* (F')^2 \Big(\la v' \ra^{\gamma / 2} & - \la v \ra^{\gamma/2} \Big)^2 \d \sigma \d v_* \d v \\
		& \lesssim \int b(\cos \theta) \sin\left( \tfrac{\theta}{2} \right)^2 \left( \varphi \la v \ra^{4 - \gamma} \right)_* ( \FF' )^2 \d \sigma \d v_* \d v \\
		& \lesssim \int \left( \varphi \la v \ra^{4 - \gamma} \right)_*  \FF ^2\d v_* \d v
		 = \| \la v \ra^{4 - \gamma} \varphi \|_{L^1_v} \| \FF \|^2_{L^2_v},
	\end{align*}
	where we used \eqref{eq:v_to_vprime} and then integrated in $\sigma$ in the second inequality. In conclusion, we have shown
	\begin{align}
		\notag
		\int b(\cos \theta) \varphi_* \la v \ra^{\gamma} (F' - F)^2 \d \sigma \d v_* \d v
		\approx \int b(\cos \theta) & \varphi_* (\FF' - \FF)^2 \d \sigma \d v_* \d v \\
		\label{eq:v_gamma_inside_outside}
		& + \OO\left( \| \varphi^{4 - \gamma } \|_{L^1_v } \| \FF \|_{H^s_v}^2 \right),
	\end{align}
	from which $(1) \approx (2)$ follows. Combining \eqref{eq:v_gamma_inside_outside} with \eqref{eq:anisotrop_homog_different_weight}, and observing that $2s \le 4 - \gamma$, we obtain
	\begin{equation*}
		\int b(\cos \theta) \varphi_* \la v \ra^{\gamma} (F' - F)^2 \d \sigma \d v_* \d v \approx \| \varphi \|_{L^1_v} \| f \|^2_{\dot{H}^{s, *}_v (m) } + \OO \left( \| \la v \ra^{4 - \gamma } \varphi \|_{L^1_v} \| \FF \|_{H^s_v}^2 \right),
	\end{equation*}
	 from which we deduce \eqref{eq:bound_int_b_g_F} thanks to \eqref{eq:aniso_iso_comp_poly}.
	
	\step{3}{Proof of $(2) \approx (3)$}
	The equivalence is immediate thanks to the previous steps and
	$$\la v_* \ra^{-|\gamma|} \la v \ra^{\gamma} \le \la v - v_* \ra^\gamma \le \la v_* \ra^{|\gamma|} \la v \ra^\gamma,$$
	as it leads to the comparison
	\begin{align*}
		\int & b(\cos \theta) \left( \varphi \la v \ra^{-|\gamma|} \right)_* \la v \ra^{\gamma} (F' - F)^2 \d \sigma \d v_* \d v \\
		& \le \int b(\cos \theta) \varphi_* \la v - v_* \ra^\gamma (F' - F)^2 \d \sigma \d v_* \d v  \\
		& \quad \quad \le \int b(\cos \theta) \left( \varphi \la v \ra^{|\gamma|} \right)_* \la v \ra^{\gamma} (F' - F)^2 \d \sigma \d v_* \d v.
	\end{align*}
	This concludes the proof.
\end{proof}

We have seen in Lemma \ref{lem:prop_anisotropic_norm} that only the strength of the angular singularity and the growth of the weight in $v$ (i.e. $\la v \ra^\gamma$ or $\la v - v_* \ra^\gamma$) up to the change of weight $\FF \leftrightarrow F$ are the defining features of the norm $\| \cdot \|_{H^{s, *}_v }$. One could combine this result with Lemma \ref{lem:remove_kin_sing} to show yet another equivalence:
$$ \| f \|^2_{ H^{s, *}_v(m) } \approx \| \la v \ra^{\gamma / 2} f \|_{L^2_v(m)}^2 + \int B(v-v_*, \sigma) \mu_* (\FF' - \FF)^2 \d v_* \dv \d \sigma,$$
which is the definition chosen in the series \cite{AMUXY2011_1, AMUXY2011_2, AMUXY2011_3, AMUXY2012, AMUXY2013, AMUXY3}.

\subsection{Homogeneous estimates in polynomially weighted spaces}\label{scn:est_hom_nc_poly}

The goal of this section is to establish nonlinear estimates for the collision operator $Q$ in spaces with polynomial weights, which we state below.

\begin{prop}\label{prop:Q-L2v}
Assume $k > 9/2 - |\gamma|/2 + 2s$ and consider $m = \la v \ra^k$.  For any $\ell > 13/2 + 2 |\gamma| $ and smooth enough functions $f,g,h$ there holds
\begin{equation}\label{eq:est_Q_f_g_h}
\begin{aligned}
&\la Q(f, g) ,  h \ra_{L^2_v(m)} \\
&\qquad 
\lesssim   
\left( \| \la v \ra^{\gamma/2} f \|_{L^2_v(m)} \|  g \|_{H^{s,*}_v(\la v \ra^{\ell})}  
+ \| f \|_{H^{s,*}_v(\la v \ra^{\ell})} \|  g \|_{H^{s,*}_v(m)} \right) \| \la v \ra^{\gamma/2} h \|_{L^2_v(m)} \\
&\qquad\quad
+  \| f \|_{L^2_v (\la v \ra^{\ell})} \| \la v \ra^{2s} g \|_{H^{s,*}_v(m)}  \| h \|_{H^{s,*}_v(m)}
\end{aligned}
\end{equation}
Moreover there holds
\begin{equation}\label{eq:est_Q_f_g_g}
\begin{aligned}
&\la Q(g, f) ,  f \ra_{L^2_v(m)} \\
&\qquad 
\lesssim  \| g \|_{L^2_v (\la v \ra^{\ell})} \| f \|_{H^{s,*}_v(m)}^2 
+ \| g \|_{H^{s,*}_v(\la v \ra^{\ell})} \| \la v \ra^{\gamma/2} f \|_{L^2_v(m)}  \| f \|_{H^{s,*}_v(m)} \\
&\qquad\quad 
+  \| \la v \ra^{\gamma/2} g \|_{L^2_v(m)} \| f \|_{H^{s,*}_v(\la v \ra^{\ell})} \| \la v \ra^{\gamma/2} f \|_{L^2_v(m)}  .
\end{aligned}
\end{equation}

\end{prop}

These estimates will be proved by combining commutator estimates (Lemma \ref{lem:commutator}) with He's estimates in $L^2_v$ (Lemma \ref{lem:nonlinear_He}, for \eqref{eq:est_Q_f_g_h}) or new anisotropic estimates in~$L^2(m)$ (for~\eqref{eq:est_Q_f_g_g}). Let us start by recalling the estimate established in He~\cite{H2018}.

\begin{lem}[{\cite[Theorem~1.1]{H2018}}]
	\label{lem:nonlinear_He}
	Assume $-1 < \gamma + 2s < 0$. For any $w_1, w_1 \ge 0$ such that $w_1 +  w_2 = \gamma + 2 s$ and $a, b \in [0, 2 s]$ such that $a + b = 2 s$, any $\ell_0 > 3/2 + |\gamma + 2 s|$, there holds
	\begin{equation*}
		\la Q(f, g), h \ra_{L^2_v} \lesssim \| f \|_{L^2_v ( \la v \ra^{\ell_0} ) } \| g \|_{H^a_v(\la v \ra^{w_1})} \| h \|_{H^b_v(\la v \ra^{w_2})}.
	\end{equation*}
\end{lem}

Let us now state and prove the commutator estimates required to prove Proposition~\ref{prop:Q-L2v}.

\begin{lem}
	\label{lem:commutator}
	Suppose $k > 9/2 - |\gamma| /2 + 2 s$ and consider $m=\la v \ra^k$. For any smooth enough functions $f,g,h$ and any $\ell > 13/2 + 2|\gamma| $ there holds
	\begin{multline*}
		|\la Q(f, g), h \ra_{L^2_v(m) } - \la Q(f, m g), mh \ra_{L^2_v}| \\
		\lesssim \| \la v \ra^{\gamma/2} h \|_{L^2_v(m)} \Big( \| g \|_{H^{s,*}_v(m)} \| f \|_{H^{s,*}_v(\la v \ra^\ell)} + \| g \|_{H^{s,*}_v(\la v \ra^\ell)} \| \la v \ra^{\gamma/2} f \|_{L^2_v(m)} \\
		+ \| g \|_{H^{s,*}_v(m)} \| f \|_{L^2_v(\la v \ra^{\ell})}\Big).
	\end{multline*}
\end{lem}

\begin{proof}
We shall adapt the proof of \cite[Proposition 3.1]{AMSY2020} where the hard potentials case $\gamma \in [0,1]$ was considered. We start back from their decomposition:
	\begin{align*}
		\la Q(f, g), h \ra_{L^2_v(m) } - \la Q(f, m g), mh \ra_{L^2_v} = \int B(v-v_*,\sigma) f'_* g' H (m-m') \d \sigma \d v_* \d v = \sum_{j=1}^{6} \Gamma_j,
	\end{align*}
	where the terms $\Gamma_j$ are defined in the proof of \cite[Proposition 3.1]{AMSY2020} and, under the assumption $k > 9/2 - |\gamma| /2 + 2 s$, were shown to satisfy
\begin{align*}
		\Gamma_1 & = k \int b(\cos \theta) |v - v_*|^{\gamma} \la v \ra^{k-2} |v-v_*|  (v_* \cdot \omega) \cos^{k-1} \left( \tfrac{\theta}{2}\right) \sin \left(\tfrac{\theta}{2} \right) f_* g \left(m h\right) ' \d \sigma \d v_* \d v,
\end{align*}
where 
$$
\omega = \frac{\sigma - (\sigma \cdot \xi) \xi }{|\sigma - (\sigma \cdot \xi) \xi |}  \quad \text{and} \quad  \xi = \frac{v-v_*}{|v-v_*|},
$$
as well as the bounds
\begin{align*}
		\Gamma_2 &\lesssim I\left( g ; (m f)^2 \right)^{1/2} \times I\left( g ; \left(m h\right)^2 \right),\\
		\Gamma_3 &\lesssim I\left( \la v \ra g ; \left(\la v \ra^{k-1} f\right)^2\right)^{1/2} \times I\left( \la v \ra g ; \left(m h\right)^2 \right)^{1/2},\\
		\Gamma_4 &\lesssim  I\left( \la v \ra^{2} f ; (\la v \ra^{k-2} g)^2 \right) \times I\left( \la v \ra^{2} f; \left(m h\right)^2 \right),\\
		\Gamma_5 &\lesssim I\left( \la v \ra^4 f ; \left(\la v \ra^{k-4} g \right)^2 \right)^{1/2} \times I\left( \la v \ra^4 f  ; \left(m h\right)^2 \right)^{1/2},\\
		\Gamma_6 &\lesssim I\left( f ; \left(m g\right)^2 \right)^{1/2} \times I\left( f ; \left(m h\right)^2 \right)^{1/2},
\end{align*}
where we have denoted for compactness
	\begin{gather*}
		 I(\varphi ; \Phi) := \int |v - v_*|^\gamma \varphi_* \Phi \d v \d v_*.
	\end{gather*}
	First, in virtue of \eqref{eq:kinpot_convo_0_3/2_bis}, we have for $\ell_0 >  4 +3/2 + s$ the following estimate:
	\begin{equation*}
		\sum_{j = 2}^{6} \Gamma_j \lesssim \| \la v \ra^{\gamma / 2} h \|_{L^2_v(m)} \left(\| f \|_{H^s_v( \la v \ra^{\ell_0} ) }  \| \la v \ra^{\gamma/2} g \|_{L^2_v(m)} + \| g \|_{H^s_v( \la v \ra^{\ell_0} ) }  \| \la v \ra^{\gamma/2} f \|_{L^2_v(m)} \right).
	\end{equation*}

It remains to estimate the term $\Gamma_1$. Still following the proof of \cite[Proposition 3.1]{AMSY2020}), we denote $\tilde \omega = \frac{v'-v}{|v'-v|}$, so that $\tilde \omega$ is orthogonal to $v'-v_*$, and thus we split $\Gamma_1 = \Gamma_{1,1} + \Gamma_{1,2}$ with
$$
\Gamma_{1,1} = k \int b(\cos \theta) \cos^{k} \left( \tfrac{\theta}{2}\right) \sin \left(\tfrac{\theta}{2} \right) |v - v_*|^{\gamma+1} \la v \ra^{k-2} (v_* \cdot \omega)  f_* g \left(m h\right) ' \d \sigma \d v_* \d v,
$$
and
$$
\Gamma_{1,2} = k \int b(\cos \theta) \cos^{k-1} \left( \tfrac{\theta}{2}\right) \sin^2 \left(\tfrac{\theta}{2} \right) |v - v_*|^{\gamma+1} \la v \ra^{k-2} (v_* \cdot \tilde \omega)  f_* g \left(m h\right) ' \d \sigma \d v_* \d v.
$$
For the term $\Gamma_{1,2}$ we can then argue as for the terms $(\Gamma_j)_{2 \le j \le 6}$ and we obtain
$$
\begin{aligned}
\Gamma_{1,2} 
&\lesssim I\left( \la v \ra^{2} f ; (\la v \ra^{k-1} g)^2 \right) \times I\left( \la v \ra^{2} f; \left(m h\right)^2 \right) \\
&\lesssim \| \la v \ra^{\gamma / 2} h \|_{L^2_v(m)} \| f \|_{H^s_v( \la v \ra^{\ell_0} ) }  \| \la v \ra^{\gamma/2} g \|_{L^2_v(m)} .
\end{aligned}
$$
Moreover the term $\Gamma_{1,1}$ is shown to satisfy, denoting $\mathcal{G} := \la v \ra^{k-2} g$,
$$
\Gamma_{1,1} \lesssim \int b(\cos \theta) \sin \left(\tfrac{\theta}{2}\right) |v-v_*|^{\gamma} \la v_* \ra^2 |f_*| |\mathcal G - \mathcal G' | |m' h'| \d \sigma \d v_* \d v.
$$
Using Hölder's inequality for some $n > 0$ to choose later,
	\begin{align*}
		\Gamma_{1,1} &\lesssim  \int b^{1/2}(\cos \theta) |v - v_*|^{\gamma/2} \sin \left(\tfrac{\theta}{2}\right) (m h)' \left(\la v \ra^2 f\right)^{1/2}_*  \\
		&\quad \times 
		b^{1/4}(\cos \theta) \la v - v_* \ra^{\gamma/4} \la v \ra^{1/2} \left| \mathcal{G} - \mathcal{G}'\right|^{1/2} \left(\la v \ra^{2+ n} f\right)_*^{1/2}  \\
		&\quad \times 
		b^{1/4}(\cos \theta ) |v - v_*|^{\gamma/2} \la v - v_* \ra^{- \gamma / 4} \la v \ra^{1/2} \left|\mathcal{G} - \mathcal{G}'\right|^{1/2} \la v_* \ra^{-n/2} \d \sigma \d v_* \d v\\
		&\lesssim 
		\left(\int b(\cos \theta) |v - v_*|^{\gamma} \sin^2 \left(\tfrac{\theta}{2}\right) \left[(m h)'\right]^2\left(\la v \ra^2 f\right)_* \d \sigma \d v_* \d v\right)^{1/2} \\
		&\quad \times
		\left(\int b(\cos \theta) \la v - v_* \ra^{\gamma} \la v \ra^{2} \left( \mathcal{G} - \mathcal{G}'\right)^{2} \left(\la v \ra^{2+ n} f\right)_*^2  \d \sigma \d v_* \d v\right)^{1/4} \\
		&\quad \times 
		\left(\int b(\cos \theta ) |v-v_*|^{2 \gamma} \la v - v_* \ra^{- \gamma} \la v \ra^{2} \left(\mathcal{G} - \mathcal{G}'\right)^2 \la v_* \ra^{-2 n} \d \sigma \d v_* \d v\right)^{1/4}.
	\end{align*}
	The change of variable \eqref{eq:v_to_vprime} followed by \eqref{eq:kinpot_convo_0_3/2_bis} in the first integral with $\ell_1 > 2 + 3/2+s$ gives
	\begin{align*}
		\Gamma_{1,1} &\lesssim 
		\| f \|_{H^s_v ( \la v \ra^{\ell_1} ) }^{1/2} \|  \la v \ra^{\gamma/2} h \|_{L^2_v(m)} \\
		&\quad \times
		\left(\int b(\cos \theta) \la v - v_* \ra^{\gamma} \la v \ra^{2} \left( \mathcal{G} - \mathcal{G}'\right)^{2} \left(\la v \ra^{2+ n} f\right)_*^2  \d \sigma \d v_* \d v\right)^{1/4} \\
		&\quad \times 
		\left(\int b(\cos \theta ) |v-v_*|^{2 \gamma} \la v - v_* \ra^{- \gamma} \la v \ra^{2} \left(\mathcal{G} - \mathcal{G}'\right)^2 \la v_* \ra^{-2 n} \d \sigma \d v_* \d v\right)^{1/4}. 
	\end{align*}
	Next, using Lemma \ref{lem:remove_kin_sing} in the third integral, taking $2 n > 1 + 2 s + \gamma$ (note that $2 \gamma > -3$), we have
	\begin{align*}
		\Gamma_{1,1} &\lesssim   \| f \|_{H^s_v ( \la v \ra^{\ell_1} ) }^{1/2} \|  \la v \ra^{\gamma/2} h \|_{L^2_v(m)} \\
		&\quad \times 
		\left(\int b(\cos \theta) \la v - v_* \ra^{\gamma} \la v \ra^{2} \left( \mathcal{G} - \mathcal{G}'\right)^{2} \left(\la v \ra^{2+ n} f\right)_*^2  \d \sigma \d v_* \d v\right)^{1/4} \\
		&\quad \times 
		\left(\int b(\cos \theta ) \la v - v_* \ra^{\gamma} \la v \ra^{2} \left(\mathcal{G} - \mathcal{G}'\right)^2 \la v_* \ra^{-2 n} \d \sigma \d v_* \d v\right)^{1/4}.
	\end{align*}
	The inequality $\la v \ra \le \la v -v _* \ra \la v_* \ra$ and the fact that $\gamma \le 0$ then imply
	\begin{align*}
		\Gamma_{1,1} &\lesssim  \| f \|_{H^s_v ( \la v \ra^{\ell_1} ) }^{1/2} \|  \la v \ra^{\gamma/2} h \|_{L^2_v(m)} \\
		&\quad \times 
		\left(\int b(\cos \theta) \la v \ra^{2 + \gamma} \left( \mathcal{G} - \mathcal{G}'\right)^{2} \left(\la v \ra^{2+ n - \gamma /2} f\right)_*^2  \d \sigma \d v_* \d v\right)^{1/4} \\
		&\quad \times 
		\left(\int b(\cos \theta ) \la v  \ra^{2 + \gamma} \left(\mathcal{G} - \mathcal{G}'\right)^2 \la v_* \ra^{-2 n - \gamma} \d \sigma \d v_* \d v\right)^{1/4}.
	\end{align*}
	As $2 + \gamma \le 2$, we may use \eqref{eq:bound_int_b_g_F} to bound these two integrals:
	\begin{align*}
		\Gamma_{1,1} & \lesssim \| f \|_{H^s_v ( \la v \ra^{\ell_1} ) }^{1/2} \|  \la v \ra^{\gamma/2} h \|_{L^2_v(m)} \| g \|_{H^{s, *}_v(m)} \| \la v \ra^{2 - \gamma} \left( \la v \ra^{2 + n - \gamma/2} f \right)^2 \|_{L^1_v}^{1/4} \| \la v \ra^{-2(n - 2 + \gamma)} \|_{L^1_v}^{1/4} \\
		& \lesssim \| f \|_{H^s_v ( \la v \ra^{\ell_1} ) }^{1/2} \|  \la v \ra^{\gamma/2} h \|_{L^2_v(m)} \| g \|_{H^{s, *}_v(m)} \|  \la v \ra^{\ell_2} f \|_{L^2_v}^{1/2},
	\end{align*}
	where we considered $n > 7 / 2 + | \gamma|$ and $\ell_2 = 3 + n + |\gamma| > 13/2 + 2 | \gamma|$.
	Note that since $13/2 + 2 | \gamma| = 11/2 - \gamma / 2 + s + ( - \gamma - \frac{\gamma + 2 s}{2} + 1 )$ and $\gamma + 2 s < 0$, we have $\ell_2 > \ell_0 + | \gamma| / 2$. We conclude the proof by gathering previous estimates, using \eqref{eq:aniso_iso_comp_poly}, and taking~$\ell = \max(\ell_0 + |\gamma|/2 , \ell_1, \ell_2) = \ell_2$.
\end{proof}

We can now prove the main estimates of this subsection, that is to say those of Proposition~\ref{prop:Q-L2v}.

\begin{proof}[Proof of Proposition \ref{prop:Q-L2v}]
	In this proof, we denote $F := mf$, $G := mg$ and $H := mh$.

	\step{1}{Proof of \eqref{eq:est_Q_f_g_h}}
	The first estimate \eqref{eq:est_Q_f_g_h} is a combination of Lemmas \ref{lem:nonlinear_He} and \ref{lem:commutator}. We first observe that Lemma \ref{lem:nonlinear_He} with $w_1 = k + \gamma / 2 + 2 s$, $w_2 = k + \gamma / 2$ and $a=b=s$ yields, for $\ell_0> 3/2 + |\gamma+2s|$,
	\begin{align*}
		\la Q(f, G), H \ra_{L^2_v} & \lesssim  \| f \|_{L^2_v ( \la v \ra^{\ell_0} ) } \| \la v \ra^{\gamma / 2 + 2 s} g \|_{H^{s}_v(m)} \| \la v \ra^{\gamma / 2} h \|_{H^{s}_v(m)}\\
		& \lesssim \| f \|_{L^2 ( \la v \ra^{\ell_0} )} \| \la v \ra^{2 s} g \|_{H^{s, *}_v(m) } \| h \|_{H^{s, *}_v(m)},
	\end{align*}
	where we used \eqref{eq:aniso_iso_comp_poly} in the last line. We then deduce \eqref{eq:est_Q_f_g_h} by putting this estimate together with Lemma~\ref{lem:commutator} for $\ell>13/2 + 2|\gamma| > \ell_0$.

	\step{2}{Reductions for the proof of \eqref{eq:est_Q_f_g_g}}
	First, we decompose the trilinear form using a commutator:
	\begin{align*}
		\lla Q (g, f), f \rra_{L^2_v(m)} = \lla Q(g, F), F \rra_{L^2_v} + \Integ_3,
	\end{align*}
	where we denoted
	\begin{equation*}
		\Integ_3 := \lla Q (g, f), f \rra_{L^2_v(m)} - \lla Q(g, F), F \rra_{L^2_v}.
	\end{equation*}
	Second, we decompose the remaining term as
	\begin{align*}
		\lla Q(g, F), F \rra_{L^2_v(m)}  &= \int B(v-v_*, \sigma) (g'_* F' - g_* F) F \d \sigma \d v \d v_* \\
		&= \frac{1}{2} \int B(v-v_*, \sigma) (2 g'_* F' F - g_* F^2 - g_* (F')^2 ) \d \sigma \d v \d v_*\\
		&\quad + \frac{1}{2} \int B(v-v_*, \sigma) g_* ((F')^2 - F^2) \d \sigma \d v \d v_*.
	\end{align*}
	Using the change of variables \eqref{eq:pre_post} in the first term of the first integral and the cancellation lemma \ref{prop:cancellation_lemma} in the second integral, we obtain for some $c > 0$
	\begin{align*}
		\lla Q(g, F), F \rra_{L^2_v(m)}  &= - \int B(v-v_*, \sigma) g_* (F' - F)^2 \d \sigma \d v \d v_* - c \int |v - v_*|^\gamma g_* F^2 \d v \d v_* \\
		& =: \Integ_{1} + \Integ_{2},
	\end{align*}
	To sum up, we have the decomposition
	\begin{gather*}
		\lla Q(g, f), f \rra_{L^2_v(m)} = \Integ_1 + \Integ_2 + \Integ_3.
	\end{gather*}
	The term $\Integ_2$ satisfies by Lemma \ref{lem:kinpot_convo}, for any $\ell_0 > 3/2 + s$,
	\begin{equation*}
		\Integ_2 \lesssim \| g \|_{H^s_v(\la v \ra^{\ell_0})} \| \la v \ra^{\gamma / 2} f \|_{L^2(m)}^2,
	\end{equation*}
	and the term $\Integ_3$ satisfies by Lemma \ref{lem:commutator}, for any $\ell > 13/2 + 2 |\gamma|$,
	\begin{align*}
		\Integ_3 &\lesssim \| \la v \ra^{\gamma/2} f \|_{L^2_v(m)} \Big( \| f \|_{H^{s,*}_v(m)} \| g \|_{H^{s,*}_v(\la v \ra^\ell)}  + \| f \|_{H^{s,*}_v(\la v \ra^\ell)} \| \la v \ra^{\gamma/2} g \|_{L^2_v(m)} \\
		&\quad
		+ \| f \|_{H^{s,*}_v(m)} \| g \|_{L^2_v(\la v \ra^{\ell})}\Big),
	\end{align*}
	so that, since $\ell > \ell_0 + |\gamma|/2$, the sum $\Integ_2 + \Integ_3$ satisfies the same estimate as $\Integ_3$.

	Let us turn to $\Integ_1$. Using the Cauchy-Schwarz inequality with some positive $q>0$ to be chosen later, we have
	\begin{align*}
		\Integ_{1} =& \int b(\cos \theta)  |v-v_*|^\gamma g_* (F' - F)^2 \d \sigma \d v \d v_* \\
		\le & \left(\int b(\cos \theta) |v-v_*|^{2 \gamma} \la v-v_* \ra^{-\gamma} \la v_* \ra^{-q} (F'-F)^2 \d \sigma \d v \d v_* \right)^{1/2} \\
		& \times \left(\int b(\cos \theta) \la v-v_* \ra^{\gamma} \left(\la v \ra^{q/2} g\right)_*^2 (F'-F)^2 \d \sigma \d v \d v_* \right)^{1/2}.
	\end{align*}
	Assuming $q > 3 + \gamma + 2 s$, we remove the singularity in the integral prefactor using Lemma~\ref{lem:remove_kin_sing}, and then use the inequality~$\la v - v_* \ra^\gamma \le \la v \ra^\gamma \la v_* \ra^{-\gamma}$ in both integrals:
	\begin{align*}
		\Integ_{1} \lesssim & \left(\int b(\cos \theta) \la v-v_* \ra^{\gamma} \la v_* \ra^{-q} (F'-F)^2 \d \sigma \d v \d v_* \right)^{1/2} \\
		& \times \left(\int b(\cos \theta) \la v - v_* \ra^\gamma \left(\la v \ra^{q/2} g\right)_*^2 (F'-F)^2 \d \sigma \d v \d v_* \right)^{1/2} \\
		\lesssim & \left(\int b(\cos \theta) \la v_* \ra^{-{q - \gamma}} \la v \ra^{\gamma} (F'-F)^2 \d \sigma \d v \d v_* \right)^{1/2} \\
		& \times \left(\int b(\cos \theta) \left(\la v \ra^{q/2 - \gamma/2} g\right)_*^2 \la v \ra^{\gamma} (F'-F)^2 \d \sigma \d v \d v_* \right)^{1/2}.
	\end{align*}
	Using \eqref{eq:bound_int_b_g_F} and imposing $q>7-2\gamma$ so that $\la v \ra^{4-q-2\gamma}$ is integrable, we then obtain
	\begin{align*}
		\Integ_{1} & \lesssim \| f \|^2_{H^{s, *}_v(m) } \| \la v \ra^{4 - \gamma} \left(\la v \ra^{q/2} g\right)^2 \|_{L^1_v }^{1/2} \\
		& \lesssim \| f \|^2_{H^{s, *}_v(m) } \| \la v \ra^{\ell_1} g \|_{L^2_v },
	\end{align*}
	where $\ell_1 = 2 - \gamma/2 + q /2$ satisfies $\ell_1 > 11/2 + 2|\gamma|$ and $\ell_1 > 7/2 + 3|\gamma|/2 + s$.
	The estimate~\eqref{eq:est_Q_f_g_g} is then proved by putting together previous estimates and observing that $ \max(\ell , \ell_1 ) = \ell$.
\end{proof}

\section{Linear theory}\label{scn:linear_estimates}

\subsection{Estimates on $\LLL$}

The goal of this subsection is to prove the following proposition.
\begin{prop}\label{prop:dissipative_lot}
Let $k > 7/2 - |\gamma|/2 + 2 s$ and denote the weight function $m = \la v \ra^k$. For any smooth enough function $f$ there holds
\begin{equation}
		\label{eq:dissipative_noncutoff_lot}
		\int (\LLL f) f m^2 \d x \d v \le - c \| f \|_{L^2_x H^{s, *}_v (m)}^2 + C \| f \|^2_{L^2_{x, v}},
	\end{equation}
for some positive constants $c,C>0$.
\end{prop}

We introduce a splitting of the angular cross section $b(\cos \theta)$ so as to decompose the linearized operator $\LLL$ as a singular regularizing part and a weakly coercive non-singular part, namely we define for any $\delta \in (0,1]$
\begin{equation*}
	b(\cos \theta) = b(\cos \theta) \mathbf{1}_{|\theta| \le \delta \pi/2} + b(\cos \theta) \mathbf{1}_{|\theta| > \delta \pi/2} =: b_\delta(\cos \theta) + b_\delta^c(\cos \theta),
\end{equation*}
which induces the following splitting of the linearized operator:
\begin{equation*}
	\LLL = \LLL_\delta + \LLL^c_\delta.
\end{equation*}
Denote $\nu_\delta$ the approximate collision frequency defined as
$$
\nu_\delta (v) = \int_{\R^3 \times \S^2}  |v-v_*|^{\gamma} b^c_\delta(\cos \theta) \mu(v_*)  \d\sigma  \d v_*
$$
which satisfies, according to the cutoff case (see for instance~\cite{Glassey96}), for some positive constants~$\nu_{0},\nu_{1} >0$
\begin{gather}
	\label{eq:bounds_coll_freq}
	\nu_{0} \delta^{-2 s} \langle v \rangle^{\gamma} \le \nu_\delta(v) \le \nu_{1} \delta^{-2 s} \langle v \rangle^\gamma, \quad \forall v \in \R^3.
\end{gather}
The cutoff part of the linearized collision operator then splits
\begin{equation}
	\label{eq:def_L}
	\LLL_\delta^c f = - \nu_\delta f + \int_{\R^3 \times \S^2} |v-v_*|^{\gamma} b_\delta^c(\cos \theta)
	\left[ f(v'_*) \mu(v') - f(v_*) \mu(v) + \mu(v'_*) f(v')  \right]  \d\sigma \d v_*.
\end{equation}

\begin{lem}[Non-grazing collisions]
	\label{lem:wdissipativity_L_cutoff}
	Suppose $k > 3/2 + |\gamma|/2 + s$ and let $m = \la v \ra^k$. For any $\delta, \eps \in (0,1]$ there holds
	$$
		\int \left(\LLL^c_\delta g\right) g m^2 \d v \le - c \delta^{-2 s} \| \langle v \rangle^{\gamma/2} g \|_{L^2_{v} (m)}^2 + \eps \| \la v \ra^{\gamma / 2} g \|^2_{H^s_v(m)} + C_{\delta, \eps} \| g \|_{L^2_{v} }^2,
	$$
for some positive constants $c , C_{\delta, \eps}>0$.

\end{lem}

\begin{proof}
	Firstly we consider
	\begin{align*}
		\int \left(\LLL_\delta^c g\right) g m^2 \d v = &- \|\nu_\delta^{1/2} g\|_{L^2_v(m)}^2 + \int |v-v_*|^{\gamma} b_\delta^c(\cos \theta)
		g'_* \mu' g m^2 \d\sigma \d v_* \d v \\
		&-  \int |v-v_*|^{\gamma} b_\delta^c(\cos \theta) g_* \mu g m^2 \d\sigma \d v_* \d v\\
		& + \int |v-v_*|^{\gamma} b_\delta^c(\cos \theta) \mu'_* g' g m^2 \d\sigma \d v_*  \d v,
	\end{align*}
	so that, using the bounds \eqref{eq:bounds_coll_freq} on $\nu_\delta$, we have
	\begin{equation*}
		\int \left(\LLL^c_\delta g\right) g m^2 \d v + \nu_{0} \delta^{-2 s} \| \la v \ra^{\gamma/2} g \|_{L^2_{v}(m)}^2 \le \Integ_1 + \Integ_2 + \Integ_3,
	\end{equation*}
with, denoting $G = m g$, 
	\begin{align*}
		\Integ_1 &:= \int |v-v_*|^{\gamma} b_\delta^c(\cos \theta)
		g'_* \mu' G m \d\sigma  \d v_*  \d v ,\\
		\Integ_2 &:= \int |v-v_*|^{\gamma} b_\delta^c(\cos \theta) g_* \mu G m \d\sigma  \d v_* \d v,\\
		\Integ_3 &:= \int |v-v_*|^{\gamma} b_\delta^c(\cos \theta) \mu'_* g' G m \d\sigma  \d v_* \d v.
	\end{align*}
	In Step 1, we prove that the terms $\Integ_1$ and $\Integ_3$ satisfy the bound
	\begin{gather}
		\label{eq:goal_bound}
		\Integ_1 + \Integ_3 \lesssim \delta^{-1 - \gamma/2 - 2 s} \int |v-v_*|^\gamma \,  G^2 \, \varphi_*  \d v   \d v_*,
	\end{gather}
	where $\varphi$ denotes a Schwartz function (typically of the form $\mu^{a} \la v \ra^b$) from which we will deduce using Lemma \ref{lem:kinpot_convo_schwartz} with $\eta = \eps \delta^{1 + \gamma/2}$ that there holds
	\begin{equation}
		\label{eq:goal_bound_2}
		\Integ_1 + \Integ_3 \le \eps \delta^{-2 s} \| \la v \ra^{\gamma/2} g \|_{L^2_v ( m )}^2 + C_{\eps, \delta} \|g \|_{L^2_v}^2.
	\end{equation}
	In Step 2, we will prove that $\Integ_2$ satisfies 
	\begin{equation*}
		\Integ_2 \le \eps \delta^{-2 s} \| \la v \ra^{\gamma / 2} g \|_{H^s_v(m)} + C_{\eps, \delta} \| g \|_{L^2_v}^2,
	\end{equation*}
	so that taking $\eps$ small enough, we obtain
	$$\Integ_1 + \Integ_2 + \Integ_3 \le \frac{\nu_0}{2} \delta^{-2 s} \| \la v \ra^{\gamma / 2} g \|_{L^2_v(m)}^2 + \eps' \| \la v \ra^{\gamma / 2} g \|_{H^s_v(m)}^2 + C_{\delta, \eps} \| g \|_{L^2_v}^2,$$
	where $\eps'$ is arbitrarily small. This will indeed prove the lemma by taking~$\eps'$ small enough.

	\step{2}{Proof of \eqref{eq:goal_bound} for $\Integ_1$ and $\Integ_3$}
	We start by splitting $\Integ_1$ using \eqref{eq:weight_split_1}:
	\begin{align*}
		\Integ_1 = &\int |v-v_*|^{\gamma} b_\delta^c(\cos \theta) g'_* \mu' m G \d \sigma \d v_* \d v\\
		\lesssim & \int |v-v_*|^{\gamma} b_\delta^c(\cos \theta) \left(\la v \ra g\right)'_* (\mu m)' G \d \sigma \d v_* \d v\\
		& +  \int |v-v_*|^{\gamma} b_\delta^c(\cos \theta) G'_* \mu' G \d \sigma \d v_* \d v =: \Integ_{11}+\Integ_{12},
	\end{align*}
	Rewriting $\Integ_{11}$ thanks to \eqref{eq:pre_post}, using the Cauchy-Schwarz inequality, then integrating in $\sigma$ one obtains
	\begin{align*}
		\Integ_{11} =& \int |v-v_*|^{\gamma} b_\delta^c(\cos \theta) \left(\la v \ra g\right)_* (\mu m) G' \,  \d \sigma \d v_* \d v \\
		\leq &\left(\int |v-v_*|^{\gamma} b_\delta^c(\cos \theta) \left(\la v \ra g\right)_*^2 (\mu m) \,  \d \sigma \d v_* \d v\right)^{1/2} \\
		& \times \left(\int |v-v_*|^{\gamma} b_\delta^c(\cos \theta) \left(\mu m\right) (G')^2 \,  \d \sigma \d v_* \d v\right)^{1/2}.
	\end{align*}
	Using the change of variables \eqref{eq:vstar_to_vprime} in the post-factor, then integrating in $\sigma$ we get
	\begin{align*}
		\Integ_{11} \lesssim & \, \left( \int b^c_\delta(\cos \theta) |v-v_*|^{\gamma} \left(\la v \ra g\right)_*^2 (\mu m) \d v_* \d v\right)^{1/2} \\
		& \times \left(\int b^c_\delta(\cos (\pi - 2 \theta)) |v-v_*|^{\gamma} \left(\pi - 2 \theta\right)^{- 2 - \gamma} (\mu m) G_*^2 \d v_* \d v\right)^{1/2}\\
		\lesssim & \, \delta^{-1 -\gamma/2 -2 s } \int |v-v_*|^{\gamma} (\mu m) G_*^2 \d v_* \d v.
	\end{align*}
	To bound the part $\Integ_{12}$, we start again with the Cauchy-Schwarz inequality:
	\begin{align*}
		\Integ_{12} = &  \int |v-v_*|^{\gamma} b_\delta^c(\cos \theta) G'_* \mu' G \d \sigma \d v_* \d v \\
		\le & \left(\int |v-v_*|^{\gamma} b_\delta^c(\cos \theta) (G'_*)^2 \mu' \d \sigma \d v_* \d v \right)^{1/2} \\
		& \times \left(\int |v-v_*|^{\gamma}  b_\delta^c(\cos \theta) \mu' G^2 \d \sigma \d v_* \d v \right)^{1/2}.
	\end{align*}
	Up to the pre-post change of variables \eqref{eq:pre_post} in the prefactor, this term is dealt with in the same way as $\Integ_{11}$. Similar computations (using \eqref{eq:v_to_vprime} this time) lead to
	$$\Integ_3 \lesssim \delta^{-2 s} \int |v-v_*|^\gamma (\mu m) G_*^2 \d v_* \d v.$$
	This concludes this step.
	
	\step{2}{Proof of \eqref{eq:goal_bound_2} for $\Integ_2$}
	For $\Integ_2$, we integrate in $\sigma$ to get the factor $\delta^{-2s}$, then in $v_*$ using the estimate \eqref{eq:kinpot_convo_0_3/2} with the power $\theta = k + \gamma/2 > 3/2 + s$, which yields
	\begin{align*}
		\Integ_2 &=  \int b^c_\delta (\cos \theta) |v-v_*|^{\gamma} g_* \mu G \d v_* \d v \\
		&\lesssim  \delta^{-2 s} \|\la v \ra^{\theta} g \|_{H^s_v} \int\la v \ra^\gamma \mu G \d v\\
		&\lesssim  \eps \delta^{-2 s} \|\la v \ra^{\gamma/2} g \|_{H^s_v(m)} + C_{\eps, \delta} \left\| g \right\|_{L^2_v},
	\end{align*}
	where the last line comes from Young's inequality.
\end{proof}

\begin{lem}[Grazing collisions]
	\label{lem:wdissipativity_L_grazing}
	Let $k > 13/2 + 2|\gamma|$ and define $m = \la v \ra^k$. There exists some $c > 0$ such that for any $\eps > 0$ and $\delta > 0$
	\begin{gather}
		\label{eq:grazing_delta}
		\la \LLL_\delta f , f \ra_{L^2_v(m)} \le \eps \| f \|_{H^{s, *}}^{2} + C_\eps \| \la v \ra^{\gamma / 2} f \|^2_{L^2_v (m) },\\
		\label{eq:grazing}
		\la \LLL f , f \ra_{L^2_v(m)} \le - c \| f \|_{ H^{s, *}_v (m) }^2 + C \| \la v \ra^{\gamma / 2} f \|^2_{L^2_v (m) },
	\end{gather}
	for some $C_\eps > 0$.
\end{lem}

\begin{proof}
	We start by splitting the Dirichlet form using commutators:
	\begin{align*}
		\la \LLL_\delta f, f \ra_{L^2_v(m)} & = \la Q_\delta (\mu, f), f \ra_{L^2_v(m)} + \la Q_\delta (f, \mu), f \ra_{L^2_v(m)} \\
		&= \la Q_\delta (\mu, F), F \ra_{L^2_v} + R_1 + R_2 + R_3,
	\end{align*}
	where we denoted
	\begin{gather*}
		R_1 := \la Q_\delta (f, m \mu), F \ra_{L^2_v},\\
		R_2 := \la Q_\delta (\mu, f), f \ra_{L^2_v(m)} - \la Q_\delta(\mu, F), F \ra_{L^2_v},\\
		R_3 := \la Q_\delta (f, \mu), f \ra_{L^2_v(m)} - \la Q_\delta(f, m \mu), F \ra_{L^2_v}.
	\end{gather*}
	The first term is estimated using Lemma \ref{lem:nonlinear_He} (where we choose $w_1 = \gamma / 2 + 2 s$, $w_2 = \gamma /2$ for the weights, $a=2s$ and $b=0$ for the derivatives):
	\begin{equation*}
		R_1 \lesssim \| \la v \ra^{\gamma / 2} f \|_{L^2_v(m)}^2,
	\end{equation*}
	and the two other ones using Lemma \ref{lem:commutator} and Young's inequality:
	\begin{align*}
		R_2 + R_3 &\le C \left(\| \la v \ra^{\gamma / 2} f \|_{L^2_v(m)}^2 + \| \la v \ra^{\gamma / 2} f \|_{L^2_v(m)} \| f \|_{H^{s, *}_v(m) }\right) \\
		& \le C_\eps \| \la v \ra^{\gamma / 2} f \|_{L^2_v(m)}^2 + \eps \| f \|_{H^{s, *}_v(m) }^2.
	\end{align*}
	We then focus on the first term which provides the anisotropic dissipation $\dot{H}^{s, *}_v (m)$:
	\begin{align*}
		\la Q_\delta(\mu, F), F \ra_{L^2_v(m) } =& \int B_\delta(v-v_*, \sigma) (\mu'_* F' - \mu_* F) F \d \sigma \d v_* \d v\\
		=& \frac{1}{2} \int B_\delta(v-v_*, \sigma) (2 \mu'_* F' F - \mu_* F^2 - \mu_*' F^2)\d \sigma \d v_* \d v \\
		& + \frac{1}{2} \int B_\delta(v-v_*, \sigma) (\mu'_* - \mu_*) F^2 \d \sigma \d v_* \d v.
	\end{align*}
	We use \eqref{eq:pre_post} to change the term $\mu'_* F^2$ of the first integral into $\mu_* (F')^2$, and the cancellation lemma (Proposition \ref{prop:cancellation_lemma}) in the second integral:
	\begin{align*}
		\la Q_\delta(\mu, F), F \ra_{L^2_v(m) } = -& \frac{1}{2} \int B_\delta(v-v_*, \sigma)  \mu'_* ( F' - F)^2 \d \sigma \d v_* \d v \\
		&+ C_\delta \int |v -v_*|^\gamma \mu_* F^2 \d v_* \d v,
	\end{align*}
	where $C_\delta \lesssim 1$. We thus have in virtue of \eqref{eq:kinpot_convo_0_3/2}
	\begin{equation*}
		\la Q_\delta(\mu, F), F \ra_{L^2_v(m) } + \frac{1}{2} \int B_\delta(v-v_*, \sigma)  \mu'_* ( F' - F)^2 \d \sigma \d v_* \d v \lesssim \| \la v \ra^{\gamma/2} f\|^2_{L^2_v(m)}.
	\end{equation*}
	Next, using \eqref{eq:pre_post}, and then $|v - v_*| \le \la v \ra \la v_* \ra$ combined with the fact that $\gamma \le 0$, we have
	\begin{align*}
		-\int B_\delta(v-v_*, \sigma)  \mu'_* & ( F' - F)^2 \d \sigma \d v_* \d v \\
		=& - \int b_\delta(\cos \theta) |v-v_*|^\gamma \mu_* ( F' - F)^2 \d \sigma \d v_* \d v\\
		\le & -\int b_\delta(\cos \theta) \la v \ra^\gamma \left(\mu \la v \ra^{-\gamma}\right)_* ( F' - F)^2 \d \sigma \d v_* \d v.
	\end{align*}
	One shows as in the proof of $(1) \approx (2)$ from Lemma \ref{lem:prop_anisotropic_norm} that
	\begin{align*}
		-\int B_\delta(v-v_*, \sigma)  \mu'_* ( F' - F)^2 \d \sigma \d v_* \d v + \frac{1}{2} \int b_\delta(\cos \theta) & \left(\mu \la v \ra^{\gamma}\right)_* ( \FF' - \FF)^2 \d \sigma \d v_* \d v\\
		 \lesssim & \| \la v \ra^{\gamma/2} f\|^2_{L^2_v(m)}.
	\end{align*}
	Together with the previous estimates, we conclude that
	\begin{align*}
		\la \LLL_\delta f, f \ra_{L^2_v(m)} \le & - \frac{1}{2} \int b_\delta(\cos \theta)\left(\mu \la v \ra^{\gamma}\right)_* ( \FF' - \FF)^2 \d \sigma \d v_* \d v \\
		& + \eps \| f \|^2_{H^{s, *}_v(m) } + C_\eps \| \la v \ra^{\gamma / 2} f \|^2_{L^2_v(m)}.
	\end{align*}
	The second term being non-positive, we conclude that \eqref{eq:grazing_delta} holds.

Furthermore, this proof works when replacing $b_\delta$ by $b$ (which corresponds, in a way, to taking $\delta$ large), thus for $\LLL$ we get  
\begin{align*}
		\la \LLL f, f \ra_{L^2_v(m)} \le & - \frac{1}{2} \int b(\cos \theta)\left(\mu \la v \ra^{\gamma}\right)_* ( \FF' - \FF)^2 \d \sigma \d v_* \d v \\
		& + \eps \| f \|^2_{H^{s, *}_v(m) } + C_\eps \| \la v \ra^{\gamma / 2} f \|^2_{L^2_v(m)}.
\end{align*}
Recalling the definition of the norm $\dot H^{s, *}_v(m)$ in \eqref{eq:def-dotHs*}, we therefore deduce that \eqref{eq:grazing} also holds by taking $\eps$ small enough. 
\end{proof}

We are now able to complete the proof of Proposition \ref{prop:dissipative_lot}.

\begin{proof}[Proof of Proposition \ref{prop:dissipative_lot}]
	We get from Lemmas \ref{lem:wdissipativity_L_cutoff} and \eqref{eq:grazing_delta} that for $\delta$ small enough
	\begin{align*}
		\la \LLL f, f \ra_{L^2_v(m)} &= \la \LLL_\delta^c f, f \ra_{L^2_v(m)} + \la \LLL_\delta f, f \ra_{L^2_v(m)} \\
		&\le   \eps \| f \|_{\dot{H}^{s, *}_v (m) }^2 + C_\eps \| \la v \ra^{\gamma / 2} f \|_{L^2_v(m)} 
		 - c \delta^{- 2 s} \| \la v \ra^{\gamma / 2} f \|_{L^2_v(m)}^2 + C \| f \|_{L^2_v}^2   \\
		&\le  \, \eps \| f \|_{\dot{H}^{s, *}_v (m)}^2 - \| \la v \ra^{\gamma / 2} f \|_{L^2_v(m) }^2 + C \| f \|_{L^2_v(m)}^2.
	\end{align*}
	We interpolate this estimate with \eqref{eq:grazing}: for any $\theta \in [0, 1]$
	\begin{align*}
		\la \LLL f, f \ra_{L^2_v(m)} & \le  \left[\theta \eps - (1 - \theta) c\right] \| f \|_{\dot{H}^{s, *}_v (m)}^2 \\
		&\quad + \left[- \theta +(1 - \theta) C\right] \| \la v \ra^{\gamma / 2} f \|_{L^2_v(m)} + \theta C \| f \|_{L^2_v}^2.
	\end{align*}
	We deduce \eqref{eq:dissipative_noncutoff_lot} by taking $\theta$ close enough to 1, $\eps$ small enough, and integrating in space.
\end{proof}

\subsection{Estimates on $\LLL - v \cdot \nabla_x$ in exponentially weighted spaces}

We present here dissipativity estimates for the full linearized operator $\Lambda = \LLL - v \cdot \nabla_x $ in the space $L^2_v\left( \mu^{-1/2} \right)$ and its weighted counterpart~$L^2_v\left( \la v \ra^q \mu^{-1/2} \right)$.  These estimates were initially proved in the works \cite{Strain, AMUXY2012}, but we formulate them in a manner compatible with our approach. 

As in \cite{Strain} let us introduce the following bilinear symmetric form using the notations introduced in \eqref{eq:pif}:
\begin{equation*}
	\begin{aligned}
		\Psi[ f ,  g ] (\xi)
		&:=  \frac{\eta_1 \mathrm{i}  }{1+|\xi|^2}  \xi \theta[\widehat f(\xi)]  \cdot \Upsilon[ \widehat g^\perp(\xi)]  
		+  \frac{\eta_1 \mathrm{i} }{1+|\xi|^2}   \xi \theta[\widehat g(\xi)] \cdot \Upsilon[ \widehat f^\perp (\xi)] \\
		&\quad 
		+  \frac{\eta_2\mathrm{i}}{1+|\xi|^2} (\xi \otimes u[\widehat f(\xi)] )^{\mathrm{sym}} : \left\{ \Theta[\widehat g^\perp(\xi)] + \theta[\widehat g(\xi)] \id  \right\}\\
		&\quad 
		+  \frac{\eta_2\mathrm{i}}{1+|\xi|^2} (\xi \otimes u[\widehat g(\xi)] )^{\mathrm{sym}} : \left\{\Theta[\widehat f^\perp(\xi)] + \theta[\widehat f(\xi)] \id \right\}\\
		&\quad 
		+  \frac{\eta_3\mathrm{i} }{1+|\xi|^2} \xi \rho[\widehat f(\xi)] \cdot  u[\widehat g(\xi)] 
		+  \frac{\eta_3\mathrm{i} }{1+|\xi|^2} \xi \rho[\widehat g(\xi)] \cdot  u[\widehat f(\xi)] 
	\end{aligned}
\end{equation*}
with $0 < \eta_3 \ll \eta_2 \ll \eta_1 \ll 1$, for any $\xi \in \R^3$, where $\id$ is the $3 \times 3$ identity matrix and
$$
\Upsilon [f] = \int_{\R^3} v (|v|^2-5) f \dv, \qquad 
\Theta[f] = \int_{\R^3} \left(v \otimes v - \id\right) f \dv,
$$
and where for vectors $a, b \in \R^3$ and matrices $A, B \in \R^{3 \times 3}$, we denoted
$$(a \otimes b)^{\mathrm{sym}} = \frac12 (a_j b_k + a_kb_j)_{1 \le j,k \le 3}, \qquad A : B = \sum_{j,k=1}^3 A_{jk} B_{jk}.$$
We use $\Psi$ to define the following inner product $\dlla \cdot, \cdot \drra_{L^2_v ( \la v \ra^q \mu^{-1/2} )}$ (with respect to the variable $v$) for some small enough $\kappa > 0$ for any functions $f=f(x, v)$ and $g=g(x, v)$, and any $\xi \in \R^3$ as
\begin{equation}
	\label{eq:equivalent_inner_product}
	\begin{aligned}
		\dlla \widehat f(\xi), \widehat g (\xi) \drra_{L^2_v \left( \la v \ra^q \mu^{-1/2} \right)} 
		:= & \la \widehat f(\xi), \widehat g (\xi) \ra_{L^2_v\left( \mu^{-1/2} \right)} + \Re \Psi[f, g](\xi) \\
		& + \kappa \mathbf{1}_{| \xi | \le 1} \la \widehat f^\perp (\xi), \widehat g^\perp(\xi) \ra_{ L^2_v \left( \la v \ra^q \mu^{-1/2} \right) } \\
		& + \kappa \mathbf{1}_{| \xi | \ge 1} \la \widehat f(\xi), \widehat g (\xi) \ra_{ L^2_v \left( \la v \ra^q \mu^{-1/2} \right) }.
	\end{aligned}
\end{equation}

The next lemma states dissipativity estimates for this inner product.
\begin{lem}
	\label{lem:dissipative_equivalent_gauss}
	Let $q \ge 0$ be fixed. Denote by $\Nt \cdot \Nt_{L^2_v\left ( \la v \ra^q \mu^{-1/2} \right)}$ the norm associated with the inner product $\dlla \cdot, \cdot \drra_{L^2_v \left(\la v \ra^q \mu^{-1/2} \right)}$. The following holds for some constants $C_q > 0$ and~$\lambda_q> 0$.
	\begin{enumerate}[leftmargin=*]
		\item The norm $\Nt \cdot \Nt_{ L^2_v \left( \la v \ra^q \mu^{-1/2} \right) }$ is equivalent to the natural one, and more precisely in Fourier variables, there holds uniformly in $\xi \in \R^3$
		\begin{equation*}
			\frac{1}{C_q} \| \widehat{f}(\xi) \|^2_{L^2_v \left( \la v \ra^q \mu^{-1/2} \right) } \le \Nt \widehat{f}(\xi) \Nt^2_{ L^2_v \left( \la v \ra^q \mu^{-1/2} \right) } \le C_q \| \widehat{f}(\xi) \|^2_{L^2_v \left( \la v \ra^q \mu^{-1/2} \right) } .
		\end{equation*}
		\item The full linearized operator in Fourier variables satisfies the dissipativity estimate
		$$
		\Re\dlla (\LLL - i v \cdot \xi) \widehat{f}(\xi), \widehat{f}(\xi) \drra_{L^2_v \left( \la v \ra^q \mu^{-1/2} \right)}
		\le - \lambda_q \| \widehat{f}(\xi) \|_{H^{s,**}_v\left( \la v \ra^q \mu^{-1/2} \right)}^2
		$$
		where we defined the ($\xi$-dependent) norm $H^{s,**}_v\left( \la v \ra^q \mu^{-1/2} \right)$ as
		\begin{equation}
			\label{eq:def_H_s_**}
			\| \widehat{f}(\xi) \|_{H^{s,**}_v\left( \la v \ra^q \mu^{-1/2} \right)}^2 =
			\| \widehat{f}^\perp(\xi) \|_{H^{s,*}_v( \la v \ra^q \mu^{-1/2})}^2 + \frac{|\xi|^2}{1+|\xi|^2} \| \pi \widehat{f}(\xi) \|_{L^2_v(\mu^{-1/2})}^2.
		\end{equation}
	\end{enumerate}	
\end{lem}

	\begin{proof}
		The norm it induces is equivalent to the natural one of $L^2\left( \la v \ra^q \mu^{-1/2} \right)$ since
		$$\left| \Psi[f, f](\xi) \right| \lesssim \eta_1 \| \widehat{f}(\xi) \|_{L^2_v \left( \la v \ra^q \mu^{-1/2} \right) }^2,$$
		uniformly in $\xi \in \R^3$, where we recall that $\eta_1$ is assumed to be small, and also because
		$$\| \widehat{f}(\xi) \|_{L^2_v\left( \mu^{-1/2} \right)} \ge \| \pi \widehat{f}(\xi) \|_{L^2_v \left( \mu^{-1/2} \right)} \approx \| \pi \widehat{f}(\xi) \|_{L^2_v \left( \la v \ra^q \mu^{-1/2} \right)}.$$
		The dissipativity comes from estimates established in \cite[Section 2.2]{Strain}, which we recall below in a crude but sufficient form; the hypocoercive estimate \cite[(2.20)]{Strain} for some $\alpha > 0$:
		\begin{align*}
			\forall \xi \in \R^3, \quad
			\Re\big\la \big(\LLL - i v \cdot \xi\big) \widehat{f}(\xi), & \widehat{f}(\xi) \big\ra_{ L^2_v \left( \mu^{-1/2} \right) } + \Re \Psi\left[ \left( \LLL - i v \cdot \xi \right) \widehat{f}(\xi), \widehat{f}(\xi) \right](\xi) \\
			& \le - \alpha \| \widehat{f}^\perp(\xi) \|_{ H^{s, *} \left( \mu^{-1/2}  \right) }^2 - \alpha \frac{| \xi |^2}{1 + | \xi |^2} \| \pi \widehat{f}(\xi) \|^2_{ L^2_v \left( \mu^{-1/2} \right) }.
		\end{align*}
		as well as the weighted degenerate estimates \cite[(2.9)]{Strain} and \cite[(2.12)]{Strain} (with $g=0$)
		\begin{align*}
			\forall | \xi | \le 1, \quad \Re\Big\la \left\{ \big(\LLL - i v \cdot \xi\big)  \widehat{f}(\xi) \right\}^\perp , \,  & \widehat{f}^\perp(\xi) \Big\ra_{ L^2_v \left( \la v \ra^q \mu^{-1/2} \right) } + \alpha \|  \widehat{f}^\perp(\xi) \|^2_{ H^{s, *}_v \left( \la v \ra^q \mu^{-1/2} \right) } \\
			& \le C \| \widehat{f}^\perp(\xi) \|^2_{ H^{s, *}_v \left( \mu^{-1/2} \right) } + C | \xi |^2 \| \pi \widehat{f}(\xi) \|^2_{ L^2_v \left( \mu^{-1/2} \right) },
		\end{align*}
		\begin{align*}
			\forall \xi \in \R^3, \quad \Re\big\la \big(\LLL - i v \cdot \xi\big) \widehat{f}(\xi), \widehat{f}(\xi) \big\ra_{ L^2_v \left( \la v \ra^q \mu^{-1/2} \right) } & + \alpha \|  \widehat{f}^\perp(\xi) \|^2_{ H_v^{s, *}\left( \la v \ra^q \mu^{-1/2} \right) } \\
			& \le C \| \widehat{f}(\xi) \|^2_{ H^{s, *}_v \left( \mu^{-1/2} \right) }.
		\end{align*}
		It is clear that for $\kappa$ small enough, the dissipativity estimate for $\dlla \cdot, \cdot \drra_{L^2_v ( \la v \ra^q \mu^{-1/2} )}$ holds:
		$$
		\Re\dlla (\LLL - i v \cdot \xi) \widehat{f}(\xi), \widehat{f}(\xi) \drra_{L^2_v \left( \la v \ra^q \mu^{-1/2} \right)}
		\le - \lambda_q \| \widehat{f}(\xi) \|_{H^{s,**}_v\left( \la v \ra^q \mu^{-1/2} \right)}^2.
		$$
		This concludes the proof.
	\end{proof}

We now introduce the exponentially weighted Sobolev spaces $\spg_q$ and $\disg_q$, with $q \ge 0$, as the spaces associated to the norms
\begin{equation}\label{eq:def-bfE}
	\| f \|_{\spg_q}^2 :=  \| \la v \ra^q f \|_{L^2_x L^2_v ( \mu^{-1/2} ) }^2 + \| \la v \ra^q \nabla_x^3 f \|_{L^2_x L^2_v ( \mu^{-1/2} ) }^2  
\end{equation}
and, respectively,
\begin{equation}\label{eq:def-bfE*}
	\| f \|_{\disg_q}^2 := \| \nabla_x \pi f \|_{ H^2_x L^2_v ( \mu^{-1/2} ) }^2 + \| \la v \ra^q f^\perp \|_{L^2_x E^*_v }^2 + \| \la v \ra^q \nabla^3_x f^\perp \|_{L^2_x E^*_v }^2,
\end{equation}
where the anisotropic space $E^*_v$ is defined by (see \cite{AMUXY2012}, or the equivalent norm of \cite{GS2011_2})
\begin{gather}
	\| f \|_{E^*_v}^2 := \| \la v \ra^{\gamma / 2 + s} f \|^2_{L^2_v ( \mu^{-1/2} ) } + \int |v-v_*|^\gamma b(\cos \theta ) \mu_* (\FF' - \FF)^2 \d \sigma \d v_* \d v, \label{eq:def-E*}
\end{gather}
where we denoted $\FF := \FF(v) = \mu^{-1/2}(v) f(v)$ and $\FF' = \FF(v')$. Let us recall that this norm can be compared to isotropic Sobolev norms \cite[Proposition 2.2]{AMUXY2012}:
\begin{equation*}
	\| \la v \ra^{\gamma / 2 + s} f \|_{L^2_v ( \mu^{-1/2} ) }  + \| \la v \ra^{\gamma/2 } f \|_{H^{s}_v(\mu^{-1/2})} \lesssim \| f \|_{E^*_v} \lesssim \| \la v \ra^{\gamma/2 + s} f \|_{H^{s}_v ( \mu^{-1/2} )}.
\end{equation*}

Using \eqref{eq:equivalent_inner_product}, we can define
the new inner product $ \dlla \cdot, \cdot \drra_{\spg_q}$ on $\spg_q$ by
\begin{equation}\label{eq:def-bfE-equivalent}
\dlla f,g \drra_{\spg_q} = \int_{\R^3}  (1+|\xi|^6) \dlla \widehat f(\xi), \widehat g (\xi) \drra_{L^2_v (\la v \ra^q \mu^{-1/2})}  \d \xi,
\end{equation}
with associated norm $\Nt \cdot \Nt_{\spg_q}$. Observing that $\| \cdot \|_{\disg_q}^2$ is equivalent to
\begin{equation}\label{eq:def-bfE*-equivalent}
\int_{\R^3}  (1+|\xi|^6) \| \widehat f(\xi) \|_{H^{s,**}_v (\la v \ra^q \mu^{-1/2})}^2  \d \xi,
\end{equation}
where we recall that $H^{s, **}_v$ is defined in \eqref{eq:def_H_s_**}, as a direct consequence of Lemma~\ref{lem:dissipative_equivalent_gauss} we obtain a dissipativity-type estimate for $\LLL - v \cdot \nabla_x$ in the space $\spg_q$:

\begin{cor}\label{cor:coercivity_Lambda_gauss}
Let $q \ge 0$ be fixed. The norm $\Nt \cdot \Nt_{\spg_q}$ is equivalent to the natural one $\| \cdot \|_{\spg_q}$, and the full linearized operator satisfies, for some $\lambda>0$ and any smooth enough function $f$,
$$
\dlla (\LLL - v \cdot \nabla_x) f , f \drra_{\spg_q}
\le - \lambda \| f \|_{\disg_q}^2.
$$

\end{cor}

\subsection{Estimates on $\BB-v \cdot \nabla_x$ in polynomially weighted spaces}

We already know from above subsection that in the gaussian spaces $\spg_q$, the full linearized operator $\Lambda = \LLL - v \cdot \nabla_x$ dissipates the $\disg_q$-norm. Concerning the polynomial space $\spp(m)$, we will rely on the following splitting of the linearized collision operator $\LLL$:
\begin{gather*}
	\LLL = \AA + \BB,\\
	\AA := M \chi_R, \quad \BB := \LLL - \AA,
\end{gather*}
with constants $M,R>0$ and $\chi_R(v) = \chi (v/R)$, where $\chi \in C^\infty_c (\R^3)$ is a smooth function satisfying $\mathbf{1}_{|v| \le 1} \le \chi \le \mathbf{1}_{|v| \le 2}$. The parameters $M, R > 0$  will be tuned later (to be chosen large enough) in order to make $\BB -v \cdot \nabla_x$ dissipative.

\begin{prop}
	\label{prop:w_diss_B}
	Assume $k > 13/2 + 2 |\gamma|$ and define $m= \la v \ra^k$. There are $M_0,R_0>0$ large enough such that for any $M \ge M_0$ and $R \ge R_0$ there holds, for any smooth enough function $f$,
	\begin{equation*}
		\la (\BB-v \cdot \nabla_x) f, f \ra_{L^2_{x, v}(m)} \lesssim - \| f \|_{ L^2_x H^{s, *}_v(m) }^2.
	\end{equation*}
\end{prop}

\begin{proof}
	We write
	\begin{equation*}
		\la (\BB-v \cdot \nabla_x) f, f \ra_{L^2_{x, v}(m)} = \la \LLL - \AA f, f \ra_{L^2_{x, v}(m)} - \la v \cdot \nabla_x f, f \ra_{L^2_{x, v}(m)}.
	\end{equation*}
	The second term vanishes in virtue of its gradient structure $(\nabla_x f) \, f = \nabla_x ( |f|^2 )$,	thus by Proposition \ref{prop:dissipative_lot}, we have for some constants $c, C > 0$
	\begin{align*}
		\la (\BB-v \cdot \nabla_x) f, f \ra_{L^2_{x, v}(m)} 
		&\le  - c  \| f \|_{L^2_x H^{s,*}_v(m)}^2 + C \|f\|_{L^2_x L^2_v }^2 -  M \| \chi_R(v) f \|_{L^2_x L^2_v(m)}^2 \\
		&\le - \frac{c}{2} \| f \|_{L^2_x H^{s,*}_v(m)}^2 - \int \left(\frac{c}{2} \la v \ra^\gamma - C m^{-2} + M \chi_R(v) \right) |f|^2 m^2 \d v \d x,
	\end{align*}
	by using that $\| f \|_{L^2_x H^{s,*}_v(m)} \ge \| \la v \ra^{\gamma/2} f \|_{L^2_x L^2_v (m)}$.
	For large values of $|v|$, we have that $\frac{c}{2} \la v \ra^\gamma - C m^{-2} > 0$ by the assumption $k > |\gamma| / 2$, thus there are $M_0,R_0 >0$ large enough such that for all $M \ge M_0$ and $R \ge R_0$ we have
	$$
	\frac{c}{2} \la v \ra^\gamma - C m^{-2} + M \chi_R(v) > 0,
	$$
	from which we deduce the desired estimate.
\end{proof}

As an immediate consequence of Proposition~\ref{prop:w_diss_B} and the fact that $\BB$ commutes with $\nabla_x$, we obtain the following dissipative estimate for $\BB$ in spaces of the type $\mathbf X(m)$ and $\mathbf{X}^*(m)$, recalling the definition in \eqref{eq:def-bfX} and \eqref{eq:def-bfX*}, respectively:
\begin{cor}\label{cor:w_diss_B}
Assume $k > 13/2 + 2 |\gamma|+6s$ and define $m= \la v \ra^k$. There are $M_0,R_0>0$ large enough such that for any $M \ge M_0$ and $R \ge R_0$ there holds, for any smooth enough function $f$,
\begin{equation*}
\la (\BB-v \cdot \nabla_x) f, f \ra_{\mathbf{X}(m)} \lesssim - \| f \|_{\mathbf{X}^* (m)}^2.
\end{equation*}
\end{cor}

We henceforth fix constants $M \ge M_0$ and $R \ge R_0$ in such a way that Proposition~\ref{prop:w_diss_B} and Corollary~\ref{cor:w_diss_B} hold.

\section{Nonlinear estimates}
\label{sec:nonlinear_estimates_inhom}

This section is devoted to inhomogeneous nonlinear estimates for the collision onperator $Q$ that will be needed in the proof our main result in Section~\ref{scn:cauchy_th_non_cutoff}. These estimates are of two different kinds: In Section~\ref{sec:nonlinear_Sobolev} we focus on estimates in the Sobolev-type spaces; whereas in Section~\ref{sec:nonlinear_Fourier} we treat estimates in Fourier-based spaces.

\subsection{Nonlinear estimates in Sobolev-type spaces}
\label{sec:nonlinear_Sobolev}

We shall prove inhomogeneous nonlinear estimates for the collision operator $Q$ by using the homogeneous estimates proven in Section~\ref{scn:est_hom_nc_poly}. More precisely we shall prove bilinear estimates (Proposition~\ref{prop:est_inhom_Q_g_f_f}) and trilinear estimates (Proposition~\ref{prop:est_inhom_Q_f_g_h} and Proposition~\ref{prop:est_inhom_Q_f_g_h_gaussian}), for polynomially weighted spaces $\spp(m)$ as well as for exponentially weighted spaces $\spg_q$ (defined in \eqref{eq:def-bfE}). It is worth mentioning that, because of our strategy employed in Section~\ref{scn:cauchy_th_non_cutoff}, some of these estimates are of \emph{mixed type}, that is, they involve one function in a polynomially weighted space $\spp(m)$ and another function in a exponentially weighted space $\spg_q$.

We start by proving the estimates we will use to prove the stability of the iterative scheme from Section \ref{scn:cauchy_th_non_cutoff}.

\begin{prop}
	\label{prop:est_inhom_Q_g_f_f}
	Assume $k > 13/2 + 2 |\gamma| + 6s$ and consider $m = \la v \ra^k$. For any $f,g \in \spp(m) \cap \dispa (m) $ there holds
	\begin{equation}\label{eq:est_inhom_Q_g_f_f}
		\la Q(g, f), f \ra_{\spp(m)} \lesssim \| f \|_{\dispa( m ) }^2 \| g \|_{ \spp(m) } + \| f \|_{\spp(m)} \| f \|_{\dispa (m)} \| g \|_{ \dispa(m) }.
	\end{equation}
Supposing moreover that $g \in \spg_0 \cap \disg_0 $, then there holds
	\begin{equation}\label{eq:est_inhom_Q_g_f_f_gauss_g}
		\la Q(g, f), f \ra_{\spp(m)} \lesssim \| f \|_{\dispa(m) }^2 \| g \|_{ \spg_0 } + \| f \|_{ \spp(m)} \| f \|_{ \dispa (m)} \| g \|_{\disg_0}.
	\end{equation}
\end{prop}

\begin{proof}
Let us start by expanding the inner product defining the norm of $\spp(m)$ in \eqref{eq:def-bfX}
\begin{align*}
\la Q(g, f), f \ra_{\spp(m)} 
&=  \la Q(g, f), f \ra_{L^2_x L^2_v(m)} 
+ \la \nabla_x^3 Q(g, f), \nabla_x^3 f \ra_{L^2_x L^2_v ( m \la v \ra^{-6 s} ) } ,
\end{align*}
thus we get 
\begin{equation}\label{QgffXm}
\begin{aligned}
&\la Q(g, f), f \ra_{\spp(m)} \\
&\qquad\lesssim \left|  \la Q(g, f), f \ra_{L^2_x L^2_v(m)}  \right|  
+ \sum_{ |\alpha| = 3} \sum_{0 \le \beta \le \alpha} \left|\la Q(\partial^{\alpha - \beta}_x g, \partial_x^\beta f), \partial^{\alpha}_x f \ra_{L^2_x L^2_v ( m \la v \ra^{-2 |\alpha| s} ) } \right|  ,
\end{aligned}
\end{equation}
and we shall estimate each term separately.

We fix some $\ell > 13/2 + 2 |\gamma|$ such that $k \ge \ell + 6s$, and observe that in particular we can apply both estimates of Proposition~\ref{prop:Q-L2v} with the weight $\la v \ra^{k-6s}$ in the sequel.

\step{1}{General estimates of \eqref{QgffXm} in $L^p_x$-norms}
The first term in \eqref{QgffXm} is bounded by integrating estimate \eqref{eq:est_Q_f_g_g} in space and using H\"older's inequality $L^\infty_x-L^2_x-L^2_x$, which yields:
\begin{equation}\label{Qgff-1}
\begin{aligned}
&\la Q(g, f), f \ra_{L^2_x L^2_v(m)} \\
&\qquad \lesssim  
\|  g \|_{L^\infty_x L^2_v (\la v \ra^{\ell})} \| f \|_{L^2_x H^{s,*}_v(m)}^2 \\
&\qquad\quad
+ \| g \|_{L^\infty_x H^{s,*}_v(\la v \ra^{\ell})} \| \la v \ra^{\gamma/2} f \|_{L^2_x L^2_v(m)}  \| f \|_{L^2_x H^{s,*}_v(m)} \\
&\qquad\quad
+ \| \la v \ra^{\gamma/2} g \|_{L^2_x L^2_v(m)} \| f \|_{L^\infty_x H^{s,*}_v(\la v \ra^{\ell})} \| \la v \ra^{\gamma/2} f \|_{L^2_x L^2_v(m)}.
\end{aligned}
\end{equation}
We bound the second term in \eqref{QgffXm} depending on the value of $\beta$. When $\beta = \alpha$, we obtain using \eqref{eq:est_Q_f_g_g} a similar estimate as before:
\begin{equation}\label{Qgff-2}
\begin{aligned}
&\la Q(g, \partial^\alpha_x f), \partial^\alpha_x f \ra_{L^2_x L^2_v(m \la v \ra^{-6 s})}  \\
&\qquad \lesssim 
\|  g \|_{L^\infty_x L^2_v (\la v \ra^{\ell})} \| \nabla^3_x f \|_{L^2_x H^{s,*}_v(m \la v \ra^{-6 s})}^2 \\
&\qquad\quad
+ \| g \|_{L^\infty_x H^{s,*}_v(\la v \ra^{\ell})} \| \la v \ra^{\gamma/2} \nabla^3_x f \|_{L^2_x L^2_v(m \la v \ra^{-6 s})}  \| \nabla^3_x f \|_{L^2_x H^{s,*}_v(m \la v \ra^{-6 s})} \\
&\qquad\quad
+ \| \la v \ra^{\gamma/2} g \|_{L^\infty_x L^2_v(m \la v \ra^{-6 s})} \| \nabla^3_x f \|_{L^2_x H^{s,*}_v(\la v \ra^{\ell})} \| \la v \ra^{\gamma/2} \nabla^3_x f \|_{L^2_x L^2_v(m \la v \ra^{-6 s})}.
\end{aligned}
\end{equation}
When $\beta = 0$, we use \eqref{eq:est_Q_f_g_h} that we integrate in space and using again H\"older's inequality $L^\infty_x-L^2_x-L^2_x$, which gives 
\begin{equation}\label{Qgff-3}
\begin{aligned}
&\la Q( \partial^\alpha_x g, f), \partial^\alpha_x f \ra_{L^2_x L^2_v ( m \la v \ra^{-6 s} ) } \\
&\qquad \lesssim  
\| \la v \ra^{\gamma/2} \nabla^3_x g \|_{L^2_x L^2_v(m \la v \ra^{-6 s})} \|  f \|_{L^\infty_x H^{s,*}_v(\la v \ra^{\ell})}  \| \la v \ra^{\gamma/2} \nabla^3_x f \|_{L^2_x L^2_v(m \la v \ra^{-6 s})} \\
&\qquad\quad
+ \| \nabla^3_x g  \|_{L^2_x H^{s,*}_v(\la v \ra^{\ell})} \| f \|_{L^\infty_x H^{s,*}_v(m \la v \ra^{-6 s})}  \| \la v \ra^{\gamma/2} \nabla^3_x f \|_{L^2_x L^2_v(m \la v \ra^{-6 s})} \\
&\qquad\quad
+  \| \nabla^3_x g  \|_{L^2_x L^2_v (\la v \ra^{\ell})} \| f \|_{L^\infty_x H^{s,*}_v(m \la v \ra^{-4 s})}  \| \nabla^3_x f \|_{L^2_x H^{s,*}_v(m \la v \ra^{-6 s})} .
\end{aligned}
\end{equation}
When $|\beta| = 1$, we we integrate estimate \eqref{eq:est_Q_f_g_h} 
using H\"older's inequality $L^2_x-L^4_x-L^4_x$:
\begin{equation}\label{Qgff-4}
\begin{aligned}
&\la Q(\partial^{\alpha-\beta}_x g, \partial^\beta_x f), \partial^\alpha_x f \ra_{L^2_x L^2_v ( m \la v \ra^{-6 s} ) }\\
&\qquad \lesssim  
\| \la v \ra^{\gamma/2} \nabla^2_x g \|_{L^4_x L^2_v(m \la v \ra^{-6 s})} \| \nabla_x f \|_{L^4_x H^{s,*}_v(\la v \ra^{\ell})}  \| \la v \ra^{\gamma/2} \nabla^3_x f \|_{L^2_x L^2_v(m \la v \ra^{-6 s})} \\
&\qquad\quad
+ \| \nabla^2_x g  \|_{L^4_x H^{s,*}_v(\la v \ra^{\ell})} \| \nabla_x f \|_{L^4_x H^{s,*}_v(m \la v \ra^{-6 s})}  \| \la v \ra^{\gamma/2} \nabla^3_x f \|_{L^2_x L^2_v(m \la v \ra^{-6 s})} \\
&\qquad\quad
+  \| \nabla^2_x g  \|_{L^4_x L^2_v (\la v \ra^{\ell})} \| \nabla_x f \|_{L^4_x H^{s,*}_v(m \la v \ra^{-4 s})}  \| \nabla^3_x f \|_{L^2_x H^{s,*}_v(m \la v \ra^{-6 s})}.
\end{aligned}
\end{equation}
For $|\beta|=2$ we integrate again \eqref{eq:est_Q_f_g_h} in space using H\"older's inequality $L^\infty_x-L^2_x-L^2_x$, which yields
\begin{equation}\label{Qgff-5}
\begin{aligned}
&\la Q(\partial^{\alpha-\beta}_x g, \partial^\beta_x f), \partial^\alpha_x f \ra_{L^2_x L^2_v ( m \la v \ra^{-6 s} ) } \\ 
&\qquad \lesssim  
\| \la v \ra^{\gamma/2} \nabla_x g \|_{L^\infty_x L^2_v(m \la v \ra^{-6 s})} \| \nabla^2_x f \|_{L^2_x H^{s,*}_v(\la v \ra^{\ell})}  \| \la v \ra^{\gamma/2} \nabla^3_x f \|_{L^2_x L^2_v(m \la v \ra^{-6 s})} \\
&\qquad\quad
+ \| \nabla_x g  \|_{L^\infty_x H^{s,*}_v(\la v \ra^{\ell})} \| \nabla^2_x f \|_{L^2_x H^{s,*}_v(m \la v \ra^{-6 s})}  \| \la v \ra^{\gamma/2} \nabla^3_x f \|_{L^2_x L^2_v(m \la v \ra^{-6 s})} \\
&\qquad\quad
+  \| \nabla_x g  \|_{L^\infty_x L^2_v (\la v \ra^{\ell})} \| \nabla^2_x f \|_{L^2_x H^{s,*}_v(m \la v \ra^{-4 s})}  \| \nabla^3_x f \|_{L^2_x H^{s,*}_v(m \la v \ra^{-6 s})}.
\end{aligned}
\end{equation}

\step{2}{Sobolev embeddings for \eqref{eq:est_inhom_Q_g_f_f}} 
We first observe that
$$
\begin{aligned}
\| f \|_{\spp(m)} & \approx \| f \|_{L^2_x L^2_v (m)} + \| \la v \ra^{-2s} \nabla_x f \|_{L^2_x L^2_v (m)} \\
&\quad 
+ \| \la v \ra^{-4s} \nabla^2_x f \|_{L^2_x L^2_v (m)}  + \| \la v \ra^{-6s} \nabla^3_x f \|_{L^2_x L^2_v (m)}
\end{aligned}
$$
and
$$
\begin{aligned}
\| f \|_{\dispa(m)} & \approx \| f \|_{L^2_x H^{s,*}_v (m)} + \| \la v \ra^{-2s} \nabla_x f \|_{L^2_x H^{s,*}_v (m)} \\
&\quad
+ \| \la v \ra^{-4s} \nabla^2_x f \|_{L^2_x H^{s,*}_v (m)}  + \| \la v \ra^{-6s} \nabla^3_x f \|_{L^2_x H^{s,*}_v (m)}
\end{aligned}
$$
Moreover, since $k \ge \ell + 6s$, we have
$$
\| g \|_{H^2_x L^2_v (\la v \ra^\ell) } 
+ \| \nabla_x g \|_{H^1_x L^2_v (\la v \ra^{\ell})}
+ \| \nabla^2_x g \|_{H^1_x L^2_v (\la v \ra^{\ell})}
+ \| \nabla^3_x g \|_{L^2_x L^2_v (\la v \ra^{\ell})} \lesssim \| g \|_{\spp(m)}
$$ 
and
$$
\| g \|_{H^2_x H^{s,*}_v (\la v \ra^\ell) } 
+ \| \nabla_x g \|_{H^1_x H^{s,*}_v (\la v \ra^{\ell})}
+ \| \nabla^2_x g \|_{H^1_x H^{s,*}_v (\la v \ra^{\ell})}
+ \| \nabla^3_x g \|_{L^2_x H^{s,*}_v (\la v \ra^{\ell})} \lesssim \| g \|_{\dispa(m)}.
$$ 
From \eqref{Qgff-1}, using the Sobolev embedding $H^2_x(\R^3) \hookrightarrow L^\infty_x (\R^3)$ we thus get
$$
\la Q(g, f), f \ra_{L^2_x L^2_v(m)} \lesssim \| f \|_{\dispa( m ) }^2 \| g \|_{ \spp(m) } + \| f \|_{\spp(m)} \| f \|_{\dispa (m)} \| g \|_{ \dispa(m) }.
$$
Arguing similarly from \eqref{Qgff-2} we obtain
$$
\la Q(g, \partial^\alpha_x f), \partial^\alpha_x f \ra_{L^2_x L^2_v(m)} \lesssim \| f \|_{\dispa( m ) }^2 \| g \|_{ \spp(m) } + \| f \|_{\spp(m)} \| f \|_{\dispa (m)} \| g \|_{ \dispa(m) }.
$$
For \eqref{Qgff-3}, we use again $H^2_x(\R^3) \hookrightarrow L^\infty_x (\R^3)$ and 
$$
\| f \|_{H^2_x H^{s,*}_v(m \la v \ra^{-4s})} \lesssim \| f \|_{\dispa(m)}
$$
to deduce
$$
\la Q(\partial^\alpha_x g, f), \partial^\alpha_x f \ra_{L^2_x L^2_v(m)} 
\lesssim  \| f \|_{\dispa( m ) }^2 \| g \|_{ \spp(m) }
+\| f \|_{\spp(m)} \| f \|_{\dispa (m)} \| g \|_{ \dispa(m) } .
$$
For $|\beta|=1$, from \eqref{Qgff-4} and the Sobolev embedding $H^1_x(\R^3) \hookrightarrow L^4_x (\R^3)$, we remark that 
$$
\| \nabla_x f \|_{H^1_x H^{s,*}_v(m \la v \ra^{-4s})} \lesssim \| f \|_{\dispa(m)} ,
$$
hence we get 
$$
\la Q(\partial^{\alpha-\beta}_x g, \partial^\beta_x f), \partial^\alpha_x f \ra_{L^2_x L^2_v(m)} 
\lesssim  \| f \|_{\dispa( m ) }^2 \| g \|_{ \spp(m) } + \| f \|_{\spp(m)} \| f \|_{\dispa (m)} \| g \|_{ \dispa(m) }.
$$
Finally, for the case $|\beta|=2$, estimate \eqref{Qgff-5} together with $H^2_x(\R^3) \hookrightarrow L^\infty_x (\R^3)$ yields
$$
\la Q(\partial^{\alpha-\beta}_x g, \partial^\beta_x f), \partial^\alpha_x f \ra_{L^2_x L^2_v(m)} 
\lesssim  \| f \|_{\dispa( m ) }^2 \| g \|_{ \spp(m) } + \| f \|_{\spp(m)} \| f \|_{\dispa (m)} \| g \|_{ \dispa(m) }.
$$
This concludes the proof of \eqref{eq:est_inhom_Q_g_f_f}.

\step{2}{Proof of estimate \eqref{eq:est_inhom_Q_g_f_f_gauss_g}} 
We first remark that 
$$
\| g \|_{H^2_x H^{s,*}_v( \la v \ra^\ell)}
\lesssim \| \pi g \|_{H^2_x L^2_v} + \|  g^\perp \|_{H^2_x H^{s,*}_v( \la v \ra^\ell)}
\lesssim \| g \|_{\spg_0} + \| g \|_{\disg_0},
$$
and that 
$$
\| \la v \ra^{\gamma/2} g \|_{L^2_x L^2_v(m)}
+ \| \la v \ra^{\gamma/2} g \|_{H^2_x L^2_v(m \la v \ra^{-6s})}
\lesssim \| g \|_{\spg_0} .
$$
Moreover
$$
\| \la v \ra^{\gamma/2} f \|_{L^2_x L^2_v (m)} 
+ \| \la v \ra^{\gamma/2} \nabla^3_x f \|_{L^2_x L^2_v (m \la v \ra^{-6s})} 
\lesssim  \min \{ \| f \|_{\spp(m)}   ,  \| f \|_{\dispa(m)} \}.
$$
Therefore from \eqref{Qgff-1} we get 
$$
\la Q(g,f) , f \ra_{L^2_x L^2_v (m)}
\lesssim  \| f \|_{\dispa(m) }^2 \| g \|_{ \spg_0 } + \| f \|_{ \spp(m)} \| f \|_{ \dispa (m)} \| g \|_{\disg_0} ,
$$
and furthermore, for the case $\beta = 0$, we deduce from \eqref{Qgff-2}
$$
\la Q(g, \partial^{\alpha}_x f), \partial^{\alpha}_x f \ra_{L^2_x L^2_v (m \la v \ra^{- 6 s} )} 
\lesssim
\| f \|_{\dispa(m) }^2 \| g \|_{ \spg_0 } + \| f \|_{ \spp(m)} \| f \|_{ \dispa (m)} \| g \|_{\disg_0} .
$$

For all the other cases $|\beta|=1$, $|\beta|=2$ and $\beta= \alpha$, we can argue as in Step~1 by observing that 
$$
\| \nabla_x g \|_{H^2_x H^{s,*}_v (\la v \ra^{\ell})}
+ \| \nabla_x^2 g \|_{H^1_x H^{s,*}_v (\la v \ra^{\ell})}
+ \| \nabla_x^3 g \|_{L^2_x H^{s,*}_v (\la v \ra^{\ell})}
\lesssim \| g \|_{\disg_0},
$$
which thus implies from \eqref{Qgff-3}--\eqref{Qgff-4}--\eqref{Qgff-5} that 
$$
\la Q(\partial^{\alpha-\beta}_x g, \partial^\alpha_x f), \partial^{\alpha}_x f \ra_{L^2_x L^2_v (m \la v \ra^{- 6 s} )} 
\lesssim
\| f \|_{\dispa(m) }^2 \| g \|_{ \spg_0 } + \| f \|_{ \spp(m)} \| f \|_{ \dispa (m)} \| g \|_{\disg_0} .
$$
This concludes the proof of \eqref{eq:est_inhom_Q_g_f_f_gauss_g}.
\end{proof}

We now prove the estimates which we will use to prove the convergence of the iterative scheme in Section \ref{scn:cauchy_th_non_cutoff}.
\begin{prop}
	\label{prop:est_inhom_Q_f_g_h}
	Assume $k > 13/2 + 5 |\gamma|/2 + 6s$ and consider $m = \la v \ra^k$. For any $f,g,h \in \spp(m) \cap \dispa(m)$ there holds
	\begin{equation}
		\label{eq:est_inhom_Q_f_g_h_poly}
		\la Q(f, g), h \ra_{\spp(m)} \lesssim 
		\| f\|_{ \dispa(m) } \| g \|_{ \dispa(m)} \| h \|_{\spp(m)} 
		+  \| f \|_{\spp(m)} \| \la v \ra^{2 s} g \|_{ \dispa(m) } \| h \|_{ \dispa(m) } .
	\end{equation}
Supposin moreover that $g \in \spg_0 \cap \disg_0$, then there holds
	\begin{equation}\label{eq:est_inhom_Q_f_g_h_poly_gauss}
		\begin{aligned}
		\la Q(f, g), h \ra_{\spp(m)}  
		& \lesssim \| f\|_{ \dispa(m) } \| g \|_{ \spg_0} \| h \|_{ \dispa(m) } + \| f \|_{\spp(m)} \| g \|_{ \disg_0 } \| h \|_{ \dispa(m) } \\ &\quad +  \| f\|_{ \dispa(m) } \| g \|_{ \disg_0 } \| h \|_{\spp(m)} ,
		\end{aligned}
	\end{equation}
and
\begin{equation}\label{eq:est_inhom_Q_g_f_h_poly_gauss}
		\begin{aligned}
		&\la Q(g, f), h \ra_{\spp(m)} 
		\lesssim  
		\| \la v \ra^{2s} f\|_{ \dispa(m) } \| g \|_{ \spg_0} \| h \|_{ \dispa(m) }  +  \| f\|_{ \dispa(m) } \| g \|_{ \disg_0 } \| h \|_{\spp(m)} ,
		\end{aligned}
	\end{equation}
\end{prop}

\begin{proof}
By expanding the inner product of $\spp(m)$, we are led to estimate 
\begin{equation}\label{QfghXm}
\begin{aligned}
\la Q(f, g), h \ra_{\spp(m)} 
&\lesssim \left|  \la Q(f, g), h \ra_{L^2_x L^2_v(m)}  \right| \\
&\quad 
+ \sum_{|\alpha| = 3} \sum_{0 \le \beta \le \alpha} \left|\la Q(\partial^{\alpha - \beta}_x f, \partial_x^\beta g), \partial^{\alpha}_x h \ra_{L^2_x L^2_v ( m \la v \ra^{-6 s} ) } \right|  .
\end{aligned}
\end{equation}
The proof of each one of the estimates \eqref{eq:est_inhom_Q_f_g_h_poly}, \eqref{eq:est_inhom_Q_f_g_h_poly_gauss} and \eqref{eq:est_inhom_Q_g_f_h_poly_gauss} then follows the same approach: For each term appearing in \eqref{QfghXm} we integrate in space the corresponding homogeneous estimate and then use H\"older's inequality and Sobolev embeddings arguing similarly as in Step~1 of the proof of Proposition~\ref{prop:est_inhom_Q_g_f_f}. 

We fix some $\ell > 13/2 + 2|\gamma|$ such that $k \ge \ell + 6s + |\gamma|/2$, and remark that we can apply estimate \eqref{eq:est_Q_f_g_h} of Proposition~\ref{prop:Q-L2v}  with the weight $\la v \ra^{k-6s}$ in the sequel.

\step{1}{General estimates of \eqref{QfghXm} in $L^p_x$-norms}
The first term in \eqref{QfghXm} is estimated using \eqref{eq:est_Q_f_g_h} and H\"older's inequality $L^\infty_x-L^2_x-L^2_x$, which yields
\begin{equation}\label{Qfgh-1}
\begin{aligned}
&\la Q(f, g), h \ra_{L^2_x L^2_v (m)} \\
&\qquad\lesssim  
\| \la v \ra^{\gamma/2} f \|_{L^2_x L^2_v(m)} \|  g \|_{L^\infty_x H^{s,*}_v(\la v \ra^{\ell})}  \| \la v \ra^{\gamma/2} h \|_{L^2_x L^2_v(m)} \\
&\qquad\quad
+ \| f \|_{L^\infty_x H^{s,*}_v(\la v \ra^{\ell})} \| g \|_{L^2_x H^{s,*}_v(m)}  \| \la v \ra^{\gamma/2} h \|_{L^2_x L^2_v(m)} \\
&\qquad\quad
+  \| f \|_{L^\infty_x L^2_v (\la v \ra^{\ell})} \| \la v \ra^{2s} g \|_{L^2_x H^{s,*}_v(m)}  \| h \|_{L^2_x H^{s,*}_v(m)}.
\end{aligned}
\end{equation}

The second term in \eqref{QfghXm} is then estimated depending on the value of $\beta$. For the case $\beta = 0$, we also have by using \eqref{eq:est_Q_f_g_h} and H\"older's inequality $L^\infty_x-L^2_x-L^2_x$ that
\begin{equation}\label{Qfgh-2}
\begin{aligned}
&\la Q(\partial^{\alpha}_x f, g), \partial^{\alpha}_x h\ra_{L^2_x L^2_v (m \la v \ra^{- 6 s} )} \\
&\qquad\lesssim  
\| \la v \ra^{\gamma/2} \nabla^3_x f \|_{L^2_x L^2_v(m \la v \ra^{- 6 s})} \|  g \|_{L^\infty_x H^{s,*}_v(\la v \ra^{\ell})}  \| \la v \ra^{\gamma/2} \nabla^3_x h \|_{L^2_x L^2_v(m \la v \ra^{- 6 s})} \\
&\qquad\quad
+ \| \nabla^3_x f \|_{L^2_x H^{s,*}_v(\la v \ra^{\ell})} \| g \|_{L^\infty_x H^{s,*}_v(m \la v \ra^{- 6 s})}  \| \la v \ra^{\gamma/2} \nabla^3_x h \|_{L^2_x L^2_v(m \la v \ra^{- 6 s})} \\
&\qquad\quad
+  \| \nabla^3_x f \|_{L^2_x L^2_v (\la v \ra^{\ell})} \| \la v \ra^{2s} g \|_{L^\infty_x H^{s,*}_v(m \la v \ra^{- 6 s})}  \| \nabla^3_x h \|_{L^2_x H^{s,*}_v(m \la v \ra^{- 6 s})}.
\end{aligned}
\end{equation}
When $|\beta|=1$ we use H\"older's inequality $L^4_x-L^4_x-L^2_x$ to get
\begin{equation}\label{Qfgh-3}
\begin{aligned}		
&\la Q(\partial^{\alpha - \beta}_x f, \partial_x^\beta g), \partial^{\alpha}_x h \ra_{L^2_x L^2_v(m \la v \ra^{-6s} )}\\
&\qquad \lesssim 
\| \la v \ra^{\gamma/2} \nabla^2_x f \|_{L^4_x L^2_v(m \la v \ra^{- 6 s})} \| \nabla_x g \|_{L^4_x H^{s,*}_v(\la v \ra^{\ell})}  \| \la v \ra^{\gamma/2} \nabla^3_x h \|_{L^2_x L^2_v(m \la v \ra^{- 6 s})} \\
&\qquad\quad
+ \| \nabla^2_x f \|_{L^4_x H^{s,*}_v(\la v \ra^{\ell})} \| \nabla_x g \|_{L^4_x H^{s,*}_v(m \la v \ra^{- 6 s})}  \| \la v \ra^{\gamma/2} \nabla^3_x h \|_{L^2_x L^2_v(m \la v \ra^{- 6 s})} \\
&\qquad\quad
+  \| \nabla^2_x f \|_{L^4_x L^2_v (\la v \ra^{\ell})} \| \la v \ra^{2s} \nabla_x g \|_{L^4_x H^{s,*}_v(m \la v \ra^{- 6 s})}  \| \nabla^3_x h \|_{L^2_x H^{s,*}_v(m \la v \ra^{- 6 s})}.
\end{aligned}
\end{equation}
and for $|\beta|=2$:
\begin{equation}\label{Qfgh-4}
\begin{aligned}	
&\la Q(\partial^{\alpha - \beta}_x f, \partial_x^\beta g), \partial^{\alpha}_x h \ra_{L^2_x L^2_v(m \la v \ra^{-6s} )}\\
&\qquad \lesssim 
\| \la v \ra^{\gamma/2} \nabla_x f \|_{L^4_x L^2_v(m \la v \ra^{- 6 s})} \| \nabla^2_x g \|_{L^4_x H^{s,*}_v(\la v \ra^{\ell})}  \| \la v \ra^{\gamma/2} \nabla^3_x h \|_{L^2_x L^2_v(m \la v \ra^{- 6 s})} \\
&\qquad\quad
+ \| \nabla_x f \|_{L^4_x H^{s,*}_v(\la v \ra^{\ell})} \| \nabla^2_x g \|_{L^4_x H^{s,*}_v(m \la v \ra^{- 6 s})}  \| \la v \ra^{\gamma/2} \nabla^3_x h \|_{L^2_x L^2_v(m \la v \ra^{- 6 s})} \\
&\qquad\quad
+  \| \nabla_x f \|_{L^4_x L^2_v (\la v \ra^{\ell})} \| \la v \ra^{2s} \nabla^2_x g \|_{L^4_x H^{s,*}_v(m \la v \ra^{- 6 s})}  \| \nabla^3_x h \|_{L^2_x H^{s,*}_v(m \la v \ra^{- 6 s})}.
\end{aligned}
\end{equation}	
Finally, for $\beta=\alpha$ we get, using again H\"older's inequality $L^\infty_x-L^2_x-L^2_x$,
\begin{equation}\label{Qfgh-5}
\begin{aligned}		
&\la Q( f, \partial_x^\alpha g), \partial^{\alpha}_x h \ra_{L^2_x L^2_v (m \la v \ra^{- 6 s} )} \\
&\qquad\lesssim
\| \la v \ra^{\gamma/2}  f \|_{L^\infty_x L^2_v(m \la v \ra^{- 6 s})} \| \nabla^3_x g \|_{L^2_x H^{s,*}_v(\la v \ra^{\ell})}  \| \la v \ra^{\gamma/2} \nabla^3_x h \|_{L^2_x L^2_v(m \la v \ra^{- 6 s})} \\
&\qquad\quad
+ \|  f \|_{L^\infty_x H^{s,*}_v(\la v \ra^{\ell})} \| \nabla^3_x g \|_{L^2_x H^{s,*}_v(m \la v \ra^{- 6 s})}  \| \la v \ra^{\gamma/2} \nabla^3_x h \|_{L^2_x L^2_v(m \la v \ra^{- 6 s})} \\
&\qquad\quad
+  \|  f \|_{L^\infty_x L^2_v (\la v \ra^{\ell})} \| \la v \ra^{2s} \nabla^3_x g \|_{L^2_x H^{s,*}_v(m \la v \ra^{- 6 s})}  \| \nabla^3_x h \|_{L^2_x H^{s,*}_v(m \la v \ra^{- 6 s})} .
	\end{aligned}
\end{equation}	
We now split the proof into three steps.

\step{2}{Sobolev embeddings to prove \eqref{eq:est_inhom_Q_f_g_h_poly}}
It follows from estimates \eqref{Qfgh-1}--\eqref{Qfgh-2}--\eqref{Qfgh-3}--\eqref{Qfgh-4}--\eqref{Qfgh-5} by using the Sobolev embeddings $H^2_x(\R^3) \hookrightarrow L^\infty_x (\R^3)$ and  $H^1_x(\R^3) \hookrightarrow  L^4_x (\R^3)$ as in Step~1 of Proposition~\ref{prop:est_inhom_Q_g_f_f}.

\step{3}{Proof of \eqref{eq:est_inhom_Q_f_g_h_poly_gauss}}
We first observe that 
$$
\| g \|_{H^2_x H^{s,*}_v( \la v \ra^\ell )}
\lesssim \| \pi g \|_{H^2_x L^2_v} + \|  g^\perp \|_{L^2_x H^{s,*}_v( \la v \ra^\ell )}
\lesssim \| g \|_{\spg_0} + \| g \|_{\disg_0},
$$
and
$$
\| \la v \ra^{2s} g \|_{L^2_x H^{s,*}_v(m)}
\lesssim \| \pi g \|_{L^2_x L^2_v} + \| \la v \ra^{2s} g^\perp \|_{L^2_x H^{s,*}_v(m)}
\lesssim \| g \|_{\spg_0} + \| g \|_{\disg_0}.
$$
Moreover 
$$
\| \la v \ra^{\gamma/2} f \|_{L^2_x L^2_v (m)} + \| \la v \ra^{\gamma/2} f \|_{H^2_x L^2_v (\la v \ra^{\ell})}
\lesssim \min \left\{ \| f \|_{\spp(m)} , \| f \|_{\dispa(m)} \right\}
$$
since $k \ge \ell + 6s$. Therefore from \eqref{Qfgh-1} we get 
$$
\begin{aligned}
&\la Q(f,g) , h \ra_{L^2_x L^2_v (m)} \\
&\qquad 
\lesssim  \| f\|_{ \dispa(m) } \| g \|_{ \spg_0} \| h \|_{ \dispa(m) } + \| f \|_{\spp(m)} \| g \|_{ \disg_0 } \| h \|_{ \dispa(m) }  
+  \| f\|_{ \dispa(m) } \| g \|_{ \disg_0 } \| h \|_{\spp(m)}.
\end{aligned}
$$
For the case $\beta = 0$ we remark that, similarly as above, we have 
$$
\| \la v \ra^{2s} g \|_{H^2_x H^{s,*}_v(m \la v \ra^{-6s})}
\lesssim \| \pi g \|_{H^2_x L^2_v} + \| \la v \ra^{2s} g^\perp \|_{H^2_x H^{s,*}_v(m \la v \ra^{-6s})}
\lesssim \| g \|_{\spg_0} + \| g \|_{\disg_0}.
$$
moreover
$$
\| \la v \ra^{\gamma/2} \nabla^3_x f \|_{L^2_x L^2_v (m \la v \ra^{-6s})} 
+\| \nabla^3_x f \|_{L^2_x L^2_v (\la v \ra^\ell)} 
\lesssim \min \left\{ \| f \|_{\spp(m)} , \| f \|_{\dispa(m)} \right\},
$$
since $k + \gamma/2 \ge \ell + 6s$. Hence we deduce from \eqref{Qfgh-2}
$$
\begin{aligned}
&\la Q(\partial^{\alpha}_x f, g), \partial^{\alpha}_x h\ra_{L^2_x L^2_v (m \la v \ra^{- 6 s} )} \\
&\qquad 
\lesssim  \| f\|_{ \dispa(m) } \| g \|_{ \spg_0} \| h \|_{ \dispa(m) } + \| f \|_{\spp(m)} \| g \|_{ \disg_0 } \| h \|_{ \dispa(m) }  
+  \| f\|_{ \dispa(m) } \| g \|_{ \disg_0 } \| h \|_{\spp(m)}.
\end{aligned}
$$
For all the other cases $|\beta|=1$, $|\beta|=2$ and $\beta= \alpha$, we can argue as in Step~1 by observing that 
$$
\| \nabla_x g \|_{L^4_x H^{s,*}_v (\la v \ra^{\ell})}
+ \| \nabla_x^2 g \|_{L^4_x H^{s,*}_v (\la v \ra^{\ell})}
+ \| \nabla_x^3 g \|_{L^2_x H^{s,*}_v (\la v \ra^{\ell})}
\lesssim \| g \|_{\disg_0},
$$
as well as
$$
\| \nabla_x g \|_{L^4_x H^{s,*}_v (m \la v \ra^{-4s})}
+ \| \nabla_x^2 g \|_{L^4_x H^{s,*}_v (m \la v \ra^{-4s})}
+ \| \nabla_x^3 g \|_{L^2_x H^{s,*}_v (m \la v \ra^{-4s})}
\lesssim \| g \|_{\disg_0},
$$
which thus implies from \eqref{Qfgh-3}--\eqref{Qfgh-4}--\eqref{Qfgh-5} that 
$$
\la Q(\partial^{\alpha-\beta}_x f, \partial^\alpha_x g), \partial^{\alpha}_x h\ra_{L^2_x L^2_v (m \la v \ra^{- 6 s} )} 
\lesssim
\| g \|_{ \disg_0} \left( \| f\|_{ \dispa(m) } \| h \|_{\spp(m)}  + \| f \|_{\spp(m) } \| h \|_{\dispa(m)} \right).
$$

\step{4}{Proof of \eqref{eq:est_inhom_Q_g_f_h_poly_gauss}}
It follows similarly as in Step~2 above, so we omit the proof.
\end{proof}

\subsection{Nonlinear estimates in Fourier-based spaces}
\label{sec:nonlinear_Fourier}

We shall now prove nonlinear estimates in Fourier-based spaces by using the estimates of Proposition~\ref{prop:Q-L2v}.
For polynomial weight functions $m = \la v \ra^k$ we define
\begin{equation}\label{eq:def-Nm}
\NN_m [f,g,h] (\xi) = \left\{ \int_0^\infty \left|\left\la \widehat Q (  f(t) ,  g(t)) (\xi) , \widehat h(t,\xi) \right\ra_{L^2_v(m)}\right|  \d t  \right\}^{1/2}.
\end{equation}

\begin{prop}\label{prop:nonlinear_Nm_fgh}
Let $p \in [1, \infty]$. Let $k> 13/2 + 5|\gamma|/2$ and define the weight function $m = \la v \ra^k$. For any $f,\la v \ra^{2s} g , h \in (L^1_\xi \cap L^p_\xi) L^\infty_t L^2_v(m) \cap(L^1_\xi \cap L^p_\xi) L^2_t H^{s,*}_v(m)$ there holds
\begin{equation}\label{eq:Nm_fgh}
\begin{aligned}
\| \NN_m [f,g,h] \|_{L^p_\xi}
&\lesssim 
\| \widehat f \|_{L^p_\xi L^\infty_t L^2_v (m)}^{1/2} \| \la v \ra^{2s} \widehat g \|_{L^1_\xi L^2_t H^{s,*}_v(m)}^{1/2} \| \widehat h \|_{L^p_\xi L^2_t H^{s,*}_v(m)}^{1/2} \\
&\quad
+ \| \widehat f \|_{L^p_\xi L^2_t H^{s,*}_v (m)}^{1/2} \| \widehat g \|_{L^1_\xi L^2_t H^{s,*}_v(m)}^{1/2} \| \widehat h \|_{L^p_\xi L^\infty_t L^2_v(m)}^{1/2}.
\end{aligned}
\end{equation}
Supposing moreover that $g \in (L^1_\xi \cap L^p_\xi) L^\infty_t L^2_v(\mu^{-1/2}) \cap(L^1_\xi \cap L^p_\xi) L^2_t H^{s,**}_v(\mu^{-1/2})$, then there holds
\begin{equation}\label{eq:Nm_fgh_exp}
\begin{aligned}
\| \NN_m [f,g,h] \|_{L^p_\xi}
&\lesssim 
\| \widehat f  \|_{L^p_\xi L^2_t H^{s,*}_v(m)}^{1/2}  \| \widehat g  \|_{L^1_\xi L^2_t H^{s,**}_v(\mu^{-1/2})}^{1/2}  \| \widehat h \|_{L^p_\xi L^\infty_t L^2_v(m)}^{1/2} \\
&\quad
+  \| \widehat f  \|_{L^1_\xi L^2_t H^{s,*}_v(m)}^{1/2}  \| \widehat g  \|_{L^p_\xi L^\infty_t L^2_v(\mu^{-1/2})}^{1/2}  \| \widehat h \|_{L^p_\xi L^2_t H^{s,*}_v(m)}^{1/2} \\
&\quad
+  \| \widehat f  \|_{L^p_\xi L^\infty_t L^2_v(m)}^{1/2}  \| \widehat g  \|_{L^1_\xi L^2_t H^{s,**}_v(\mu^{-1/2})}^{1/2}  \| \widehat h \|_{L^p_\xi L^2_t H^{s,*}_v(m)}^{1/2},
\end{aligned}
\end{equation}
and also
\begin{equation}\label{eq:Nm_gfh_exp}
\begin{aligned}
\| \NN_m [g,f,h] \|_{L^p_\xi}
&\lesssim 
 \| \widehat g  \|_{L^1_\xi L^2_t H^{s,**}_v(\mu^{-1/2})}^{1/2}   \| \widehat f  \|_{L^p_\xi L^2_t H^{s,*}_v(m)}^{1/2} \| \widehat h \|_{L^p_\xi L^\infty_t L^2_v(m)}^{1/2} \\
&\quad
+  \| \widehat g  \|_{L^p_\xi L^\infty_t L^2_v(\mu^{-1/2})}^{1/2} \| \la v \ra^{2s} \widehat f  \|_{L^1_\xi L^2_t H^{s,*}_v(m)}^{1/2}   \| \widehat h \|_{L^p_\xi L^2_t H^{s,*}_v(m)}^{1/2} .\end{aligned}
\end{equation}

\end{prop}

\begin{proof}
Observe that we can apply Proposition~\ref{prop:Q-L2v} with $\la v \ra^\ell \lesssim \la v \ra^{\gamma/2} m $ and we split the proof into three steps.

\medskip\noindent
\textit{Step 1: Proof of \eqref{eq:Nm_fgh}.}
Thanks to estimate \eqref{eq:est_Q_f_g_h} we first obtain
\begin{equation}\label{eq:Q_hatf_hatg_hath}
\begin{aligned}
&\left|\left\la \widehat Q (  f ,  g) (\xi) , \widehat h(\xi) \right\ra_{L^2_v(m)} \right| 
\lesssim
\int_{\R^3} \left| \left\la  Q ( \widehat f(\xi-\eta) , \widehat g(\eta) )  , \widehat h(\xi) \right\ra_{L^2_v(m)}   \right| \d \eta \\
&\qquad\lesssim
\| \widehat h(\xi) \|_{L^2_v(\la v \ra^{\gamma/2} m)} \int_{\R^3} \| \widehat f (\xi-\eta) \|_{H^{s,*}_v(m)} \| \widehat g (\eta) \|_{H^{s,*}_v(m)} \d \eta  \\
&\qquad\quad
+ \| \widehat h(\xi) \|_{H^{s,*}_v(m)} \int_{\R^3} \| \widehat f (\xi-\eta) \|_{L^2_v(\la v \ra^{\gamma/2} m)} \| \widehat g (\eta) \|_{H^{s,*}_v(\la v \ra^{2s} m)} \d \eta ,
\end{aligned}
\end{equation}
which implies
\begin{equation*}
\NN_m [f,g,h] (\xi)
\lesssim I_1(\xi) + I_2(\xi)
\end{equation*}
with
$$
I_1(\xi) := \left\{ \int_0^\infty \| \widehat h(t,\xi) \|_{L^2_v(m)} \int_{\R^3} \| \widehat f (t,\xi-\eta) \|_{H^{s,*}_v(m)} \| \widehat g (t,\eta) \|_{H^{s,*}_v(m)} \d \eta \d t \right\}^{1/2}
$$
and
$$
I_2(\xi) := \left\{ \int_0^\infty \| \widehat h(t,\xi) \|_{H^{s,*}_v(m)} \int_{\R^3} \| \widehat f (t,\xi-\eta) \|_{L^2_v(m)} \| \widehat g (t,\eta) \|_{H^{s,*}_v(\la v \ra^{2s} m)} \d \eta \d t \right\}^{1/2}.
$$
On the one hand we have
\begin{equation*}
\begin{aligned}
I_1(\xi)^2 
&\lesssim  \| \widehat h(\xi) \|_{L^\infty_t L^2_v(m)} \left\{\int_0^\infty \int_{\R^3} \| \widehat f (t,\xi-\eta) \|_{H^{s,*}_v(m)} \| \widehat g (t,\eta) \|_{H^{s,*}_v(m)} \d \eta \d t \right\} \\
&\lesssim  \| \widehat h(\xi) \|_{L^\infty_t L^2_v(m)}  \int_{\R^3} \left[ \int_0^\infty \| \widehat f (t,\xi-\eta) \|_{H^{s,*}_v(m)}^2 \d t\right]^{1/2} \left[ \int_0^\infty \| \widehat g (t,\eta) \|_{H^{s,*}_v(m)}^2 \d t \right]^{1/2} \d \eta  
\\
&\lesssim  \| \widehat h(\xi) \|_{L^\infty_t L^2_v(m)}  \int_{\R^3} \| \widehat f (\xi-\eta) \|_{L^2_t H^{s,*}_v(m)}  \| \widehat g (\eta) \|_{L^2_t H^{s,*}_v(m)}  \d \eta ,
\end{aligned}
\end{equation*}
thanks to H\"older and Cauchy-Schwarz inequalities. Moreover, using the same inequalities, we also obtain
\begin{equation*}
\begin{aligned}
I_2(\xi)^2 
&\lesssim
\| \widehat h(\xi) \|_{L^2_t H^{s,*}_v(m)}  \left\{ \int_0^\infty \left[ \int_{\R^3} \| \widehat f (t,\xi-\eta) \|_{L^2_v(m)} \| \widehat g (t,\eta) \|_{H^{s,*}_v(\la v \ra^{2s} m)} \d \eta \right]^2 \d t \right\}^{1/2} \\
&\lesssim
\| \widehat h(\xi) \|_{L^2_t H^{s,*}_v(m)}   \int_{\R^3}  \left[ \int_0^\infty  \| \widehat f (t,\xi-\eta) \|_{L^2_v(m)}^2 \| \widehat g (t,\eta) \|_{H^{s,*}_v(\la v \ra^{2s} m)}^2 \d t \right]^{1/2} \d \eta  \\
&\lesssim
\| \widehat h(\xi) \|_{L^2_t H^{s,*}_v(m)}  \int_{\R^3}    \| \widehat f (\xi-\eta) \|_{L^\infty_t L^2_v(m)} \| \widehat g (\eta) \|_{L^2_t H^{s,*}_v(\la v \ra^{2s} m)}  \d \eta .
\end{aligned}
\end{equation*}
Taking the $L^p_\xi$ norms of each term yields
\begin{equation}\label{eq:I1_Lp_xi}
\begin{aligned}
\| I_1 \|_{L^p_\xi} 
&\lesssim  
\left\| \| \widehat h(\xi) \|_{L^\infty_t L^2_v(m)}^{1/2}  \left\{ \int_{\R^3} \| \widehat f (\xi-\eta) \|_{L^2_t H^{s,*}_v(m)}  \| \widehat g (\eta) \|_{L^2_t H^{s,*}_v(m)}  \d \eta \right\}^{1/2} \right\|_{L^p_\xi}\\
&\lesssim   
\| \widehat f  \|_{L^p_\xi L^2_t H^{s,*}_v(m)}^{1/2}  \| \widehat g  \|_{L^1_\xi L^2_t H^{s,*}_v(m)}^{1/2}  \| \widehat h(\xi) \|_{L^p_\xi L^\infty_t L^2_v(m)}^{1/2}  ,
\end{aligned}
\end{equation}
as well as
\begin{equation}\label{eq:I2_Lp_xi}
\begin{aligned}
\| I_2 \|_{L^p_\xi} 
&\lesssim  
\left\| \| \widehat h(\xi) \|_{L^2_t H^{s,*}_v(m)}^{1/2}  \left\{ \int_{\R^3} \| \widehat f (\xi-\eta) \|_{L^\infty_t L^2_v(m)}  \| \widehat g (\eta) \|_{L^2_t H^{s,*}_v(\la v \ra^{2s} m)}  \d \eta \right\}^{1/2} \right\|_{L^p_\xi}\\
&\lesssim   
\| \widehat f  \|_{L^p_\xi L^\infty_t L^2_v(m)}^{1/2}  \| \widehat g  \|_{L^1_\xi L^2_t H^{s,*}_v(\la v \ra^{2s} m)}^{1/2}  \| \widehat h(\xi) \|_{L^p_\xi L^2_t H^{s,*}_v(m)}^{1/2}  ,
\end{aligned}
\end{equation}
where we have used Hölder's inequality in $L^{2p}_\xi \times L^{2p}_\xi$ followed by Young's inequality in~$L^p_\xi \times L^1_\xi$. Gathering previous estimates concludes the proof of \eqref{eq:Nm_fgh}.

\medskip\noindent
\textit{Step 2: Proof of \eqref{eq:Nm_fgh_exp}.}
Starting from estimate \eqref{eq:Q_hatf_hatg_hath}, we first observe that
\begin{equation}\label{eq:hatg_decomposition}
\begin{aligned}
\| \widehat g (\eta) \|_{H^{s,*}_v(m)}  
\lesssim \| \widehat g (\eta) \|_{H^{s,*}_v(\la v \ra^{2s} m)} 
&\lesssim \| \widehat{g}^\perp (\eta) \|_{H^{s,*}_v(\la v \ra^{2s} m)} + \| \pi \widehat{g}(\eta) \|_{L^2_v(\la v \ra^{2s} m)} \\
&\lesssim \| \widehat{g}^\perp (\eta) \|_{H^{s,*}_v(\mu^{-1/2})} + \| \pi \widehat{ g}(\eta) \|_{L^2_v(\mu^{-1/2})} \\
&\lesssim \| \widehat{g} (\eta) \|_{H^{s,**}_v(\mu^{-1/2})} + \| \widehat{g}(\eta) \|_{L^2_v(\mu^{-1/2})} .
\end{aligned}
\end{equation}
On the other hand, we also remark that
\begin{align}
\| \widehat f (\xi-\eta) \|_{L^2_v(\la v \ra^{\gamma/2} m)} 
&\le \min\left\{ \| \widehat f (\xi-\eta) \|_{L^2_v(m)} , \| \widehat f (\xi-\eta) \|_{H^{s,*}_v(m)} \right\} \label{eq:hatf_min_L2_Hs}\\
\| \widehat h(\xi) \|_{L^2_v(\la v \ra^{\gamma/2} m)} 
&\le \min\left\{ \| \widehat h(\xi) \|_{L^2_v(m)} , \| \widehat h(\xi) \|_{H^{s,*}_v(m)} \right\} \label{eq:hath_min_L2_Hs}.
\end{align}
Therefore
\begin{equation*}
\begin{aligned}
\left|\left\la \widehat Q (  f ,  g) (\xi) , \widehat h(\xi) \right\ra_{L^2_v(m)} \right|
&\lesssim
\| \widehat h(\xi) \|_{L^2_v(m)} \int_{\R^3} \| \widehat f (\xi-\eta) \|_{H^{s,*}_v(m)} \| \widehat{g} (\eta) \|_{H^{s,**}_v(\mu^{-1/2})} \d \eta  \\
&\quad
+\| \widehat h(\xi) \|_{H^{s,*}_v(m)}  \int_{\R^3} \| \widehat f (\xi-\eta) \|_{H^{s,*}_v(m)} \| \widehat{g}(\eta) \|_{L^2_v(\mu^{-1/2})} \d \eta  \\
&\quad
+ \| \widehat h(\xi) \|_{H^{s,*}_v(m)} \int_{\R^3} \| \widehat f (\xi-\eta) \|_{L^2_v(m)} \| \widehat{g} (\eta) \|_{H^{s,**}_v(\mu^{-1/2})} \d \eta ,
\end{aligned}
\end{equation*}
which implies
\begin{equation*}
\NN_m [f,g,h] (\xi)
\lesssim T_1(\xi) + T_2(\xi) + T_3(\xi) 
\end{equation*}
with
$$
T_1(\xi) := \left\{ \int_0^\infty \| \widehat h(t,\xi) \|_{L^2_v(m)} \int_{\R^3} \| \widehat f (t,\xi-\eta) \|_{H^{s,*}_v(m)} \| \widehat{g} (t,\eta) \|_{H^{s,**}_v(\mu^{-1/2})} \d \eta  \d t \right\}^{1/2} ,
$$
$$
T_2(\xi) := \left\{ \int_0^\infty \| \widehat h(t, \xi) \|_{H^{s,*}_v(m)}  \int_{\R^3} \| \widehat f (t, \xi-\eta) \|_{H^{s,*}_v(m)} \| \widehat{g}(t, \eta) \|_{L^2_v(\mu^{-1/2})} \d \eta  \d t \right\}^{1/2}, 
$$
and
$$
T_3(\xi) := \left\{ \int_0^\infty  \| \widehat h(t, \xi) \|_{H^{s,*}_v(m)} \int_{\R^3} \| \widehat f (t, \xi-\eta) \|_{L^2_v(m)} \| \widehat{g} (t, \eta) \|_{H^{s,**}_v(\mu^{-1/2})} \d \eta \d t \right\}^{1/2}.
$$

For the term $T_1$ we can argue exactly as for obtaining estimate \eqref{eq:I1_Lp_xi} for $I_1$ in Step~2, which yields
$$
\| T_1 \|_{L^p_\xi} \lesssim \| \widehat f  \|_{L^p_\xi L^2_t H^{s,*}_v(m)}^{1/2}  \| \widehat g  \|_{L^1_\xi L^2_t H^{s,**}_v(\mu^{-1/2})}^{1/2}  \| \widehat h(\xi) \|_{L^p_\xi L^\infty_t L^2_v(m)}^{1/2} .
$$
Finally, for the terms $T_2$ and $T_3$ we can argue exactly as for obtaining estimate \eqref{eq:I2_Lp_xi} for $I_2$ in Step~2, thus we obtain
$$
\| T_2 \|_{L^p_\xi} 
\lesssim  \| \widehat f  \|_{L^1_\xi L^2_t H^{s,*}_v(m)}^{1/2}  \| \widehat g  \|_{L^p_\xi L^\infty_t L^2_v(\mu^{-1/2})}^{1/2}  \| \widehat h(\xi) \|_{L^p_\xi L^2_t H^{s,*}_v(m)}^{1/2}
$$
and
$$
\| T_3 \|_{L^p_\xi} 
\lesssim  \| \widehat f  \|_{L^p_\xi L^\infty_t L^2_v(m)}^{1/2}  \| \widehat g  \|_{L^1_\xi L^2_t H^{s,**}_v(\mu^{-1/2})}^{1/2}  \| \widehat h(\xi) \|_{L^p_\xi L^2_t H^{s,*}_v(m)}^{1/2}.
$$
The proof of \eqref{eq:Nm_fgh_exp} is then complete by gathering previous estimates.

\medskip\noindent
\textit{Step 3: Proof of \eqref{eq:Nm_gfh_exp}.} 
Thanks to estimate \eqref{eq:est_Q_f_g_h} we obtain
$$
\begin{aligned}
\left|\left\la \widehat Q (  g ,  f) (\xi) , \widehat h(\xi) \right\ra_{L^2_v(m)} \right| 
&\lesssim \| \widehat h(\xi) \|_{L^2_v(\la v \ra^{\gamma/2} m)} \int_{\R^3} \| \widehat f (\xi-\eta) \|_{H^{s,*}_v(m)} \| \widehat g (\eta) \|_{H^{s,*}_v(m)} \d \eta  \\
&\quad
+ \| \widehat h(\xi) \|_{H^{s,*}_v(m)} \int_{\R^3} \| \widehat f (\xi-\eta) \|_{H^{s,*}_v(\la v \ra^{2s} m)} \| \widehat g (\eta) \|_{L^2_v(\la v \ra^{\gamma/2} m)} \d \eta .
\end{aligned}
$$
As in Step 2, we use \eqref{eq:hatg_decomposition} and \eqref{eq:hath_min_L2_Hs} to obtain
$$
\begin{aligned}
\left|\left\la \widehat Q (  g ,  f) (\xi) , \widehat h(\xi) \right\ra_{L^2_v(m)} \right| 
&\lesssim 
\| \widehat h(\xi) \|_{L^2_v(m)} \int_{\R^3} \| \widehat f (\xi-\eta) \|_{H^{s,*}_v(m)} \| \widehat g (\eta) \|_{H^{s,**}_v(\mu^{-1/2})} \d \eta  \\
&\quad
+ \| \widehat h(\xi) \|_{H^{s,*}_v(m)} \int_{\R^3} \| \widehat f (\xi-\eta) \|_{H^{s,*}_v(\la v \ra^{2s} m)} \| \widehat g (\eta) \|_{L^2_v(\mu^{-1/2})} \d \eta .
\end{aligned}
$$
We then prove \eqref{eq:Nm_gfh_exp} by arguing as in previous steps.
\end{proof}

We now provide nonlinear estimates involving exponential weights. For $q \ge 0$ we define
\begin{equation}\label{eq:def-tildeNq}
\widetilde{\NN}_q [f,g,h] (\xi) = \left\{ \int_0^\infty \left|\left\la \!\! \left\la \widehat Q (  f(t) ,  g(t)) (\xi) , \widehat h(t,\xi) \right\ra \!\! \right\ra_{L^2_v \left(\la v \ra^q \mu^{-1/2} \right)} \right|  \d t  \right\}^{1/2}
\end{equation}
where we recall that $ \la \! \la \cdot , \cdot \ra \! \ra_{L^2_v \left(\la v \ra^q \mu^{-1/2} \right) }$ is defined in \eqref{eq:equivalent_inner_product}, and that
$$\widehat{Q}(f, g)(\xi) = \int_{\R^3} Q\left( \widehat{f}(\xi - \eta) , \widehat{g}(\eta) \right) \d \eta .
$$
Before proving the result on $\widetilde{\NN}_q$, we state the following nonlinear estimate for $Q$ in the new inner product:

\begin{lem}
	\label{lem:est_Q_fgauss_ggauss_hgauss}
	Let $q \ge 0$ be fixed.
	For any $f=f(x, v)$, $g=g(x, v)$ and $h=h(x, v)$ such that $\widehat{f}(\xi), \widehat{g}(\xi), \widehat{h}(\xi) \in L^2_v \cap H^{s,**}_v(\la v \ra^q \mu^{-1/2})$ one has uniformly in $\xi \in \R^3$
	\begin{equation}
		\label{eq:est_Q_f_g_h_gauss_fourier}
		\begin{aligned}
			\dlla \widehat{Q}( f, & g)(\xi), \widehat{h}(\xi) \drra_{ L^2_v \left( \la v \ra^q \mu^{-1/2} \right) } \\
			\lesssim & \| \widehat{h}(\xi) \|_{H^{s, **}_v \left( \la v \ra^q \mu^{-1/2} \right) } \int_{\R^3} \Big( \| \widehat{f}(\xi - \eta) \|_{ L^2_v\left( \la v \ra^q \mu^{-1/2} \right) } \| \widehat{g}(\eta) \|_{ H^{s, **}_v\left( \la v \ra^q \mu^{-1/2} \right) } \\
			& \qquad \qquad \qquad \qquad \qquad \quad + \| \widehat{f}(\xi - \eta) \|_{ H^{s, **}_v\left( \la v \ra^q \mu^{-1/2} \right) } \| \widehat{g}(\eta) \|_{ L^2_v\left( \la v \ra^q \mu^{-1/2} \right) } \Big) \d \eta \\
			& + \| \widehat{h}(\xi) \|_{H^{s, **}_v \left( \la v \ra^q \mu^{-1/2} \right) } \int_{\R^3} \Big( \| \widehat{f}(\xi - \eta) \|_{ L^2_v\left( \la v \ra^q \mu^{-1/2} \right) } \| \pi \widehat{g}(\eta) \|_{ L^2_v } \\
			& \qquad \qquad \qquad \qquad \qquad \qquad + \| \pi \widehat{f}(\xi - \eta) \|_{ L^2_v } \| \widehat{g}(\eta) \|_{ L^2_v\left( \la v \ra^q \mu^{-1/2} \right) } \Big) \d \eta.
		\end{aligned}
	\end{equation}
\end{lem}

\begin{proof}
	We first recall the nonlinear estimates coming respectively from \cite[Theorem 1.2]{AMUXY2012} and \cite[Proposition 3.13]{AMUXY2012}, both in a crude form using the assumption $\gamma + 2 s < 0$:
	\begin{align*}
		\la Q(f_1, f_2), f_3 \ra_{ L^2_v\left( \mu^{-1/2} \right) } \lesssim \| f_3 \|_{H^{s, *}_v\left( \mu^{-1/2} \right) } \Big( & \| f_1 \|_{H^{s, *}_v\left( \mu^{-1/2} \right) } \| f_2 \|_{ L^2_v \left( \mu^{-1/2} \right) } \\
		& +  \| f_1 \|_{ L^2_v \left( \mu^{-1/2} \right) } \| f_2 \|_{H^{s, *}_v\left( \mu^{-1/2} \right) } \Big)
	\end{align*}
	\begin{align*}
		\big| \la Q(f_1, f_2), f_3 \ra_{L^2\left( \la v \ra^q \mu^{-1/2} \right) } &- \la Q(f_1, \la v \ra^q f_2), \la v \ra^q f_3 \ra_{L^2\left( \mu^{-1/2} \right) } \big| \\
		 & \lesssim \| f_3 \|_{ H^{s, *}_v \left( \la v \ra^q \mu^{-1/2} \right) } \| f_1 \|_{ L^2_v\left( \la v \ra^q \mu^{-1/2} \right) } \| f_2 \|_{ L^2_v\left( \la v \ra^q \mu^{-1/2} \right) },
	\end{align*}
	which together yield
	\begin{align*}
		\la Q(f_1, f_2), f_3 \ra_{ L^2_v\left( \la v \ra^q \mu^{-1/2} \right) } \lesssim \| f_3 \|_{H^{s, *}_v\left( \la v \ra^q \mu^{-1/2} \right) } \Big( & \| f_1 \|_{H^{s, *}_v\left( \la v \ra^q \mu^{-1/2} \right) } \| f_2 \|_{ L^2_v \left( \la v \ra^q \mu^{-1/2} \right) } \\
		& +  \| f_1 \|_{ L^2_v \left( \la v \ra^q \mu^{-1/2} \right) } \| f_2 \|_{H^{s, *}_v\left( \la v \ra^q \mu^{-1/2} \right) } \Big).
	\end{align*}
	Let us turn to the estimates in term of $\dlla \cdot, \cdot \drra_{L^2_v ( \la v \ra^q \mu^{-1/2} )}$. Concerning the $\Psi$--estimate, note that since~$\pi Q = 0$, the bilinear term $\Psi\left[ Q(\widehat{f}(\xi - \eta), \widehat{g}(\eta) ) , \widehat{h}(\xi) \right]$ reduces to
	\begin{align*}
		\Psi\Big[ Q\left(\widehat{f}(\xi - \eta), \widehat{g}(\eta) \right) , \widehat{h}(\xi) \Big]
		= & \frac{i \kappa_1 }{1 + | \xi |^2} \xi \widehat{\theta}_h(\xi) \cdot \Upsilon\left[ Q\left(\widehat{f}(\xi - \eta), \widehat{g}(\eta) \right) \right] \\
		& + \frac{i \kappa_2 }{1 + | \xi |^2} \left( \xi \otimes \widehat{u}_h(\xi) \right) : \Theta\left[ Q\left(\widehat{f}(\xi - \eta), \widehat{g}(\eta) \right) \right],
	\end{align*}
	which is then easily estimated as (where $p=p(v)$ is some polynomial)
	\begin{align*}
		\Psi\Big[ Q\big(\widehat{f}(\xi - \eta)&, \widehat{g}(\eta) \big) , \widehat{h}(\xi) \Big](\xi) \\
		\lesssim & \frac{|\xi|}{1+| \xi|} \| \pi \widehat{h}(\xi) \|_{L^2_v\left( \mu^{-1/2} \right)} \times \frac{1}{1+|\xi|} \left| \left\la p \mu, Q\left(\widehat{f}(\xi - \eta), \widehat{g}(\eta) \right) \right\ra_{ L^2_v \left( \mu^{-1/2} \right) } \right| \\
		\lesssim & \| \widehat{h}(\xi) \|_{H^{s, **}_v \left( \la v \ra^q \mu^{-1/2}  \right) } \Big( \|\widehat{f}(\xi - \eta)\|_{ H^{s, *}_v\left( \mu^{-1/2} \right) } \| \widehat{g}(\eta) \|_{ L^2_v\left( \mu^{-1/2} \right) } \\
		& \qquad \qquad \qquad \qquad \qquad + \|\widehat{f}(\xi - \eta)\|_{ L^2_v\left( \mu^{-1/2} \right) } \| \widehat{g}(\eta) \|_{ H^{s, *}_v\left( \mu^{-1/2} \right) } \Big),
	\end{align*}
	and using the fact that
	\begin{align*}
		\| \widehat{\varphi}(\xi) \|_{H^{s, *}_v \left( \la v \ra^q \mu^{-1/2} \right) } & \lesssim \| \pi \widehat{\varphi}(\xi) \|_{L^2_v} + \| \widehat{\varphi}^\perp(\xi) \|_{H^{s, *}_v \left( \la v \ra^q \mu^{-1/2} \right) } \\
		& \lesssim \| \pi \widehat{\varphi}(\xi) \|_{L^2_v} + \| \widehat{\varphi}(\xi) \|_{H^{s, **}_v \left( \la v \ra^q \mu^{-1/2} \right) },
	\end{align*}
	we actually have
	\begin{align*}
		\Psi\Big[ Q\big(\widehat{f}(\xi - \eta)&, \widehat{g}(\eta) \big) , \widehat{h}(\xi) \Big](\xi) \\
		\lesssim & \| \widehat{h}(\xi) \|_{H^{s, **}_v \left( \la v \ra^q \mu^{-1/2} \right) } \Big( \| \widehat{f}(\xi - \eta) \|_{ H^{s, **}_v\left( \la v \ra^q \mu^{-1/2} \right) } \| \widehat{g}(\eta) \|_{ L^2_v\left( \la v \ra^q \mu^{-1/2} \right) } \\
		& \qquad \qquad \qquad \qquad \quad + \| \widehat{f}(\xi - \eta) \|_{ L^2_v\left( \la v \ra^q \mu^{-1/2} \right) } \| \widehat{g}(\eta) \|_{ H^{s, **}_v\left( \la v \ra^q \mu^{-1/2} \right) } \Big) \\
		& + \| \widehat{h}(\xi) \|_{H^{s, **}_v \left( \la v \ra^q \mu^{-1/2} \right) } \Big( \| \widehat{f}(\xi - \eta) \|_{ L^2_v\left( \la v \ra^q \mu^{-1/2} \right) } \| \pi \widehat{g}(\eta) \|_{ L^2_v } \\
		& \qquad \qquad \qquad \qquad \qquad \quad + \| \pi \widehat{f}(\xi - \eta) \|_{ L^2_v } \| \widehat{g}(\eta) \|_{ L^2_v\left( \la v \ra^q \mu^{-1/2} \right) } \Big).
	\end{align*}
	Next, using again $\pi Q = 0$, i.e. $Q^\perp = Q$, we have the term for $| \xi | \le 1$:
	\begin{align*}
		\mathbf{1}_{| \xi | \le 1} \Big\la & Q\left(\widehat{f}(\xi - \eta), \widehat{g}(\eta) \right), \widehat{h}^\perp(\xi) \Big\ra_{ L^2_v\left( \la v \ra^q \mu^{-1/2} \right) } \\
		\lesssim &\mathbf{1}_{| \xi | \le 1} \| \widehat{h}^\perp(\xi) \|_{H^{s, *}_v\left( \la v \ra^q \mu^{-1/2} \right) } \Big( \| \widehat{f}(\xi - \eta) \|_{H^{s, *}_v\left( \la v \ra^q \mu^{-1/2} \right) } \| \widehat{g}(\eta) \|_{ L^2_v \left( \la v \ra^q \mu^{-1/2} \right) } \\
		& \qquad \qquad \qquad \qquad \quad +  \| \widehat{f}(\xi - \eta) \|_{ L^2_v \left( \la v \ra^q \mu^{-1/2} \right) } \| \widehat{g}(\eta) \|_{H^{s, *}_v\left( \la v \ra^q \mu^{-1/2} \right) } \Big) \\
		\lesssim & \mathbf{1}_{| \xi | \le 1} \| \widehat{h}(\xi) \|_{H^{s, **}_v \left( \la v \ra^q \mu^{-1/2} \right) } \Big( \| \widehat{f}(\xi - \eta) \|_{ H^{s, **}_v\left( \la v \ra^q \mu^{-1/2} \right) } \| \widehat{g}(\eta) \|_{ L^2_v\left( \la v \ra^q \mu^{-1/2} \right) } \\
		& \qquad \qquad \qquad \qquad \quad + \| \widehat{f}(\xi - \eta) \|_{ L^2_v\left( \la v \ra^q \mu^{-1/2} \right) } \| \widehat{g}(\eta) \|_{ H^{s, **}_v\left( \la v \ra^q \mu^{-1/2} \right) } \Big) \\
		& + \| \widehat{h}(\xi) \|_{H^{s, **}_v \left( \la v \ra^q \mu^{-1/2} \right) } \Big( \| \widehat{f}(\xi - \eta) \|_{ L^2_v\left( \la v \ra^q \mu^{-1/2} \right) } \| \pi \widehat{g}(\eta) \|_{ L^2_v } \\
		& \qquad \qquad \qquad \qquad \qquad \quad + \| \pi \widehat{f}(\xi - \eta) \|_{ L^2_v } \| \widehat{g}(\eta) \|_{ L^2_v\left( \la v \ra^q \mu^{-1/2} \right) } \Big)
	\end{align*}
	and, using that $\frac{|\xi|}{1+|\xi|^2} \approx 1$ for $| \xi | \ge 1$:
	\begin{align*}
		\mathbf{1}_{| \xi | \ge 1} \Big\la & Q\big(\widehat{f}(\xi - \eta), \widehat{g}(\eta) \big) , \widehat{h}(\xi) \Big\ra_{ L^2_v\left( \la v \ra^q \mu^{-1/2} \right) } \\
		\lesssim & \mathbf{1}_{| \xi | \ge 1}  \| h \|_{H^{s, *}_v\left( \la v \ra^q \mu^{-1/2} \right) } \Big( \| \widehat{f}(\xi - \eta) \|_{H^{s, *}_v\left( \la v \ra^q \mu^{-1/2} \right) } \| \widehat{g}(\eta) \|_{ L^2_v \left( \la v \ra^q \mu^{-1/2} \right) } \\
		& \qquad \qquad \qquad \qquad \qquad +  \| \widehat{f}(\xi - \eta) \|_{ L^2_v \left( \la v \ra^q \mu^{-1/2} \right) } \| \widehat{g}(\eta) \|_{H^{s, *}_v\left( \la v \ra^q \mu^{-1/2} \right) } \Big) \\
		\lesssim & \| \widehat{h}(\xi) \|_{H^{s, **}_v \left( \la v \ra^q \mu^{-1/2} \right) } \Big( \| \widehat{f}(\xi - \eta) \|_{ H^{s, **}_v\left( \la v \ra^q \mu^{-1/2} \right) } \| \widehat{g}(\eta) \|_{ L^2_v\left( \la v \ra^q \mu^{-1/2} \right) } \\
		& \qquad \qquad \qquad \qquad \quad + \| \widehat{f}(\xi - \eta) \|_{ L^2_v\left( \la v \ra^q \mu^{-1/2} \right) } \| \widehat{g}(\eta) \|_{ H^{s, **}_v\left( \la v \ra^q \mu^{-1/2} \right) } \Big) \\
		& + \| \widehat{h}(\xi) \|_{H^{s, **}_v \left( \la v \ra^q \mu^{-1/2} \right) } \Big( \| \widehat{f}(\xi - \eta) \|_{ L^2_v\left( \la v \ra^q \mu^{-1/2} \right) } \| \pi \widehat{g}(\eta) \|_{ L^2_v } \\
		& \qquad \qquad \qquad \qquad \qquad \quad + \| \pi \widehat{f}(\xi - \eta) \|_{ L^2_v } \| \widehat{g}(\eta) \|_{ L^2_v\left( \la v \ra^q \mu^{-1/2} \right) } \Big).
	\end{align*}
	We concludes by gathering previous estimates.
\end{proof}

We can now estimate $\widetilde N_q$.

\begin{prop}\label{prop:tildeNq_fgh}
Consider some $q \ge 0$, as well as some $p \in [1, \infty]$ and $r \in (3/2, \infty]$. For any~$h \in L^p_\xi L^\infty_t L^2_v(\la v \ra^q \mu^{-1/2})$ and $f,g \in L^p_\xi L^\infty_t L^2_v(\la v \ra^q \mu^{-1/2}) \cap(L^1_\xi \cap L^p_\xi) L^2_t H^{s,**}_v(\la v \ra^q\mu^{-1/2})$, there holds
\begin{equation}\label{eq:tildeNq_fgh_exp}
\begin{aligned}
\| & \widetilde \NN_q [f, g, h] \|_{L^p_\xi} \\
&\lesssim \| \widehat h \|_{L^p_\xi L^2_t H^{s,**}_v(\la v \ra^q \mu^{-1/2})}^{1/2} \Big( 
\| \widehat f \|_{L^p_\xi L^\infty_t L^2_v (\la v \ra^q \mu^{-1/2})}^{1/2} 
\|  \widehat g \|_{\left(L^1_\xi \cap L^r_\xi\right) L^2_t H^{s,**}_v(\la v \ra^q \mu^{-1/2})}^{1/2} 
 \\
&
\qquad \qquad  \qquad\qquad \qquad \quad + \| \widehat f \|_{\left(L^1_\xi \cap L^r_\xi\right) L^2_t H^{s,**}_v (\la v \ra^q \mu^{-1/2})}^{1/2} 
\| \widehat g \|_{L^p_\xi L^\infty_t L^{2}_v(\la v \ra^q \mu^{-1/2})}^{1/2} \Big).
\end{aligned}
\end{equation}

\end{prop}

\begin{proof}
	We have from \eqref{eq:est_Q_f_g_h_gauss_fourier} the bound
	$$\| \widetilde{\NN}_q [f, g, h] \|_{ L^p_\xi } \lesssim \| I_1 \|_{ L^p_\xi }^{1/2} + \| I_2 \|_{ L^p_\xi }^{1/2} + \| J_1 \|_{ L^p_\xi }^{1/2} + \| J_2 \|_{ L^p_\xi }^{1/2}$$
	where each term is defined as
	\begin{align*}
		I_1(\xi) := \int_{\R^3} \int_0^\infty \| \widehat{h}(t, \xi) \|_{H^{s, **}_v \left( \la v \ra^q \mu^{-1/2}  \right) } \| \widehat{f}(t, \xi - \eta) \|_{ L^2_v\left( \mu^{-1/2} \right) } \| \widehat{g}(t, \eta) \|_{ H^{s, **}_v\left( \la v \ra^q \mu^{-1/2}  \right) } \d t \d \eta
	\end{align*}
	\begin{gather*}
		I_2(\xi)^2 := \int_{\R^3} \int_0^\infty \| \widehat{h}(t, \xi) \|_{H^{s, **}_v \left( \la v \ra^q \mu^{-1/2}  \right) } \| \widehat{f}(t, \xi - \eta) \|_{ H^{s, **}_v\left( \la v \ra^q \mu^{-1/2}  - \eta \right) } \| \widehat{g}(t, \eta) \|_{ L^2_v\left( \la v \ra^q \mu^{-1/2} \right) } \d t \d \eta , \\
		J_1(\xi)^2 := \int_{\R^3} \int_0^\infty  \| \widehat{h}(t, \xi) \|_{H^{s, **}_v \left( \la v \ra^q \mu^{-1/2}  \right) }  \| \widehat{f}(t , \xi - \eta) \|_{ L^2_v\left( \la v \ra^q \mu^{-1/2} \right) } \| \pi \widehat{g}(t, \eta) \|_{ L^2_v } \d t \d \eta, \\
		J_2(\xi)^2 := \int_{\R^3} \int_0^\infty  \| \widehat{h}(t, \xi) \|_{H^{s, **}_v \left( \la v \ra^q \mu^{-1/2}  \right) }  \| \pi \widehat{f}(t, \xi - \eta) \|_{ L^2_v } \| \widehat{g}(t, \eta) \|_{ L^2_v\left( \la v \ra^q \mu^{-1/2} \right) } \d t \d \eta.
	\end{gather*}
	Using Hölder's inequality in $L^2_t \times L^2_t \times L^\infty_t$, we have
	$$I_1(\xi)^2 \lesssim \| \widehat{h}(\xi) \|_{L^2_t H^{s, **}_v \left( \la v \ra^q \mu^{-1/2}  \right) } \int_{\R^3} \| \widehat{g}(\xi - \eta) \|_{ L^2_t H^{s, **}_v\left( \la v \ra^q \mu^{-1/2} \right) } \| \widehat{f}(\eta) \|_{ L^\infty_t L^2_v\left( \la v \ra^q \mu^{-1/2} \right) } \d \eta,$$
	and then, by Hölder's inequality in $L^{2 p}_\xi \times L^{2 p }_\xi$
	$$\| I_1\|_{L^p_\xi} \lesssim \| \widehat{h} \|_{L^p_\xi L^2_t H^{s, **}_v \left( \la v \ra^q \mu^{-1/2} \right) }^{1/2} \left\| \int_{\R^3} \| \widehat{g}(\xi - \eta) \|_{ L^2_t H^{s, **}_v\left( \la v \ra^q \mu^{-1/2}  \right) } \| \widehat{f}(\eta) \|_{ L^\infty_t L^2_v\left( \la v \ra^q \mu^{-1/2} \right) } \d \eta \right\|_{L^p_\xi}^{1/2},$$
	from which we deduce thanks to Young's convolution inequality in $L^1_\xi \times L^p_\xi \rightarrow L^p_\xi$
	$$\| I_1\|_{L^p_\xi} \lesssim \| \widehat{h} \|_{L^p_\xi L^2_t H^{s, **}_v \left( \la v \ra^q \mu^{-1/2} \right) }^{1/2} \| \widehat{g} \|_{ L^1_\xi L^2_t H^{s, **}_v\left( \la v \ra^q \mu^{-1/2} \right) }^{1/2} \| \widehat{f} \|_{ L^p_\xi  L^\infty_t L^2_v\left( \la v \ra^q \mu^{-1/2} \right) }^{1/2}.$$
	Similarly, there holds
	$$\| I_2\|_{L^p_\xi} \lesssim \| \widehat{h} \|_{L^p_\xi L^2_t H^{s, **}_v \left( \la v \ra^q \mu^{-1/2}  \right) }^{1/2} \| \widehat{f} \|_{ L^1_\xi L^2_t H^{s, **}_v\left( \la v \ra^q \mu^{-1/2} \right) }^{1/2} \| \widehat{g} \|_{ L^p_\xi  L^\infty_t L^2_v\left( \la v \ra^q \mu^{-1/2} \right) }^{1/2},$$
	$$\| J_1\|_{L^p_\xi} \lesssim \| \widehat{h} \|_{L^p_\xi L^2_t H^{s, **}_v \left( \la v \ra^q \mu^{-1/2} \right) }^{1/2} \| \pi \widehat{g} \|_{ L^1_\xi L^2_t L^2_v }^{1/2} \| \widehat{f} \|_{ L^p_\xi  L^\infty_t L^2_v\left( \la v \ra^q \mu^{-1/2} \right) }^{1/2},$$
	$$\| J_2 \|_{L^p_\xi} \lesssim \| \widehat{h} \|_{L^p_\xi L^2_t H^{s, **}_v \left( \la v \ra^q \mu^{-1/2}  \right) }^{1/2} \| \pi \widehat{f} \|_{ L^1_\xi L^2_t L^2_v }^{1/2} \| \widehat{g} \|_{ L^p_\xi  L^\infty_t L^2_v\left( \la v \ra^q \mu^{-1/2} \right) }^{1/2}.$$
	Finally, we notice that, since $r > 3/2$, we have that $\mathbf{1}_{| \xi | \le 1} | \xi |^{-1} \in L^{r'}_\xi$, thus
	\begin{align*}
		\| \pi \widehat{\varphi} \|_{ L^1_\xi L^2_t L^2_v\left( \la v \ra^q \mu^{-1/2} \right) } \le & \| \mathbf{1}_{| \xi | \le 1} \pi \widehat{\varphi} \|_{ L^1_\xi L^2_t L^2_v\left( \la v \ra^q \mu^{-1/2} \right) } + \| \mathbf{1}_{| \xi | \ge 1} \pi \widehat{\varphi} \|_{ L^1_\xi L^2_t L^2_v\left( \la v \ra^q \mu^{-1/2} \right) } \\
		\lesssim & \| \widehat{\varphi} \|_{ L^r_\xi L^2_t H^{s,**}_v\left( \la v \ra^q \mu^{-1/2} \right) } + \| \widehat{\varphi} \|_{ L^1_\xi L^2_t H^{s,**}_v\left( \la v \ra^q \mu^{-1/2} \right) },
	\end{align*}
	from which we deduce \eqref{eq:tildeNq_fgh_exp}. This concludes the proof.
\end{proof}

As a direct consequence of Lemma~\ref{lem:est_Q_fgauss_ggauss_hgauss} and the proof of Proposition~\ref{prop:tildeNq_fgh}, we also deduce nonlinear estimates for $Q$ in the exponentially weighted Sobolev spaces $\spg_q$, with $q \ge 0$, defined in \eqref{eq:def-bfE}, by using the new inner product $\dlla  \cdot , \cdot \drra_{\spg_q}$ defined in \eqref{eq:def-bfE-equivalent} and its associated norm $\Nt \cdot \Nt_{\spg_q}$, together with the space $\disg_q$ defined in \eqref{eq:def-bfE*}.

\begin{prop}\label{prop:est_inhom_Q_f_g_h_gaussian}
Let $q \ge 0$. For any $f,g,h \in \spg_q \cap \spg^*_q$ one has
	\begin{gather}
		\label{eq:est_Q_f_g_h_gauss}
		\dlla Q(f, g), h \drra_{\spg_q} \lesssim \| h \|_{\disg_q} \left( \| f \|_{ \spg_q } \| g \|_{ \disg_q } + \| f \|_{ \disg_q } \| g \|_{\spg_q} \right).
	\end{gather}
	Furthermore, for any smooth compactly supported function $\chi = \chi(v)$, we also have for any $\ell \ge 0$
	\begin{gather}
		\label{eq:est_regularized_twisted_naive}
		\dlla g, \chi f \drra_{\spg_q} \lesssim \| g \|_{\spg_q} \| f \|_{\spp(\la v \ra^\ell)},\\
		\label{eq:est_regularized_twisted}
		\dlla g, \chi f \drra_{\spg_q} \lesssim \| g \|_{\disg_q} \| f \|_{\spp( \la v \ra^\ell )} + \| \pi g \|_{L^2_{x, v}} \| f \|_{\dispa( \la v \ra^\ell )}.
	\end{gather}
\end{prop}

\section{Proof of the main result}
\label{scn:cauchy_th_non_cutoff}

This section is devoted to the proof of Theorem~\ref{theo:non-cutoff}. We fix 
\begin{equation}\label{eq:k}
k 
> 13/2 + 7|\gamma| / 2 + 8s 
\end{equation}
and define the weight function $m=\la v \ra^k$. We also consider $\theta , k_0 >0$ such that
\begin{equation}\label{eq:k0_theta}
1 < \theta < \frac{k - k_0}{|\gamma|} 
\quad \text{and} \quad
k_0 > 13/2 + 5 |\gamma| / 2 + 8s 
\end{equation}
and define the weight function $m_0 = \la v \ra^{k_0}$. Finally we define the weight functions 
\begin{equation}\label{eq:underline_m}
\underline{m} = m \la v \ra^{-2s} = \la v \ra^{k-2s}, \quad 
\underline{m_0} = m_0 \la v \ra^{-2s} = \la v \ra^{k_0-2s}.
\end{equation}
Observe that we have 
$$
k_0 - 2s > 13/2 + 5 |\gamma| / 2+ 6s 
$$
so that we may apply in the sequel Propositions~\ref{prop:est_inhom_Q_g_f_f} and~\ref{prop:est_inhom_Q_f_g_h} as well as Corollary~\ref{cor:w_diss_B} with the weight function $\underline{m_0}$, and consequently also with the weights $\underline m$, $m_0$ and $m$.

Drawing inspiration from \cite{BMM2019}, we seek a solution to \eqref{eq:main-pert} of the form
$$
f(t) = h(t) + g(t) \in \spp(m) + \spg
$$
where the two parts evolve according to the differential system
\begin{equation}\label{eq:systeme-hg}
	\begin{cases}
		\partial_t h = (\BB  - v \cdot \nabla_x) h + Q(h, h) + Q(g, h) + Q(h, g),\\
		\partial_t g = (\LLL - v \cdot \nabla_x) g + Q(g, g) + \AA h,\\
		h(0, x, v) = f_0(x, v), \quad g(0, x, v) = 0.
	\end{cases}
\end{equation}
We will construct a solution to this system by building a sequence of approximate solutions~$(h_N, g_N)_{N=0}^\infty$ initialized as $(h_0, g_0) = 0$ and defined inductively by
\begin{equation}\label{eq:systeme-hgN}
	\begin{cases}
		\partial_t h_{N+1} = (\BB  - v \cdot \nabla_x) h_{N+1} + Q(h_N, h_{N+1}) + Q(g_N, h_{N+1}) + Q(h_{N+1}, g_{N}),\\
		\partial_t g_{N+1} = (\LLL - v \cdot \nabla_x) g_{N+1} + Q(g_N, g_{N+1}) + \AA h_N,\\
		h_{N+1}(0, x, v) = f_0(x, v), \quad g_{N+1}(0, x, v) = 0.
	\end{cases}
\end{equation}
To do so, we introduce the functional spaces $\sppt(m)$, $\sppdt(m_0)$ and $\spgt_q$, with $q \ge 0$, as the spaces associated to the norms, respectively:
\begin{align}
	\| h \|_{\sppt(m)}^2 &:= \sup_{t \ge 0} \| h(t) \|_{\spp(m)}^2 + \int_0^\infty \| h(t) \|_{\dispa(m)}^2 \d t, \label{eq:XXXm}\\
	\| h \|_{\sppdt(m_0)}^2 &:= \sup_{t \ge 0} \, (1+t)^{2\theta} \| h(t) \|_{\spp\left( m_0 \right)}^2 + \int_0^\infty (1+t)^{2\theta} \| h(t) \|_{\dispa\left( m_0 \right)}^2 \d t, \label{eq:YYYm}\\
	\| g \|_{\spgt_q}^2 &:= \sup_{t \ge 0} \Nt g(t) \Nt_{\spg_q}^2 + \int_0^\infty \| g(t) \|_{\disg_q}^2 \d t, \label{eq:EEE}
\end{align}
where we recall that the norm $\Nt \cdot \Nt_{\spg_q}$ is defined in \eqref{eq:equivalent_inner_product} and it is equivalent to $\| \cdot \|_{\spg_q}$, and that the spaces $\spp(m)$, $\dispa(m)$, $\spg_q$ and $\disg_q$ are defined respectively in \eqref{eq:def-bfX}, \eqref{eq:def-bfX*}, \eqref{eq:def-bfE} and \eqref{eq:def-bfE*}. 
Similarly, we also consider the spaces $\sppt(\underline{m})$ and $\sppdt(\underline{m_0})$ defined respectively by \eqref{eq:XXXm} and \eqref{eq:YYYm} but with the weights $\underline{m}$ and $\underline{m_0}$.

\subsection{Stability of the scheme}\label{sec:scheme_stability}

In this subsection, we will show by induction that if the initial data satisfies
$$
\| f_0 \|_{\spp(m)} \le  \eps_0
$$
with $\eps_0>0$ small enough, then the following bound holds for all $N \ge 0$:
\begin{equation}
	\label{eq:simple_inductive_bound}
	 \| h_N \|_{\sppt(m)} + \| h_N \|_{\sppdt(m_0)} + \| g_N \|_{\spgt_q} \lesssim \| f_0 \|_{\spp(m)} \le \eps_0.
\end{equation}
This is of course true for $N = 0$. Assume this bound for some~$N \ge 0$ and let us deduce it for $N+1$. 


\subsubsection{Stability of $h_{N+1}$ in norm $\sppt(m)$}\label{sec:stability_hN_X}

We start with the first equation of \eqref{eq:systeme-hgN}. The weak coercivity estimate on $\BB-v\cdot \nabla_x$ from Corollary~\ref{cor:w_diss_B} gives for some $\lambda > 0$
\begin{gather*}
	\la (\BB-v\cdot \nabla_x) h_{N+1}, h_{N+1} \ra_{\spp(m)} \le -\lambda \| h_{N+1} \|_{\dispa(m)}^2,
\end{gather*}
and the nonlinear terms are estimated using Proposition~\ref{prop:est_inhom_Q_g_f_f}:
\begin{align*}
	\la Q(h_N, h_{N+1}), h_{N+1} \ra_{\spp(m)} \lesssim  & \| h_{N+1} \|_{\dispa(m)}^2 \| h_N \|_{\spp(m)} \\
	& +  \| h_{N+1} \|_{\spp(m)} \| h_{N+1} \|_{\dispa(m)} \| h_{N} \|_{\dispa(m)},
\end{align*}
\begin{align*}
	\la Q(g_N, h_{N+1}), h_{N+1} \ra_{\spp(m)} \lesssim & \| h_{N+1} \|_{\dispa(m)}^2 \| g_N \|_{\spg_q} \\
	& + \| h_{N+1} \|_{\spp(m)} \| h_{N+1} \|_{\dispa(m)} \| g_{N} \|_{\disg_q},
\end{align*}
as well as the bound \eqref{eq:est_inhom_Q_f_g_h_poly_gauss}:
\begin{align*}
	\la Q(h_{N+1}, g_N), h_{N+1} \ra_{\spp(m)} \lesssim& \| h_{N+1} \|_{\dispa(m)}^2 \| g_N \|_{\spg_q} \\
	& +  \| h_{N+1} \|_{\spp(m)} \| h_{N+1} \|_{\dispa(m)} \| g_N \|_{\disg_q}.
\end{align*}
To sum up, we have the following energy estimate for $h_{N+1}$:
\begin{align*}
	\frac{1}{2} \frac{\mathrm d}{\mathrm d t} \| h_{N+1} \|^2_{\spp(m)} +\lambda \| h_{N+1} \|_{\dispa(m)}^2 \lesssim & \| h_{N+1} \|_{\dispa(m)}^2 \left( \| g_N \|_{\spg_q} + \| h_N \|_{\spp(m)} \right) \\
	& + \| h_{N+1} \|_{\spp(m)} \| h_{N+1} \|_{\dispa(m)} \left(\| h_{N} \|_{\dispa(m)} + \| g_{N} \|_{\disg_q} \right).
\end{align*}
Integrating in time and simplifying by $\| h_{N+1} \|_{\sppt(m)}$, we obtain using the bounds \eqref{eq:simple_inductive_bound}
\begin{align*}
\| h_{N+1} \|_{\sppt(m)} \lesssim \eps_0 \| h_{N+1} \|_{\sppt(m)} + \| f_0 \|_{\spp(m)},
\end{align*}
which, assuming $\eps_0>0$ small enough, simplifies as
\begin{equation}
	\label{eq:integrated_stability_poly}
	\| h_{N+1} \|_{\sppt(m)} \lesssim  \| f_0 \|_{\spp(m)} \le \eps_0.
\end{equation}
This concludes this step.

\subsubsection{Stability of $h_{N+1}$ in norm $\sppdt(m_0)$}\label{sec:stability_hN_Y}
As for the estimates with the weight $m$, we have for some $\lambda > 0$
\begin{align*}
&\frac{1}{2} \frac{\mathrm d}{\mathrm d t} \| h_{N+1} \|^2_{\spp(m_0)} +\lambda \| h_{N+1} \|_{\dispa(m_0)}^2 \\
&\qquad \lesssim  
\| h_{N+1} \|_{\dispa(m_0)}^2 \left( \| g_N \|_{\spg_q} + \| h_N \|_{\spp(m_0)} \right) \\
&\qquad\quad 
+ \| h_{N+1} \|_{\spp(m_0)} \| h_{N+1} \|_{\dispa(m_0)} \left(\| h_{N} \|_{\dispa(m_0)} + \| g_{N} \|_{\disg_q}\right),
\end{align*}
which, using the bounds \eqref{eq:simple_inductive_bound} and assuming $\eps_0>0$ small enough in order to absorb the first term in the right-hand side by the left-hand side, simplifies as
\begin{align*}
& \frac{1}{2} \frac{\mathrm d}{\mathrm d t} \| h_{N+1} \|^2_{\spp(m_0)} + \frac{\lambda}{2} \| h_{N+1} \|_{\dispa(m_0)}^2 \\ 
&\qquad \lesssim  \| h_{N+1} \|_{\spp(m_0)} \| h_{N+1} \|_{\dispa(m_0)} \left(\| h_{N} \|_{\dispa(m_0)} + \| g_{N} \|_{\disg_q}\right),
\end{align*}
Moreover, since we have $\|\la v \ra^{\gamma / 2} h\|_{\spp(m_0)} \le \| h \|_{\dispa(m_0)}$, the following interpolation inequality holds for any $R > 0$:
\begin{equation}\label{eq:interpolation_Rt}
	\la R \ra^{-|\gamma|} \| h \|_{\spp(m_0)}^2 \le \| h \|_{\dispa(m_0)}^2 + \la R \ra^{-|\gamma|-2(k - k_0)} \| h \|_{\spp(m)}^2,
\end{equation}
thus, taking $\la R \ra = \left(\frac{\lambda}{4 \theta}\right)^{1 / |\gamma|} (1+t)^{1 / \gamma|}$, we have for all $t \ge 0$
\begin{equation*}
	 \frac{4 \theta}{\lambda} (1+t)^{-1} \| h \|_{\spp(m_0)}^2 \le \| h \|_{\dispa(m_0)}^2 + \left(\frac{4 \theta}{\lambda}\right)^{1 + \frac{2(k - k_0)}{|\gamma|}} (1+t)^{ -1 - \frac{2(k - k_0)}{|\gamma|} } \| h \|_{\spp(m)}^2.
\end{equation*}
We now plug this control in the energy estimate:
\begin{align*}
	&\frac{1}{2} \frac{\mathrm d}{\mathrm d t} \| h_{N+1} \|^2_{\spp(m_0)} + \frac{\lambda}{4} \| h_{N+1} \|_{\dispa(m_0)}^2 + \theta (1+t)^{-1} \| h_{N+1} \|_{\dispa(m_0)}^2\\
	&\qquad \lesssim  (1+t)^{-1-\frac{2(k - k_0)}{|\gamma|} } \| h_{N+1} \|_{\spp(m)}^2 \\
	&\qquad\quad 
	+ \| h_{N+1} \|_{\spp(m_0)} \| h_{N+1} \|_{\dispa(m_0)} \left(\| h_{N} \|_{\dispa(m_0)} + \| g_{N} \|_{\disg_q}\right),
\end{align*}
and then multiply both sides by $(1+t)^{2\theta}$:
\begin{align*}
	&\frac{\mathrm d}{\mathrm d t} \Big\{ (1+t)^{2\theta} \| h_{N+1} \|^2_{\spp(m_0)} \Big\}  + \frac{\lambda}{2} (1+t)^{2\theta} \| h_{N+1} \|_{\dispa(m_0)}^2 \\
	&\qquad \lesssim  (1+t)^{2\theta - 1 - \frac{2(k - k_0)}{|\gamma|} } \| h_{N+1} \|_{\spp(m)}^2 \\
	&\qquad\quad
	+ (1+t)^{2\theta} \| h_{N+1} \|_{\spp(m_0)} \| h_{N+1} \|_{\dispa(m_0)} \left(\| h_{N} \|_{\dispa(m_0)} + \| g_{N} \|_{\disg_q}\right).
\end{align*}
Integrating in time and using \eqref{eq:simple_inductive_bound}, we get
\begin{align*}
	\| h_{N+1} \|_{\sppdt(m_0)}^2 \lesssim  \| h_{N+1} \|_{\sppt(m)}^2 + \eps_0 \| h_{N+1} \|_{\sppdt(m_0)}^2 + \| f_0 \|_{\spp(m_0)}^2,
\end{align*}
where we used the fact that $2\theta - 1 - \frac{2(k - k_0)}{|\gamma|} < -1$ so that $(1+t)^{2\theta - 1 - \frac{2(k - k_0)}{|\gamma|} } $ is integrable. Assuming $\eps_0>0$ small enough and plugging in \eqref{eq:integrated_stability_poly}, we finally get
\begin{equation}
	\label{eq:integrated_stability_poly_decay}
	\| h_{N+1} \|_{\sppdt(m_0)} \lesssim \| f_0 \|_{\spp(m_0)} \le \eps_0.
\end{equation}
This concludes this step.

\subsubsection{Stability of $g_{N+1}$ in norm $\spgt_q$}\label{sec:stability_gN_E}

We now turn to the second equation of \eqref{eq:systeme-hgN}. The weak coercivity estimate from Corollary~\ref{cor:coercivity_Lambda_gauss} gives us for some $\lambda > 0$
\begin{equation*}
	\dlla (\LLL-v\cdot \nabla_x ) g_{N+1}, g_{N+1} \drra_{\spg_q} \le - \lambda \| g_{N+1} \|_{\disg_q}^2.
\end{equation*}
The nonlinear term is estimated using \eqref{eq:est_Q_f_g_h_gauss} of Proposition~\ref{prop:est_inhom_Q_f_g_h_gaussian}:
\begin{equation*}
	\dlla Q(g_N, g_{N+1}), g_{N+1} \drra_{\spg_q} \lesssim \| g_{N} \|_{\spg_q} \| g_{N+1} \|_{\disg_q}^2 +  \| g_{N+1} \|_{\disg_q} \| g_{N+1} \|_{\spg_q} \| g_N \|_{\disg_q}.
\end{equation*}
The coupling term is estimated using \eqref{eq:est_regularized_twisted_naive}:
\begin{equation*}
	\dlla \AA h_N, g_{N+1} \drra_{\spg_q} \lesssim \| h_N \|_{\spp(m_0)} \| g_{N+1} \|_{\spg_q}.
\end{equation*}
Putting these bounds together et recalling that $\Nt \cdot \Nt_{\spg_q}$ and $\| \cdot \|_{\spg_q}$ are equivalent, we obtain the following energy estimate for $g_{N+1}$:
\begin{align*}
	\frac{1}{2} \frac{\mathrm d}{\mathrm d t} \Nt g_{N+1} \Nt_{\spg_q}^2 + \lambda \| g_{N+1} \|_{\disg_q}^2 \lesssim & \Nt g_{N} \Nt_{\spg_q} \| g_{N+1} \|_{\disg_q}^2 + \| g_{N+1} \|_{\disg_q} \Nt g_{N+1} \Nt_{\spg_q} \| g_{N} \|_{\disg_q} \\
	& +  \| h_N \|_{\spp(m_0)} \Nt g_{N+1} \Nt_{\spg_q},
\end{align*}
which, using \eqref{eq:simple_inductive_bound} and assuming $\eps_0>0$ small enough in order to absorb the first term in the right-hand side by the left-hand side, simplifies into
\begin{equation}\label{eq:g_E}
	\frac{1}{2} \frac{\mathrm d}{\mathrm d t} \Nt g_{N+1} \Nt_{\spg_q}^2 + \frac{\lambda}{2} \| g_{N+1} \|_{\disg_q}^2 \lesssim \| g_{N+1} \|_{\disg_q} \Nt g_{N+1} \Nt_{\spg_q} \| g_{N} \|_{\disg_q}
	 +  \| h_N \|_{\spp(m_0)} \Nt g_{N+1} \Nt_{\spg_q}.
\end{equation}
Integrating in time and using \eqref{eq:simple_inductive_bound} and simplifying by $\| g_{N+1} \|_{\spgt_q}$, we finally get
\begin{equation*}
\begin{aligned}
\| g_{N+1} \|_{\spgt_q}^2 
&\lesssim \| g_{N} \|_{\spgt_q} \| g_{N+1} \|_{\spgt_q}^2 + \| h_N \|_{\sppdt(m_0)} \| g_{N+1} \|_{\spgt_q} \\
&\lesssim \eps_0 \| g_{N+1} \|_{\spgt_q}^2 + \| f_0 \|_{\spp(m)} \| g_{N+1} \|_{\spgt_q},
\end{aligned}
\end{equation*}
which implies, assuming $\eps_0>0$ small enough, that
\begin{equation}\label{eq:integrated_stability_gaussian}
\| g_{N+1} \|_{\spgt_q}
\lesssim  \| f_0 \|_{\spp(m)} \le \eps_0.
\end{equation}
This concludes this step.

\medskip

We therefore deduce \eqref{eq:simple_inductive_bound} for $N+1$ by gathering estimates \eqref{eq:integrated_stability_poly}, \eqref{eq:integrated_stability_poly_decay} and \eqref{eq:integrated_stability_gaussian}, which completes the stability part of the proof.

\subsection{Convergence of the scheme}\label{sec:scheme_convergence}

Consider the successive differences of $(h_N)_{N=0}^\infty$ denoted by $d_{N+1} := h_{N+1} - h_N$, that of $(g_N)_{N=0}^\infty$ by $e_{N+1} := g_{N+1} - g_N$, and consider the equation satisfied by $d_{N+1}$
\begin{equation}\label{eq:dN}
\left\{
\begin{aligned}
	\partial_t d_{N+1} = (\BB-v\cdot \nabla_x) d_{N+1} & + Q(h_N, d_{N+1}) + Q(d_N, h_N) \\
	& + Q(g_N, d_{N+1}) + Q(e_N, h_N) \\
	& + Q(d_{N+1}, g_N) + Q(h_N, e_N),\\
	d_{N+1}(0, x, v) = 0,
\end{aligned}
\right.
\end{equation}
as well as the one satisfied by $e_{N+1}$:
\begin{equation}\label{eq:eN}
\left\{
\begin{aligned}	&\partial_t e_{N+1} = (\LLL-v\cdot \nabla_x) e_{N+1} + Q(g_N, e_{N+1}) + Q(e_N, g_N) + \AA d_{N},\\
	&e_{N+1}(0, x, v) = 0.
\end{aligned}
\right.
\end{equation}

In this subsection, we shall establish that for $\eps_0 >0$ small enough the following bound holds, for some $C_0>0$ and all $N \ge 0$:
\begin{equation}\label{eq:bound_convergence_N}
\|e_N\|_{\spgt_q} + \| d_N \|_{\sppt(\underline{m})} + \| d_N \|_{\sppdt(\underline{m_0})} \lesssim (C_0 \eps_0)^{N/2}.
\end{equation}

\subsubsection{Recursive estimates for $d_{N+1}$ in norm $\sppt(\underline{m})$}
Let $N \in \N$. We start by considering the equation \eqref{eq:dN}. The estimates of Proposition \ref{prop:est_inhom_Q_g_f_f} give the following control:
\begin{align*}
	\la Q(h_N, d_{N+1}), d_{N+1} \ra_{\spp(\underline{m} )} \lesssim & \| d_{N+1} \|_{\dispa(\underline{m})}^2 \| h_N \|_{\spp(\underline{m})} \\
	 & + \| h_N \|_{\dispa(\underline{m})} \| d_{N+1} \|_{\spp(\underline{m})} \| d_{N+1} \|_{\dispa(\underline{m})},
\end{align*}
and also
\begin{align*}
		\la Q(g_N, d_{N+1}), d_{N+1} \ra_{\spp(\underline{m})} \lesssim & \| d_{N+1} \|_{\dispa(\underline{m})}^2 \| g_N \|_{\spg_q} \\
	& + \| g_N \|_{\disg_q} \| d_{N+1} \|_{\spp(\underline{m})} \| d_{N+1} \|_{\dispa(\underline{m})}.
\end{align*}
Moreover, estimate \eqref{eq:est_inhom_Q_f_g_h_poly_gauss} gives
\begin{align*}
	\la Q(d_{N+1}, g_N), d_{N+1} \ra_{\spp(\underline{m})} \lesssim & \| d_{N+1} \|_{ \dispa(\underline{m}) }^2 \| g_N \|_{ \spg_q} \\
	& + \| d_{N+1} \|_{ \dispa(\underline{m}) }\| g_N \|_{ \disg_q } \| d_{N+1} \|_{\spp(\underline{m})},
\end{align*}
as well as
\begin{align*}
	\la Q(h_N, e_N), d_{N+1} \ra_{\spp(\underline{m})} \lesssim  & \| d_{N+1} \|_{ \dispa(\underline{m}) } \| h_N \|_{ \dispa(\underline{m}) } \| e_N \|_{ \spg_q} \\
& + \| d_{N+1} \|_{ \dispa(\underline{m}) } \| e_N \|_{ \disg_q } \| h_N \|_{\spp(\underline{m})} \\
& + \| d_{N+1} \|_{ \spp(\underline{m}) } \| e_N \|_{ \disg_q } \| h_N \|_{\dispa(\underline{m})}.
\end{align*}
Finally, estimates \eqref{eq:est_inhom_Q_g_f_h_poly_gauss} and \eqref{eq:est_inhom_Q_f_g_h_poly} give respectively the following bounds, which force to work in the larger space $\spp(\underline{m})$ instead of $\spp(m)$:
\begin{align*}
	\la Q(e_N, h_N), d_{N+1} \ra_{\spp(\underline{m})} \lesssim & \| d_{N+1} \|_{ \dispa(\underline{m}) }  \| e_N \|_{ \disg_q } \| h_N \|_{ \spp(\underline{m})} \\
	& +  \| d_{N+1} \|_{ \dispa(\underline{m}) } \| h_N \|_{ \dispa(m) } \| e_N \|_{\spg_q},
\end{align*}
and
\begin{align*}
	\la Q(d_N, h_N), d_{N+1} \ra_{\spp(\underline{m})} \lesssim & \| d_{N+1} \|_{\spp(\underline{m})} \| d_N \|_{ \dispa(\underline{m}) } \| h_N \|_{ \dispa(\underline{m})} \\
	& + \| d_{N+1} \|_{\dispa(\underline{m})} \| h_N \|_{ \dispa(m) } \| d_N \|_{\spp(\underline{m})}.
\end{align*}
As in the step of stability in Section~\ref{sec:stability_hN_X}, we put these bounds together and integrate the resulting energy estimate to obtain the following control:
$$
\begin{aligned}
\| d_{N+1} \|_{\sppt(\underline{m})}^2
&\lesssim \eps_0 \| d_{N+1} \|_{\sppt(\underline{m})}^2 
+ \eps_0 \| d_{N+1} \|_{\sppt(\underline{m})} \| e_N \|_{\spgt_q}
+ \eps_0 \| d_{N+1} \|_{\sppt(\underline{m})} \| d_N \|_{\sppt(\underline{m})}
\end{aligned}
$$
where we used the stability estimates \eqref{eq:simple_inductive_bound}. Assuming $\eps_0>0$ small enough, this simplifies as
\begin{equation}
	\label{eq:integrated_convergence_poly}
	\| d_{N+1} \|_{\sppt(\underline{m})} \lesssim \eps_0 \| e_N \|_{\spgt_q} + \eps_0 \| d_N \|_{\sppt(\underline{m})}.
\end{equation}

\subsubsection{Recursive estimate for $d_{N+1}$ in norm $\sppdt(\underline{m_0})$ }
Let $N \in \N$. Arguing as in the step of stability in Section~\ref{sec:stability_hN_Y}, we have
\begin{align*}
& \frac{\mathrm d}{\mathrm d t} \Big \{  (1+t)^{2\theta} \| d_{N+1} \|^2_{\spp(\underline{m_0})} \Big\}  + \frac{\lambda}{2} (1+t)^{2\theta} \| d_{N+1} \|_{\dispa(\underline{m_0})}^2 \\
&\qquad \lesssim  
(1+t)^{2\theta} \| d_{N+1} \|_{\spp(\underline{m_0})} \| d_{N+1} \|_{\dispa(\underline{m_0})} \left( \| h_N \|_{\dispa(m_0)}  + \| g_N \|_{\disg_q} \right) \\
&\qquad\quad 
+ (1+t)^{2\theta} \| d_{N+1} \|_{ \dispa(\underline{m_0}) } \| h_N \|_{ \dispa(m_0) } \left( \| e_N \|_{\spg_q} + \| d_N \|_{\spp(\underline{m_0})} \right) \\
&\qquad\quad 
+ (1+t)^{2\theta} \| d_{N+1} \|_{ \dispa(\underline{m_0}) } \| h_N \|_{\spp(\underline{m_0})} \left( \| e_N \|_{ \disg_q } + \| d_N \|_{ \dispa(\underline{m_0}) } \right) \\
&\qquad\quad 
+ (1+t)^{2\theta} \| d_{N+1} \|_{ \spp(\underline{m_0}) } \| h_N \|_{\dispa(\underline{m_0})}  \| e_N \|_{ \disg_q }  \\
&\qquad\quad + (1+t)^{2\theta - 1 - \frac{2(k - k_0)}{|\gamma|} } \| d_{N+1} \|_{\spp(\underline{m})}^2.
\end{align*}
After integrating and using the bounds \eqref{eq:simple_inductive_bound} from the stability estimate, we are left with
\begin{align*}
\| d_{N+1} \|_{\sppdt(\underline{m_0})}^2 
&\lesssim  \eps_0 \| d_{N+1} \|_{\sppdt( \underline{m_0} )}^2 
+ \eps_0 \| d_{N+1} \|_{\sppdt(\underline{m_0})} \left( \| e_N \|_{\spgt_q} + \| d_N \|_{\sppdt( \underline{m_0} )} \right) 
+ \| d_{N+1} \|^2_{\sppt( \underline{m} )}  \\
& \lesssim \eps_0 \| d_{N+1} \|_{\sppdt( \underline{m_0} )}^2 + \eps_0 \left( \| e_N \|_{\spgt_q}^2 + \| d_N \|_{\sppdt( \underline{m_0} )}^2 \right) + \| d_{N+1} \|^2_{\sppt( \underline{m} )},
\end{align*}
where we used Young's inequality in the last line. Assuming $\eps_0>0$ small enough and plugging~\eqref{eq:integrated_convergence_poly} in, this bound simplifies as
\begin{equation}
	\label{eq:integrated_convergence_poly_decay}
	\| d_{N+1} \|_{\sppdt(\underline{m_0})} \lesssim \eps_0 \| e_N \|_{\spgt_q} + \eps_0 \| d_N \|_{\sppt( \underline{m_0} )} + \eps_0 \| d_N \|_{\sppdt(\underline{m_0})}.
\end{equation}

\subsubsection{Recursive estimates for $e_{N}$ in norm $\spgt_q$}
Let $N \in \N$. We now consider the equation \eqref{eq:eN}. Using \eqref{eq:est_Q_f_g_h_gauss} we have
\begin{equation*}
	\dlla Q(g_N, e_{N+1}), e_{N+1} \drra_{\spg_q} 
\lesssim 
\| g_{N} \|_{\spg_q} \| e_{N+1} \|_{\disg_q}^2 +  \| e_{N+1} \|_{\disg_q} \| e_{N+1} \|_{\spg_q} \| g_N \|_{\disg_q},
\end{equation*}
as well as
\begin{equation*}
	\dlla Q(e_N, g_{N}), e_{N+1} \drra_{\spg_q} 
	\lesssim \| e_{N+1} \|_{\disg_q} \left( \| e_N \|_{\disg_q}  \| g_N \|_{\spg_q} + \| e_N \|_{\spg_q}  \| g_N \|_{\disg_q}  \right),
\end{equation*}
and using~\eqref{eq:est_regularized_twisted_naive} we get
\begin{equation*}
	\dlla \AA d_N, e_{N+1} \drra_{\spg_q} \lesssim \| d_N \|_{\spp(\underline{m_0})} \| e_{N+1} \|_{\spg_q}.
\end{equation*}

Arguing as in Section~\ref{sec:stability_gN_E}, we gather these estimates and integrate the resulting energy estimate in time to obtain the following bound:
\begin{equation*}
	\| e_{N+1} \|_{\spgt_q} \lesssim \| g_{N} \|_{\spgt_q} \| e_{N+1} \|_{\spgt_q} + \| g_{N} \|_{\spgt_q} \| e_N \|_{\spgt_q} + \| d_{N} \|_{\sppdt(\underline{m_0})},
\end{equation*}
which, using the stability estimates \eqref{eq:simple_inductive_bound} and assuming $\eps_0>0$ small enough, simplifies as
\begin{equation}
	\label{eq:integrated_convergence_gauss}
	\| e_{N+1} \|_{\spgt_q} \lesssim \eps_0 \| e_N \|_{\spgt_q} + \| d_{N} \|_{\sppdt(\underline{m_0})}.
\end{equation}

\subsubsection{Proof of convergence}

We first prove \eqref{eq:bound_convergence_N} by using previous estimates. It is clearly true for $N=0$, so we assume that~\eqref{eq:bound_convergence_N} holds for all integers up to some $N \ge 0$, and we shall deduce it for $N+1$.
Thanks to estimates \eqref{eq:integrated_convergence_poly}, \eqref{eq:integrated_convergence_poly_decay} and \eqref{eq:integrated_convergence_gauss} we have obtained, for all $N \ge 0$,
\begin{equation*}
	\begin{aligned}
		\| d_{N+1} \|_{\sppt(\underline{m})} &\lesssim \eps_0 \left( \| d_{N} \|_{\sppt(\underline{m})}  +  \| e_{N} \|_{\spgt_q} \right) \\
		\| d_{N+1} \|_{\sppdt(\underline{m_0})} &\lesssim \eps_0 \left(  \| d_N \|_{\sppt( \underline{m} )} +  \| d_N \|_{\sppdt(\underline{m_0})} +  \| e_N \|_{\spgt_q} \right) \\
		\| e_{N+1} \|_{\spgt_q} &\lesssim \eps_0 \| e_{N} \|_{\spgt_q} + \| d_{N} \|_{\sppdt(\underline{m_0})}  .
	\end{aligned}
\end{equation*}
This implies that 
$$
\| d_{N+1} \|_{\sppt(\underline{m})} + \| d_{N+1} \|_{\sppdt(\underline{m_0})} + \| e_{N+1} \|_{\spgt_q} \lesssim \eps_0 ( \| d_{N-1} \|_{\sppt(\underline{m})} + \| d_{N-1} \|_{\sppdt(\underline{m_0})} + \| e_{N-1} \|_{\spgt_q} )
$$
and thus, using \eqref{eq:bound_convergence_N} for $N-1$, we deduce 
$$
\| d_{N+1} \|_{\sppt(\underline{m})} + \| d_{N+1} \|_{\sppdt(\underline{m_0})} + \| e_{N+1} \|_{\spgt_q} 
\lesssim \eps_0 ( C_0 \eps_0)^{\frac{N-1}{2}} 
\lesssim (C_0 \eps_0)^{\frac{N+1}{2}} ,
$$
which proves \eqref{eq:bound_convergence_N}.

Therefore, assuming $\eps_0>0$ small enough, the sequence $(h_N, g_N)_{N \ge 0}$ is a Cauchy sequence in $\sppt(\underline{m}) \times \spgt_q$ and thus converges to some limit $(h,g)$ in $\sppt(\underline{m}) \times \spgt_q$. In virtue of the stability estimates, the limit thus satisfies the bounds
\begin{equation}\label{eq:bound_h_g}
\begin{aligned}
\| h \|_{\sppt(m)}=	\sup_{t \ge 0} \| h(t) \|_{\spp(m)}^2 + \int_0^\infty \| h(t) \|_{\dispa(m)}^2 \d t &\lesssim \| f_0 \|_{\spp(m)}^2,\\
\| h \|_{\sppdt(m_0)}=	\sup_{t \ge 0} \, (1+ t)^{2\theta} \| h(t) \|_{\spp\left( m_0 \right)}^2 + \int_0^\infty (1+t)^{2\theta} \| h(t) \|_{\dispa\left( m_0 \right)}^2 \d t &\lesssim \| f_0 \|_{\spp(m)}^2,\\
\| g \|_{\spgt_q}=	\sup_{t \ge 0} \| g(t) \|_{\spg_q}^2 + \int_0^\infty \| g(t) \|_{\disg_q}^2 \d t &\lesssim \| f_0 \|_{\spp(m)}^2.
\end{aligned}
\end{equation}
The solution thus constructed to the original perturbation equation \eqref{eq:main-pert} is given by letting $f := h + g \in L^\infty( \R_+ ; \spp(m)) \cap L^2(\R_+ ; \disp(m))$, which thus satisfies
\begin{gather*}
	\sup_{t \ge 0} \| f(t) \|_{\spp(m)}^2 + \int_0^\infty \left( \| f^\perp(t) \|_{\dispa(m)}^2 + \| \nabla_x \pi f(t) \|_{H^2_x L^2_v }^2 \right) \d t \lesssim \| f_0 \|_{\spp(m)}^2.
\end{gather*}
This completes the existence part of Theorem~\ref{theo:non-cutoff} together with the energy estimate~\eqref{eq:theo:energy_estimate}.

\subsection{Uniqueness of the solution}\label{sec:uniqueness}
Consider two solutions $f , \widetilde{f} \in L^\infty( \R_+ ; \spp(m)) \cap L^2(\R_+ ; \disp(m))$ to \eqref{eq:main-pert} with the same initial condition $f_0$ and verifying
\begin{align*}
\sup_{t \ge 0} \| f(t) \|_{\spp(m)}^2 + \int_0^\infty  \| f(t) \|_{\disp(m)}^2 \d t 
&\lesssim \| f_0 \|_{\spp(m)}^2 \le \eps_0^2 , \\
\sup_{t \ge 0} \| \widetilde{f}(t) \|_{\spp(m)}^2 + \int_0^\infty  \| \widetilde{f}(t) \|_{\disp(m)}^2  \d t &\lesssim \| f_0 \|_{\spp(m)}^2 \le \eps_0^2,
\end{align*}
with $\eps_0>0$ small enough. Denote $d := f - \widetilde{f}$ the difference of these solutions, which satisfies
\begin{equation*}
	\partial_t d = (\BB-v\cdot \nabla_x) d + \AA d + Q(d, f) + Q(\widetilde{f}, d).
\end{equation*}
Arguing as in the previous steps, one gets for some $\lambda > 0$ the control
\begin{align*}
	\frac{1}{2} \frac{\mathrm d}{\mathrm d t} \| d \|^2_{\spp(\underline{m})} + \lambda \| d \|_{\dispa(\underline{m})}^2 \lesssim & \| d \|^2_{\spp(\underline{m})} + \| d \|^2_{\dispa(\underline{m})} \| f \|_{\spp(\underline{m})} + \| d \|_{\dispa(\underline{m})} \| d \|_{\spp(\underline{m})} \| f \|_{\disp(m)}\\
	& + \| d \|_{\dispa(\underline{m})}^2 \| \widetilde{f} \|_{\spp(\underline{m})} + \| d \|_{\dispa(\underline{m})} \| d \|_{\spp(\underline{m})} \| \widetilde{f} \|_{\dispa(\underline{m})},
\end{align*}
which, once integrated from $t=0$ to $t=T < \infty$, gives (with obvious notation)
\begin{align*}
	\| d \|_{\sppt(\underline{m} ; T)}^2 \lesssim \, \Big( T & + \| f \|_{L^\infty\left( [0, T] ; \spp(m) \right)} + \| f \|_{L^2\left( [0, T] ; \dispa(m) \right)} \\
	& + \| \widetilde{f} \|_{L^\infty\left( [0, T] ; \spp(m) \right)} + \| \widetilde{f} \|_{L^2\left( [0, T] ; \dispa(m) \right)} \Big) \| d \|_{\sppt(\underline{m};T)}^2.
\end{align*}
Observing that 
$$
\| \phi \|_{\dispa(m)} 
\lesssim \| \pi \phi \|_{H^3_x L^2_v} + \| \phi^\perp \|_{\dispa(m)} 
\lesssim \| \pi \phi \|_{L^2_x L^2_v} + \| \phi \|_{\disp(m)} 
\lesssim \| \phi \|_{\spp(m)} + \| \phi \|_{\disp(m)} 
$$
we have 
$$
\| f \|_{L^2( [0, T] ; \dispa(m) )} 
\lesssim \sqrt{T} \| f \|_{L^\infty( [0, T] ; \spp(m) )}  + \| f \|_{L^2( [0, T] ; \disp(m) )}
\lesssim \sqrt{T} \eps_0 + \eps_0
$$
and similarly for $\widetilde f$. Using the uniform bounds on $f$ and $\widetilde{f}$, this becomes
\begin{align*}
	\| d \|_{\sppt(\underline{m} ; T)}^2 \lesssim \, \Big( T + \eps_0 + \sqrt{T} \eps_0 \Big) \| d \|_{\sppt(\underline{m};T)}^2.
\end{align*}
Assuming $T > 0$ small enough and $\eps_0>0$ small enough, we have (for instance)
\begin{align*}
	\| d \|_{\sppt(\underline{m} ; T)}^2 \le \frac{1}{2} \| d \|_{\sppt(\underline{m};T)}^2,
\end{align*}
which means that $d=0$, or equivalently $f=\widetilde{f}$, on interval $[0, T]$. By continuing this argument, we deduce that $f_0$ gives rise to a unique (global) solution, namely, $f$. This concludes the proof of uniquenees in Theorem~\ref{theo:non-cutoff}.

\subsection{Decay estimate}\label{sec:decay}

This subsection is devoted to the proof of the decay estimate~\eqref{eq:theo:decay_estimate} in Theorem~\ref{theo:non-cutoff}. 
We therefore consider the unique solution $f=h+g$ of~\eqref{eq:main-pert} constructed in Sections~\ref{sec:scheme_stability}, \ref{sec:scheme_convergence} and \ref{sec:uniqueness} above, where $(h,g) \in \sppt(m) \times \spgt_q$ is a solution to the system~\eqref{eq:systeme-hg} and satisfies the estimates~\eqref{eq:bound_h_g}, where we fix $q> \vartheta |\gamma|$ and recall that $0 < \vartheta < \frac{3}{2} (\frac{1}{2} - \frac{1}{p})$ with $p \in (2,\infty]$. We shall obtain in the sequel the a priori estimates that implies the decay estimate \eqref{eq:theo:decay_estimate}, recalling that we suppose
$$
\| f_0 \|_{\spp(m)} + \| \widehat f_0 \|_{L^p_\xi L^2_v (\la v \ra^{-8s} m)} \le \eps_0
$$
small enough.

We observe that we have already obtained a decay estimate for the polynomial part of the solution $h$ in \eqref{eq:bound_h_g}. 
It therefore remains to obtain a decay estimate for $g$ in the space $\spgt_0$, which will be done by obtaining estimates in the Sobolev-type space $\spgt_{0,\star}$ defined by 
\begin{align}\label{eq:E_star}
	\| g \|_{\spgt_{0,\star}}^2 &:= \sup_{t \ge 0} (1+t)^{2\vartheta} \Nt g(t) \Nt_{\spg_0}^2 + \int_0^\infty (1+s)^{2\vartheta} \| g(t) \|_{\disg_0}^2 \d t.
\end{align}
This will require us to obtain uniform estimates in time of $h$ and $g$ in some Fourier-based spaces (defined below in \eqref{eq:Fp}, \eqref{eq:Fpstar} and \eqref{eq:Gp}).

Taking the Fourier transform in space of the system~\eqref{eq:systeme-hg} gives, for all $\xi \in \R^3$,
\begin{equation}\label{eq:systeme-hgN-fourier}
	\begin{cases}
		\begin{aligned}
			\partial_t \widehat h(\xi) = (\BB- i v\cdot \xi )  \widehat h(\xi)  + \widehat Q( h,  h)(\xi) 
			 + \widehat Q(  g,  h)(\xi) + \widehat Q( h,  g) (\xi),
		\end{aligned}
		\\
		\partial_t \widehat g(\xi) = (\LLL -i v\cdot \xi) \widehat g(\xi) + \widehat Q(g,g) (\xi) + \AA \widehat h (\xi),\\
		\widehat h(0, \xi,v) = \widehat f_0(\xi,v), \quad \widehat g(0, \xi, v) = 0,
	\end{cases}
\end{equation}
where we recall
$$
\widehat{Q}(f, g)(\xi) = \int_{\R^3} Q\left( \widehat{f}(\xi - \eta) , \widehat{g}(\eta) \right) \d \eta.
$$
Recalling that the weight functions $m = \la v \ra^k$ and $m_0= \la v \ra^{k_0}$ are defined in \eqref{eq:k} and \eqref{eq:k0_theta}, respectively, we define the weights $\tilde{m} = \la v \ra^{-8s} m $ and $\tilde{m}_0 = \la v \ra^{-8s} m_0$. Observe that we can apply Propositions~\ref{prop:w_diss_B} and \ref{prop:nonlinear_Nm_fgh} with the weight functions $\tilde m_0$ and $\tilde m$. We also remark that the $\la v \ra^{-8s}$ appearing in the definition of $\tilde m$ comes from our definition of the space $\spp(m)$ and the weight loss in the nonlinear estimate of Proposition~\ref{prop:nonlinear_Nm_fgh}, which allows us to control some Fourier-based norm with the energy estimate (see \eqref{eq:h_L1xi_Xm}).

We then define the Fourier-based functional spaces with polynomial weights $\mathscr{F}^p( \tilde{m})$ and $\mathscr{F}^p_{\star}( \tilde{m}_0)$ as the spaces associated to the norms, respectively:
\begin{align}
	\| h \|_{\mathscr{F}^p( \tilde{m})} &:= \| \widehat h \|_{L^p_\xi L^\infty_t L^2_v(\tilde{m})} + \| \widehat h \|_{L^p_\xi L^2_t H^{s,*}_v(\tilde{m})} , \label{eq:Fp}\\
	\| h \|_{\mathscr{F}^p_{\star}( \tilde{m}_0)} &:= \| (1+t)^\theta \widehat h \|_{L^p_\xi L^\infty_t L^2_v(\tilde{m}_0)} + \| (1+t)^\theta \widehat h \|_{L^p_\xi L^2_t H^{s,*}_v(\tilde{m}_0)}, \label{eq:Fpstar}
\end{align}
where the decay parameter $\theta >0$ is defined in \eqref{eq:k0_theta}. Moreover we also define the exponentially weighted space $\mathscr{G}^p$ as the space associated to the norm:
\begin{align}\label{eq:Gp}
\| g \|_{\mathscr{G}^p} &:= \| \widehat g \|_{L^p_\xi L^\infty_t L^2_v(\mu^{-1/2})} + \| \widehat g \|_{L^p_\xi L^2_t H^{s,**}_v ( \mu^{-1/2})}. 
\end{align}

For later use, we already observe that
\begin{equation}\label{eq:h_L1xi_Xm}
\| \widehat h \|_{L^1_\xi L^2_t H^{s,*}_v( \la v \ra^{2s} \tilde{m})} \lesssim \| \la \xi \ra^3 \widehat h \|_{L^2_\xi L^2_t H^{s,*}_v( \la v \ra^{2s} \tilde{m})}  \approx \| h \|_{L^2_t H^3_x H^{s,*}_v( \la v \ra^{2s} \tilde{m})} \lesssim \| h \|_{\sppt(m)},
\end{equation}
and similarly
\begin{equation}\label{eq:g_L1xi_E}
\| \widehat g \|_{L^1_\xi L^2_t H^{s,**}_v(\mu^{-1/2})} 
\lesssim  \| \la \xi \ra^3 \widehat g \|_{L^2_\xi L^2_t H^{s,**}_v(\mu^{-1/2})}  
\approx \|  g \|_{L^2_t H^3_x H^{s,**}_v(\mu^{-1/2})} \lesssim \| g \|_{\spgt_q},
\end{equation}
with $\| h \|_{\sppt(m)}$ and $\| g \|_{\spgt_q}$ being controlled thanks to \eqref{eq:bound_h_g}.

\medskip

We split the proof of~\eqref{eq:theo:decay_estimate} into four steps.

\medskip\noindent
\textit{Step 1: Estimate of $h$ in norm $\mathscr{F}^p(\tilde{m})$.}
We start with the first equation of \eqref{eq:systeme-hgN}. The weak coercivity estimate on $\BB-v\cdot \nabla_x$ from Proposition~\ref{prop:w_diss_B} gives for some $\lambda > 0$
\begin{gather*}
\Re	\lla (\BB-iv\cdot \xi) \widehat h (\xi), \widehat h (\xi) \rra_{L^2_v(\tilde{m})} \le -\lambda \| \widehat h(\xi) \|_{H^{s,*}_v(\tilde{m})}^2,
\end{gather*}
thus
\begin{align*}
&\frac{1}{2} \frac{\mathrm d}{\mathrm d t} \|\widehat h (\xi) \|^2_{L^2_v(\tilde{m})} +\lambda \| \widehat h(\xi) \|_{H^{s,*}_v(\tilde{m})}^2 \\
&\qquad\lesssim  \lla \widehat Q ( h , h) (\xi), \widehat h (\xi) \rra_{L^2_v(\tilde{m})} 
+ \lla \widehat Q ( g , h) (\xi), \widehat h (\xi) \rra_{L^2_v(\tilde{m})} 
+ \lla \widehat Q ( h , g) (\xi), \widehat h (\xi) \rra_{L^2_v(\tilde{m})}.\end{align*}
Integrating in time then taking the supremum in time in both sides gives
\begin{align*}
&\|\widehat h (\xi) \|_{L^\infty_t L^2_v(\tilde{m})}^2 + \| \widehat h(\xi) \|_{L^2_t H^{s,*}_v(\tilde{m})}^2 \\
&\qquad\lesssim  
\NN_{\tilde{m}} [h , h , h] (\xi)^2 
+ \NN_{\tilde{m}} [g , h , h] (\xi)^2 
+ \NN_{\tilde{m}} [h , g , h] (\xi)^2 
+ \|\widehat f_0 (\xi) \|_{L^2_v(\tilde{m})}^2
\end{align*}
and hence
\begin{equation}\label{eq:hN+1xi_X1m}
\begin{aligned}
&\|\widehat h (\xi) \|_{L^\infty_t L^2_v(\tilde{m})} + \| \widehat h(\xi) \|_{L^2_t H^{s,*}_v(\tilde{m})} \\
&\qquad\lesssim  \NN_{\tilde{m}} [h , h , h] (\xi) + \NN_{\tilde{m}} [g , h , h] (\xi) 
+ \NN_{\tilde{m}} [h , g , h] (\xi) + \|\widehat f_0 (\xi) \|_{L^2_v(\tilde{m})}.
\end{aligned}
\end{equation}

Observe that thanks to Proposition~\ref{prop:nonlinear_Nm_fgh} we have
$$
\begin{aligned}
\| \NN_{\tilde{m}} [h , h , h] \|_{L^p_\xi} 
\lesssim \| \la v \ra^{2s} \widehat h \|_{L^1_\xi L^2_t H^{s,*}_v(\tilde{m})}^{1/2} \| h \|_{\mathscr{F}^p(\tilde{m})}, 
\end{aligned}
$$
and
$$
\begin{aligned}
&\| \NN_{\tilde{m}} [g , h , h] \|_{L^p_\xi} + \| \NN_{\tilde{m}} [h , g , h] \|_{L^p_\xi} \\
&\quad
\lesssim  \| \widehat g \|_{L^1_\xi L^2_t H^{s,**}_v(\mu^{-1/2})}^{1/2}   \| h \|_{\mathscr{F}^p(\tilde{m})}
+ \|  g \|_{\mathscr{G}^p}^{1/2}  \| \la v \ra^{2s} \widehat h \|_{L^1_\xi L^2_t H^{s,*}_v(\tilde{m})}^{1/2} \| h \|_{\mathscr{F}^p(\tilde{m})}^{1/2}.
\end{aligned}
$$
Thanks to \eqref{eq:h_L1xi_Xm}--\eqref{eq:g_L1xi_E} and the bounds on $h$ and $g$ from estimate \eqref{eq:bound_h_g}, we deduce that 
$$
\begin{aligned}
&\| \NN_{\tilde{m}} [h , h , h] \|_{L^p_\xi} + \| \NN_{\tilde{m}} [g , h , h] \|_{L^p_\xi}+\| \NN_{\tilde{m}} [h , g , h] \|_{L^p_\xi} \\
&\quad\lesssim
\| f_0 \|_{\spp(m)}^{1/2} \| h \|_{\mathscr{F}^p(\tilde{m})}
+ \| f_0 \|_{\spp(m)}^{1/2} \|  g \|_{\mathscr{G}^p}^{1/2} \| h \|_{\mathscr{F}^p(\tilde{m})}^{1/2},
\end{aligned}
$$
therefore taking the $L^p_\xi$ norm of \eqref{eq:hN+1xi_X1m} yields
\begin{align*}
\| h \|_{\mathscr{F}^p(\tilde{m})} 
&\lesssim   
\| f_0 \|_{\spp(m)}^{1/2} \| h \|_{\mathscr{F}^p(\tilde{m})}
+ \| f_0 \|_{\spp(m)}^{1/2} \|  g \|_{\mathscr{G}^p}^{1/2} \| h \|_{\mathscr{F}^p(\tilde{m})}^{1/2}
+ \| \widehat f_0 \|_{L^p_\xi L^2_v(\tilde{m})}.
\end{align*}
Recalling that $\| f_0 \|_{\spp(m)} + \| \widehat f_0 \|_{L^p_\xi L^2_v (\la v \ra^{-8s} m)} \le \eps_0$ and taking $\eps_0>0$ small enough we deduce
\begin{equation}\label{eq:h_Fp}
\| h \|_{\mathscr{F}^p(\tilde{m})} 
\lesssim   
\| f_0 \|_{\spp(m)} \|  g \|_{\mathscr{G}^p} + \| \widehat f_0 \|_{L^p_\xi L^2_v(\tilde{m})}.
\end{equation}

\medskip\noindent
\textit{Step 2: Estimate of $h$ in norm $\mathscr{F}^p_\star(\tilde{m}_0)$.}\
As for the estimates with the weight $\tilde m$ in Step 1, thanks to Proposition~\ref{prop:w_diss_B} we have for some $\lambda > 0$
\begin{align*}
&\frac{1}{2} \frac{\mathrm d}{\mathrm d t} \|\widehat h (\xi) \|^2_{L^2_v(\tilde m_0)} +\lambda \| \widehat h(\xi) \|_{H^{s,*}_v(\tilde m_0)}^2 \\
&\qquad\lesssim  \lla \widehat Q ( h , h) (\xi), \widehat h (\xi) \rra_{L^2_v(\tilde m_0)} 
+ \lla \widehat Q ( g , h) (\xi), \widehat h (\xi) \rra_{L^2_v(\tilde m_0)} \\
&\qquad\quad
+ \lla \widehat Q ( h , g) (\xi), \widehat h (\xi) \rra_{L^2_v(\tilde m_0)}.
\end{align*}
Moreover, since we have $\|\la v \ra^{\gamma / 2} h\|_{L^2_v(\tilde m_0)} \le \| h \|_{H^{s,*}_v(\tilde m_0)}$, using a similar interpolation inequality as in \eqref{eq:interpolation_Rt} we have for all $ t \ge 0$
\begin{equation}\label{eq:interpolation_Rt_bis}
	 \frac{2\theta}{\lambda} (1+t)^{-1} \| h \|_{L^2_v(\tilde m_0)}^2 \le \| h \|_{H^{s,*}_v(\tilde m_0)}^2 + \left(\frac{2\theta}{\lambda}\right)^{1 + \frac{2(k - k_0)}{|\gamma|} } (1+t)^{ -1 - \frac{2(k - k_0)}{|\gamma|} } \| h \|_{L^2_v(\tilde m)}^2.
\end{equation}
We now plug this control in the previous energy estimate:
\begin{align*}
&\frac{1}{2} \frac{\mathrm d}{\mathrm d t} \|\widehat h (\xi) \|^2_{L^2_v(\tilde m_0)} +\frac{\lambda}{2} \| \widehat h(\xi) \|_{H^{s,*}_v(\tilde m_0)}^2 + \theta(1+t)^{-1} \| \widehat h(\xi) \|_{L^2_v(\tilde m_0)}^2\\
&\qquad\lesssim  \lla \widehat Q ( h , h) (\xi), \widehat h (\xi) \rra_{L^2_v(\tilde m_0)} 
+ \lla \widehat Q ( g , h) (\xi), \widehat h (\xi) \rra_{L^2_v(m_0)} \\
&\qquad\quad
+ \lla \widehat Q ( h , g) (\xi), \widehat h (\xi) \rra_{L^2_v(\tilde m_0)}
+(1+t)^{- 1 - \frac{2(k - k_0)}{|\gamma|} } \| h \|_{L^2_v(\tilde m)}^2 ,
\end{align*}
and then multiply both sides by $(1+t)^{2\theta}$:
\begin{align*}
&\frac{1}{2} \frac{\mathrm d}{\mathrm d t} \left\{ (1+t)^{2\theta} \|\widehat h (\xi) \|^2_{L^2_v(\tilde m_0)} \right\} +\frac{\lambda}{2} (1+t)^{2\theta} \| \widehat h(\xi) \|_{H^{s,*}_v(\tilde m_0)}^2 \\
&\qquad\lesssim  (1+t)^{2\theta} \lla \widehat Q ( h , h) (\xi), \widehat h (\xi) \rra_{L^2_v(\tilde m_0)} 
+ (1+t)^{2\theta}\lla \widehat Q ( g, h) (\xi), \widehat h (\xi) \rra_{L^2_v(\tilde m_0)} \\
&\qquad\quad
+ (1+t)^{2\theta}\lla \widehat Q ( h , g) (\xi), \widehat h (\xi) \rra_{L^2_v(\tilde m_0)}
+(1+t)^{2\theta - 1 - \frac{2(k - k_0)}{|\gamma|} } \|\widehat h (\xi) \|^2_{L^2_v(\tilde m)} .
\end{align*}

Integrating in time then taking the supremum in time in both sides gives
\begin{align*}
&\|(1+t)^\theta \widehat h (\xi) \|_{L^\infty_t L^2_v(\tilde m_0)}^2 + \| (1+t)^\theta \widehat h(\xi) \|_{L^2_t H^{s,*}_v(\tilde m_0)}^2 \\
&\qquad\lesssim  
\NN_{\tilde m_0} [h , (1+t)^\theta h , (1+t)^\theta h] (\xi)^2 
+ \NN_{\tilde m_0} [g , (1+t)^\theta  h, (1+t)^\theta h] (\xi)^2 \\
&\qquad\quad 
+ \NN_{\tilde m_0} [(1+t)^\theta h , g , (1+t)^\theta h] (\xi)^2  
+ \|\widehat h (\xi) \|^2_{L^\infty_t L^2_v(m)} 
+ \|\widehat f_0 (\xi) \|_{L^2_v(m_0)}^2,
\end{align*}
where for the fourth term in the right-hand side we have used the fact that $\theta<(k-k_0)/|\gamma|$ so that $(1+t)^{2\theta - 1 - \frac{2(k - k_0)}{|\gamma|} }$ is integrable. Hence we deduce
\begin{equation}\label{eq:hN+1xi_X1starm0}
\begin{aligned}
&\|(1+t)^\theta \widehat h (\xi) \|_{L^\infty_t L^2_v(\tilde m_0)} + \| (1+t)^\theta \widehat h(\xi) \|_{L^2_t H^{s,*}_v(\tilde m_0)} \\
&\qquad\lesssim  
\NN_{\tilde m_0} [h , (1+t)^\theta h, (1+t)^\theta h] (\xi)
+ \NN_{\tilde m_0} [g, (1+t)^\theta h , (1+t)^\theta h] (\xi) \\
&\qquad\quad 
+ \NN_{\tilde m_0} [(1+t)^\theta h , g , (1+t)^\theta h] (\xi)  
+ \|\widehat h (\xi) \|_{L^\infty_t L^2_v(\tilde m)} 
+ \|\widehat f_0 (\xi) \|_{L^2_v(\tilde m_0)}.
\end{aligned}
\end{equation}

Using Propositions~\ref{prop:nonlinear_Nm_fgh} we have
$$
\begin{aligned}
\| \NN_{\tilde{m}_0} [h , (1+t)^\theta h , (1+t)^\theta h] \|_{L^p_\xi} 
\lesssim \| \la v \ra^{2s} (1+t)^\theta\widehat h \|_{L^1_\xi L^2_t H^{s,*}_v(\tilde{m}_0)}^{1/2} \| h \|_{\mathscr{F}^p(\tilde{m})}^{1/2} \| h \|_{\mathscr{F}^p_\star(\tilde{m}_0)}^{1/2}, 
\end{aligned}
$$
and
$$
\begin{aligned}
&\| \NN_{\tilde{m}_0} [g ,  (1+t)^\theta h ,  (1+t)^\theta h] \|_{L^p_\xi} 
+ \| \NN_{\tilde{m}_0} [ (1+t)^\theta h , g , (1+t)^\theta  h] \|_{L^p_\xi}\\
&\quad
\lesssim  \| \widehat g \|_{L^1_\xi L^2_t H^{s,**}_v(\mu^{-1/2})}^{1/2}   \| h \|_{\mathscr{F}^p_\star(\tilde{m}_0)}
+ \|  g \|_{\mathscr{G}^p}^{1/2}  \| \la v \ra^{2s} (1+t)^\theta \widehat h \|_{L^1_\xi L^2_t H^{s,*}_v(\tilde{m}_0)}^{1/2} \| h \|_{\mathscr{F}^p_\star(\tilde{m}_0)}^{1/2}.
\end{aligned}
$$
Arguing as in previous step, using the bounds on $h$ and $g$ from estimate \eqref{eq:bound_h_g} together with \eqref{eq:h_L1xi_Xm}--\eqref{eq:g_L1xi_E}, we deduce that 
$$
\begin{aligned}
&\| \NN_{\tilde{m}_0} [h , (1+t)^\theta h , (1+t)^\theta h] \|_{L^p_\xi} 
+ \| \NN_{\tilde{m}_0} [g ,  (1+t)^\theta h ,  (1+t)^\theta h] \|_{L^p_\xi} 
+ \| \NN_{\tilde{m}_0} [ (1+t)^\theta h , g , (1+t)^\theta  h] \|_{L^p_\xi} \\
&\quad\lesssim
\| f_0 \|_{\spp(m)}^{1/2} \| h \|_{\mathscr{F}^p(\tilde{m})}^{1/2} \| h \|_{\mathscr{F}^p_\star(\tilde{m}_0)}^{1/2}
+ \| f_0 \|_{\spp(m)}^{1/2}  \| h \|_{\mathscr{F}^p_\star(\tilde{m}_0)}
+ \| f_0 \|_{\spp(m)}^{1/2} \|  g \|_{\mathscr{G}^p}^{1/2} \| h \|_{\mathscr{F}^p_\star(\tilde{m}_0)}^{1/2},
\end{aligned}
$$
therefore taking the $L^p_\xi$ norm in the above estimate implies
\begin{align*}
\| h \|_{\mathscr{F}^p_\star(\tilde m_0)} 
&\lesssim   
\| f_0 \|_{\spp(m)}^{1/2} \| h \|_{\mathscr{F}^p(\tilde{m})}^{1/2} \| h \|_{\mathscr{F}^p_\star(\tilde{m}_0)}^{1/2}
+ \| f_0 \|_{\spp(m)}^{1/2}\| h \|_{\mathscr{F}^p_\star(\tilde{m}_0)}
+ \| f_0 \|_{\spp(m)}^{1/2} \|  g \|_{\mathscr{G}^p}^{1/2} \| h \|_{\mathscr{F}^p_\star(\tilde{m}_0)}^{1/2} \\
&\quad
+ \| h \|_{\mathscr{F}^p(\tilde m)} 
+ \|\widehat f_0 (\xi) \|_{L^p_\xi L^2_v(\tilde m_0)}.
\end{align*}
Recalling that $\| f_0 \|_{\spp(m)} + \| \widehat f_0 \|_{L^p_\xi L^2_v (\la v \ra^{-8s} m)} \le \eps_0$ and taking $\eps_0>0$ small enough yields
\begin{equation}\label{eq:h_Fp_star}
\begin{aligned}
\| h \|_{\mathscr{F}^p_\star(\tilde m_0)} 
&\lesssim   
\| f_0 \|_{\spp(m)} \| h \|_{\mathscr{F}^p(\tilde{m})}
+ \| f_0 \|_{\spp(m)} \|  g \|_{\mathscr{G}^p}  \\
&\quad
+ \| h \|_{\mathscr{F}^p(\tilde m)} 
+ \|\widehat f_0 (\xi) \|_{L^p_\xi L^2_v(\tilde m_0)}.
\end{aligned}
\end{equation}

\medskip\noindent
\textit{Step 3: Estimate of $g$ in norm $\mathscr{G}^p$.}
We now turn to the second equation of \eqref{eq:systeme-hgN}. The weak coercivity estimate from Lemma~\ref{lem:dissipative_equivalent_gauss} gives us, for some $\lambda > 0$,
\begin{equation*}
\Re \lla \! \lla (\LLL-iv \cdot \xi) \widehat g(\xi), \widehat g(\xi) \rra \! \rra_{L^2_v( \mu^{-1/2})} \le - \lambda \| \widehat g(\xi) \|_{H^{s,**}_v(\mu^{-1/2})}^2,
\end{equation*}
thus
\begin{equation}\label{eq:dtgN+1_xi}
\begin{aligned}
&\frac{1}{2} \frac{\mathrm d}{\mathrm d t} \Nt \widehat g (\xi) \Nt^2_{L^2_v( \mu^{-1/2})} +\lambda \| \widehat g(\xi) \|_{H^{s,**}_v( \mu^{-1/2})}^2 \\
&\qquad\lesssim  
\lla \!\! \lla \widehat Q (  g ,  g) (\xi), \widehat g (\xi) \rra \!\! \rra_{L^2_v(\mu^{-1/2})}
+ \| \widehat h (\xi) \|_{L^2_v(\tilde m_0)} \| \widehat g (\xi) \|_{L^2_v(\mu^{-1/2})},
\end{aligned}
\end{equation}
where we have used that $\Nt \cdot \Nt_{L^2_v(\mu^{-1/2})}$ and  $\| \cdot \|_{L^2_v(\mu^{-1/2})}$ are equivalent and that the operator $\AA$ has compact support in velocity.

Integrating in time then taking the supremum in time in both sides gives
\begin{align*}
&\|\widehat g (\xi) \|_{L^\infty_t L^2_v(\mu^{-1/2})}^2 + \| \widehat g(\xi) \|_{L^2_t H^{s,**}_v (\mu^{-1/2})}^2 \\
&\qquad\lesssim  
\widetilde \NN_{0} [g , g , g] (\xi)^2 
+ \| (1+t)^\theta \widehat h (\xi) \|_{L^\infty_t L^2_v(\tilde m_0)} \| \widehat g (\xi) \|_{L^\infty_t L^2_v(\la v \ra^q \mu^{-1/2})},
\end{align*}
where we have used the fact that $\theta > 1$ so that $(1+t)^{-\theta}$ is integrable in the last term. Hence we deduce
\begin{equation}\label{eq:gN+1xi_E1}
\begin{aligned}
&\|\widehat g (\xi) \|_{L^\infty_t L^2_v(\mu^{-1/2})} + \| \widehat g(\xi) \|_{L^2_t H^{s,**}_v (\mu^{-1/2})} \\
&\qquad\lesssim  
\widetilde \NN_{0} [g , g , g] (\xi)
+ \| (1+t)^\theta \widehat h (\xi) \|_{L^\infty_t L^2_v(\tilde m_0)}.
\end{aligned}
\end{equation}

Observe that, since $p \in (2,\infty]$, we have by Proposition~\ref{prop:tildeNq_fgh}
$$
\begin{aligned}
\| \widetilde \NN_{q} [g , g , g] \|_{L^p_\xi} 
\lesssim  \| g \|_{\mathscr{G}^p}^{3/2} + \| g \|_{\mathscr{G}^p} \| \hat g \|_{L^1_\xi L^2_t H^{s,**}_v (\mu^{-1/2})}^{1/2} ,
\end{aligned}
$$
so that arguing as in previous steps, using the bounds on $g$ from estimate \eqref{eq:bound_h_g} together with \eqref{eq:g_L1xi_E}, we deduce 
$$
\| \widetilde \NN_{0} [g , g , g] \|_{L^p_\xi} 
\lesssim  \| g \|_{\mathscr{G}^p}^{3/2} + \| f_0 \|_{\spp(m)}^{1/2} \| g \|_{\mathscr{G}^p} . 
$$ 
Therefore taking the $L^p_\xi$ norm of \eqref{eq:gN+1xi_E1} implies
\begin{align*}
\| g \|_{\mathscr{G}^p} 
&\lesssim  
\| g \|_{\mathscr{G}^p}^{3/2} + \| f_0 \|_{\spp(m)}^{1/2} \| g \|_{\mathscr{G}^p}
+ \| h \|_{\mathscr{F}^p_\star(\tilde m_0)} ,
\end{align*}
thus taking $\eps_0 >0$ small enough, recalling that $\| f_0 \|_{\spp(m)} + \| \widehat f_0 \|_{L^p_\xi L^2_v (\la v \ra^{-8s} m)} \le \eps_0$,  gives
\begin{equation}\label{eq:g_Gp}
\| g \|_{\mathscr{G}^p} 
\lesssim  
\| g \|_{\mathscr{G}^p}^{3/2} + \| h \|_{\mathscr{F}^p_\star(\tilde m_0)}. 
\end{equation}

Therefore gathering \eqref{eq:h_Fp}--\eqref{eq:h_Fp_star}--\eqref{eq:g_Gp} and taking $\eps_0 >0$ small enough we obtain
\begin{equation}\label{eq:h_Fp_g_Gp}
\| h \|_{\mathscr{F}^p(\tilde m)} + \| h \|_{\mathscr{F}^p_\star (\tilde m_0)} + \| g \|_{\mathscr{G}^p}
\lesssim \| f_0 \|_{\spp(m)} + \| \widehat f_0 \|_{L^p_\xi L^2_v(\tilde{m})} \le \eps_0.
\end{equation}

\medskip\noindent
\textit{Step 4: Estimate of $g$ in $\spgt_{0,\star}$.}
We finally turn to the estimation of $g$ in $\spgt_{0,\star}$, which is the most delicate one. Arguing similarly as for obtaining~\eqref{eq:g_E}, we get for some $\lambda>0$
$$
\frac{1}{2} \frac{\mathrm d}{\mathrm d t} \Nt g \Nt_{\spg}^2 + \lambda \| g \|_{\disg}^2 \lesssim  \Nt g \Nt_{\spg} \| g \|_{\disg}^2
	 +  \| h \|_{\spp( \la v \ra^{\gamma/2} m_0)} \Nt g \Nt_{\spg}.
$$
where we observe that, in comparison with \eqref{eq:g_E}, we have a weight $\la v \ra^{\gamma/2} m_0$ instead of $m_0$ for $h$ on the second term in the right-hand side, which is possible since $\AA$ has compact support in velocity. Thanks to estimate \eqref{eq:bound_h_g} and the fact that $\eps_0>0$ is small enough, we deduce
\begin{equation}\label{eq:g_E_bis}
\frac{1}{2} \frac{\mathrm d}{\mathrm d t} \Nt g \Nt_{\spg}^2 + \frac{\lambda}{2} \| g \|_{\disg}^2 \lesssim  \| h \|_{\spp( \la v \ra^{\gamma/2} m_0)} \Nt g \Nt_{\spg}.
\end{equation}

We shall now bound by below the dissipation term $\| g \|_{\disg}$, by observing that
$$
\frac{1}{c_0} \Nt g \Nt_{\spg}^2 \le \int_{\R^3}  (1+ |\xi|^{6}) \| \widehat g (\xi) \|_{L^{2}_v(\mu^{-1/2})}^2 \, \d \xi \le c_0 \Nt g \Nt_{\spg}^2
$$ 
and
$$
\frac{1}{c_1}\| g \|_{\disg}^2 \le \int_{\R^3}  (1+ |\xi|^{6}) \| \widehat g (\xi) \|_{H^{s,**}_v(\mu^{-1/2})}^2 \, \d \xi \le c_1 \| g \|_{\disg}^2
$$
for some constants $c_0,c_1>0$. We therefore shall focus below on the quantity $\| \widehat g (\xi) \|_{H^{s,**}_v(\mu^{-1/2})}$, that we recall is given by
\begin{equation}\label{eq:hatg_xi_dissipation}
\| \widehat g (\xi) \|_{H^{s,**}_v(\mu^{-1/2})}^2 = \| \widehat g^\perp (\xi) \|_{H^{s,*}_v(\mu^{-1/2})}^2 + \frac{|\xi|^2}{1+|\xi|^2} \| \pi\widehat g (\xi) \|_{L^2_v(\mu^{-1/2})}^2.
\end{equation}

First of all, for the first term in the right-hand side of \eqref{eq:hatg_xi_dissipation} we can argue similarly as in Step 2 above in order to obtain the interpolation inequality \eqref{eq:interpolation_Rt_bis}, therefore we obtain, using that $\|\la v \ra^{\gamma / 2} \cdot \|_{L^2_v(\mu^{-1/2})} \le \| \cdot \|_{H^{s,*}_v(\mu^{-1/2})}$, that for all $\xi \in \R^3$ and $t \ge 0$ there holds
\begin{equation}\label{eq:interpolation1}
\begin{aligned}
&\frac{4\vartheta c_0 c_1}{\lambda} (1+t)^{-1} \Nt \widehat g^\perp (\xi)  \Nt_{L^2_v(\mu^{-1/2})}^2 \\
&\qquad 
\le \| \widehat g^\perp (\xi)  \|_{H^{s,*}_v(\mu^{-1/2})}^2 + C (1+t)^{ -1 - \frac{2q}{|\gamma|} } \| \widehat g^\perp (\xi)  \|_{L^2_v(\la v \ra^q \mu^{-1/2})}^2,
\end{aligned}
\end{equation}
for some constant $C>0$ (depending only on $\lambda,\vartheta,q,\gamma$).

We now investigate the  second term in the right-hand side of \eqref{eq:hatg_xi_dissipation} by splitting the analysis into two parts : low frequencies $|\xi| \le 1$ and high frequencies $|\xi| \ge 1$. 
For high frequencies $|\xi| \ge 1$ we observe that $ \frac{|\xi|^2}{1+|\xi|^2} \ge \frac{1}{2}$, which implies
$$
\frac12 \mathbf{1}_{|\xi|\ge 1}  \| \widehat g(\xi) \|_{H^{s,*}_v( \mu^{-1/2})}^2 \le
\mathbf{1}_{|\xi|\ge 1} \| \widehat g(\xi) \|_{H^{s,**}_v( \mu^{-1/2})}^2 \le \mathbf{1}_{|\xi|\ge 1} \| \widehat g(\xi) \|_{H^{s,*}_v( \mu^{-1/2})}^2.
$$
Therefore arguing as in \eqref{eq:interpolation1} we obtain 
\begin{equation}\label{eq:interpolation2}
\begin{aligned}
&\frac{4\vartheta c_0 c_1}{\lambda} (1+t)^{-1} \mathbf{1}_{|\xi|\ge 1} \Nt \widehat g(\xi)  \Nt_{L^2_v(\mu^{-1/2})}^2 \\
&\qquad
\le \mathbf{1}_{|\xi|\ge 1} \| \widehat g_{N+1} (\xi)  \|_{H^{s,**}_v(\mu^{-1/2})}^2 + C (1+t)^{ -1 - \frac{2q}{|\gamma|} } \mathbf{1}_{|\xi|\ge 1} \| \widehat g_{N+1} (\xi)  \|_{L^2_v(\la v \ra^q \mu^{-1/2})}^2.
\end{aligned}
\end{equation}
We now turn to the the case of low frequencies $|\xi| \le 1$. 
Recalling that $2 < p \le \infty$, we fix a parameter $r>1$ such that
\begin{equation}\label{eq:condition_r}
1+ \frac{2p}{3(p-2)} < r < 1+ \frac{1}{2\vartheta} ,
\end{equation}
which is possible because $0<\vartheta < \frac{3}{2} (\frac{1}{2} - \frac{1}{p})$. Remarking that, since $|\xi| \le 1$, there holds $ |\xi|^2 \le \frac{2|\xi|^2}{1+|\xi|^2}$, we obtain by Young's inequality that for any $\eta>0$ there is $C_\eta>0$ such that
$$
\forall |\xi| \le 1, ~ \forall t \ge 0, \quad 1 \le \eta (1+t)\frac{|\xi|^2}{1+|\xi|^2} + C_\eta (1+t)^{- \frac{1}{r-1}} |\xi|^{- \frac{2}{r-1}}.
$$
We thus deduce
\begin{equation}\label{eq:interpolation3}
\begin{aligned}
&\eta^{-1} (1+t)^{-1} \mathbf{1}_{|\xi|\le 1} \| \pi \widehat{ g} (\xi) \|_{L^2_v(\mu^{-1/2})}^2 \\
&\qquad 
\le   \mathbf{1}_{|\xi|\le 1} \frac{|\xi|^2}{1+|\xi|^2} \| \pi \widehat{ g} (\xi) \|_{L^2_v(\mu^{-1/2})}^2 + C'_\eta (1+t)^{-1- \frac{1}{r-1}}  \mathbf{1}_{|\xi|\le 1} |\xi|^{- \frac{2}{r-1}}
\| \pi \widehat{ g} (\xi) \|_{L^2_v(\mu^{-1/2})}^2.
\end{aligned}
\end{equation}

Gathering \eqref{eq:interpolation1}--\eqref{eq:interpolation2}--\eqref{eq:interpolation3} and choosing $\eta>0$ appropriately yields
\begin{align*}
\frac{\lambda}{4 c_1} \| \widehat  g(\xi) \|_{H^{s, **}_v(\mu^{-1/2})}^2 
&\ge c_0 \vartheta (1+t)^{-1} \Nt \widehat g(\xi) \Nt_{L^2_v(\mu^{-1/2})}^2
- C (1+t)^{-1-\frac{2q}{|\gamma|}} \| \widehat g(\xi) \|_{L^2_v(\la v \ra^q \mu^{-1/2})}^2 \\
&\quad
-C (1+t)^{-1- \frac{1}{r-1}}  \mathbf{1}_{|\xi|\le 1} |\xi|^{- \frac{2}{r-1}}
\| \pi \widehat g (\xi) \|_{L^2_v(\mu^{-1/2})}^2,
\end{align*}
which implies
\begin{align*}
\frac{\lambda}{4} \| g \|_{\disg}^2 
&\ge \vartheta (1+t)^{-1} \Nt  g \Nt_{\spg}^2
- C (1+t)^{-1-\frac{2q}{|\gamma|}} \|  g \|_{\spg_q}^2 \\
&\quad
-C (1+t)^{-1- \frac{1}{r-1}}  \int_{\R^3} \mathbf{1}_{|\xi|\le 1} |\xi|^{- \frac{2}{r-1}} \la \xi \ra^{6}
\| \pi  \widehat g(\xi) \|_{L^2_v(\mu^{-1/2})}^2 \, \d \xi.
\end{align*}
Applying H\"older's inequality for the last term we get
$$
\int_{\R^3} \mathbf{1}_{|\xi|\le 1} |\xi|^{- \frac{2}{r-1}} \la \xi \ra^{6}
\| \pi   \widehat g(\xi) \|_{L^2_v(\mu^{-1/2})}^2 \, \d \xi \lesssim \| \pi  \widehat g\|_{L^p_\xi L^\infty_t L^2_v (\mu^{-1/2})}^2
$$
since $r>1+\frac{2p}{3(p-2)}$. Thus we obtain
\begin{align*}
\frac{\lambda}{4} \| g \|_{\disg}^2 
&\ge \vartheta (1+t)^{-1} \Nt  g \Nt_{\spg}^2
- C (1+t)^{-1-\frac{2q}{|\gamma|}} \|  g \|_{\spg_q}^2 
-C (1+t)^{-1- \frac{1}{r-1}}  \| g \|_{\mathscr{G}^p}^2,
\end{align*}
which gives, coming back to \eqref{eq:g_E_bis},
$$
\begin{aligned}
&\frac{1}{2} \frac{\mathrm d}{\mathrm d t} \Nt g \Nt^2_{\spg} 
+\frac{\lambda}{4} \|  g \|_{\disg}^2 
+ \vartheta (1+t)^{-1} \Nt g \Nt_{\spg}^2 \\
&\qquad\lesssim  
\| h \|_{\spp( \la v \ra^{\gamma/2} m_0)} \Nt g \Nt_{\spg}
+ (1+t)^{-1-\frac{2q}{|\gamma|}} \|  g \|_{\spg_q}^2
+ (1+t)^{-1- \frac{1}{r-1}} \| g \|_{\mathscr{G}^p}^2.
\end{aligned}
$$
Multiplying by $(1+t)^{2 \vartheta}$ we thus get
$$
\begin{aligned}
&\frac{1}{2} \frac{\mathrm d}{\mathrm d t} \left\{ (1+t)^{2 \vartheta} \Nt g \Nt^2_{\spg} \right\}
+\frac{\lambda}{4} (1+t)^{2 \vartheta}\|  g \|_{\disg}^2  \\
&\qquad\lesssim  
(1+t)^{2 \vartheta}\| h \|_{\spp( \la v \ra^{\gamma/2} m_0)} \Nt g \Nt_{\spg}
+ (1+t)^{2 \vartheta-1-\frac{2q}{|\gamma|}} \|  g \|_{\spg_q}^2 
+ (1+t)^{2 \vartheta-1- \frac{1}{r-1}}  \| g \|_{\mathscr{G}^p}^2.
\end{aligned}
$$
Integrating in time last estimate implies (recall the definition of $\spgt_{0,\star}$ from \eqref{eq:E_star})
$$
\begin{aligned}
\| g \|_{\spgt_{0,\star}}^2 
\lesssim \| h \|_{\sppdt(m_0)} \| g \|_{\spgt_{0,\star}} + \| g \|_{\spgt_q}^2 + \| g \|_{\mathscr{G}^p}^2,
\end{aligned}
$$
where for the first term in the right-hand side we have used that $\theta> \vartheta + 1/2$ so that $(1+t)^{2(\vartheta-\theta)}$ is integrable and that $\| \cdot \|_{\spp( \la v \ra^{\gamma/2} m_0))} \le \| \cdot \|_{\dispa( m_0)}$; for the second one that $q > |\gamma| \vartheta$ so that $(1+t)^{2 \vartheta-1 -  \frac{2q}{|\gamma|}}$ is integrable; and for the third one that $r<1+1/(2\vartheta)$ so that $(1+t)^{2\vartheta -1 - \frac{1}{r-1}}$ is integrable.

Using the bounds \eqref{eq:bound_h_g} and \eqref{eq:h_Fp_g_Gp} yields
\begin{equation}\label{eq:g_dissipation}
\begin{aligned}
\| g \|_{\spgt_{0,\star}}
\lesssim \| f_0 \|_{\spp(m)} + \| \widehat f_0 \|_{L^p_\xi L^2_v(\tilde m)} \le \eps_0,
\end{aligned}
\end{equation}
which together with the second estimate of \eqref{eq:bound_h_g} concludes the proof of~\eqref{eq:theo:decay_estimate}.

\bigskip
\bibliographystyle{acm}
\bibliography{biblio}

\end{document}